\documentclass[12pt]{amsart}
\usepackage[top=1in, bottom=1.25in, left=1in, right=1in]{geometry}
\usepackage{amssymb,amsmath,amsthm}

\usepackage{url}
\usepackage{cite}
\usepackage{amsfonts}
\usepackage{verbatim}
\usepackage{graphicx}
\usepackage[utf8]{inputenc}
\usepackage{float}
\usepackage[english]{babel}
 \usepackage{xcolor}
\usepackage{caption} 
\usepackage{subcaption}

\makeatletter
\@namedef{subjclassname@2010}{%
  \textup{2010} Mathematics Subject Classification}
\makeatother

\newtheorem{thm}{Theorem}[section]

\newtheorem{prop}[thm]{Proposition}
\newtheorem{cor}[thm]{Corollary}
\newtheorem{lem}[thm]{Lemma}
\newtheorem{remark}[thm]{Remark}

\def\N{{\mathbb{N}}}
\def\C{{\mathbb{C}}}
\def\Z{{\mathbb{Z}}}
\def\E{{\mathcal{E}}}

\def\R{{\mathbb{R}}}

\def\W{{\mathcal{W}}}
\def\H{{\mathcal{H}}}
\def\a{{\alpha}}

\begin{document}

\title{Gaps of saddle connection directions for some branched covers of tori}

\author[A. Sanchez]{Anthony Sanchez}
\address{Department of Mathematics, University of Washington, Box 354350, Seattle, WA, 98195-4530}
\email{asanch33@uw.edu}


\begin{abstract}{We compute the gap distribution of directions of saddle connections for two classes of translation surfaces. One class will be the  translation surfaces arising from gluing two identical tori along a slit. These yield the first explicit computations of gap distributions for non-lattice translation surfaces.  We show that this distribution has support at 0 and quadratic tail decay. We also construct examples of translation surfaces in any genus $d>1$ that have the same gap distribution as the gap distribution of two identical tori glued along a slit. The second class we consider are twice-marked tori and saddle connections between distinct marked points. These results can be interpreted as the gap distribution of slopes of affine lattices.
We obtain our results by translating the question of gap distributions to a dynamical question of return times to a transversal under the horocycle flow on an appropriate moduli space.}
\end{abstract}
\maketitle
%%%%%%%%%%%%%%%%%%%%%%%%%%%%%%%%%%%%%%%%%%%%%%%%%%%%%%%%%%%%%%%

\section{Introduction}
We are interested in the distribution of directions of saddle connections on translation surfaces. A \emph{translation surface} is a pair $(X,\omega)$ where $X$ is a compact, connected Riemann surface of genus $g$ and $\omega$ a non-zero holomorphic 1-form on $X$. The singular flat metric induced by $\omega$ has isolated singularities at the zeros of $\omega$, with a zero of order $k$ corresponding to a point with cone angle $2\pi(k+1)$.

A \emph{saddle connection} is a geodesic segment in the flat metric connecting two (not necessarily distinct) singular points, with no singular points in its interior. See Figure \ref{fig: DoubledSlitTorus} for an example of a saddle connection on a translation surface. Integrating $\omega$ over $\gamma$ yields the \emph{holonomy vector} of $\gamma$, which keeps track of how far and in what direction $\gamma$ travels. We write
$$v_\gamma := \int_\gamma \omega \in \C.$$  The set of all holonomy vectors,
$$\Lambda_\omega = \{v_\gamma: \gamma \text{ is an oriented saddle connection on } (X,\omega)\}\subset \C$$
is a discrete subset of the plane, see e.g. \cite{MR2186246}.

 The counting and distribution of saddle connections and their holonomy vectors is a very well-studied problem and we review the relevant history in the following subsection. Our motivating question is understanding how random the set of saddle connection directions are for some classes of translation surfaces. We study this by considering the \emph{gap distribution} of slopes of holonomy vectors which we now describe.

\begin{figure}
	\centering
    \includegraphics[width=4in]{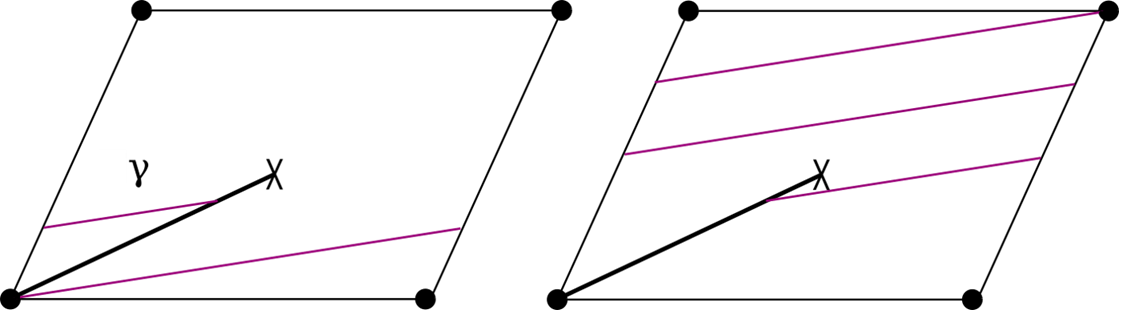}
    \caption{A saddle connection $\gamma$ on a translation surface. This translation surface is an example of a doubled slit torus.} 
    \label{fig: DoubledSlitTorus}
\end{figure}

The \emph{slopes} of saddle connections of a translation surface $(X,\omega)$ in the first quadrant are given  by
$$\mathcal S(\Lambda_{\omega}) := \left\{\text{slope}(w) = \frac{w_2}{w_1}\bigg| w =w_1+i w_2\in \Lambda_{\omega},w_1>0,w_2\ge0\right\}.$$

Let $\mathcal S^R(\Lambda_{\omega}) $ be the set of slopes of saddle connections in $\omega$ that are in the first quadrant and in the ball of radius $R$ about 0 with respect to the max norm on $\R^2$,
$$
\mathcal S^R(\Lambda_{\omega})  = \left\{\text{slope}(w) = \frac{w_2}{w_1}\bigg| w \in \Lambda_{\omega}, w_1\in (0,R], w_2\in[0,R] \right\}.
$$
Let $N(\omega, R):= |\mathcal S^R(\Lambda_{\omega})|$ and write the elements of $\mathcal S^R(\Lambda_\omega)$  in increasing order:
$$
\mathcal S^R(\Lambda_{\omega}) = \left\{s_0 = 0<s_1<\ldots <s_{N(\omega, R)} \right\}.
$$
By the seminal work of  Eskin-Masur~\cite{MR1827113}, for almost every translation surface, there are positive constants $ c_1 = c_1(\omega) \le c_2 = c_2(\omega)$ so that $$c_1 R^2 \le N(\omega, R) \le c_2 R^2.$$ 

We are interested in slope gaps and since  $N(\omega, R)$ grows quadratically, then it is natural to consider the set of \emph{normalized slope gaps},
\begin{equation}\label{NormalizedGapsSlit}
 \mathcal G^R(\Lambda_{\omega}) = \left\{R^2\cdot(s_i-s_{i-1})|i = 1,\ldots, N(\omega, R)\right\}
\end{equation}

\noindent The \emph{gap distribution} of saddle connection directions on a translation surface $(X,\omega)$ is given by the weak-* limit as $R \rightarrow \infty$ of the probability measures
\begin{equation}\label{GapDist}
\eta_R(\omega) = \frac{1}{N(\omega, R)} \sum_{x \in  \mathcal G^R(\Lambda_{\omega})} \delta_x
\end{equation}
 if it exists.

Evaluating the weak-* limit of these measures on the indicator function of an interval $I$ we obtain
\begin{equation}\label{GapDist}
\lim_{R\to\infty}\frac{| \mathcal G^R(\Lambda_{\omega}) \cap I|}{N(\omega, R)}
\end{equation}
 the proportion of gaps in an interval $I$.
%%%%

\subsection{Motivations for gap distributions}
The counting and distribution properties of saddle connections and their directions have been studied extensively.  Masur~\cite{MR850537} showed the \emph{angles of saddle connections} are dense on the circle for any translation surface. Vorobets~\cite{MR1436653} and Eskin-Masur\cite{MR1827113} showed that the angles equidistribute on the circle for a generic translation surface. More recently, Dozier~\cite[Theorem 1.3]{MR3959357} showed that for \emph{every} translation surface the distribution of angles converge to a measure that is in the same measure class as Lebesgue measure on the circle. Precise descriptions of their results are in the references above, but they suggest that angles (and thus slopes) seem to behave ``randomly" at first glance. Thus, the gap distribution of translation surfaces can be thought of as a second test of randomness that yields insight to the fine-scale statistics of holonomy vectors.

Probability theory provides some context for what one would expect from ``truly random" gaps. In the case of a sequence of independent identically distributed uniform random variables on $[0, 1]$ the gaps converge to an exponential distribution. Following \cite{MR3497256}, we call gap distributions that deviate from an exponential distribution \emph{exotic}.

\subsection{Main results}
We compute the gap distribution for the class of translation surfaces that arise from gluing two identical tori along a slit.  See Figure \ref{fig: DoubledSlitTorus}. We call such a surface a \emph{doubled slit torus}. Every doubled slit torus has genus two and two cone-type singularities of angle $4\pi$. Denote the class of doubled slit tori by $\E.$ We prove that the gap distribution of doubled slit tori exists and is given by an explicit density function. This provides the first explicit computation of a gap distribution for a non-lattice surface. 

\begin{thm}\label{ThmGapSlit}
There is a limiting probability density function $ f : [0,\infty) \to [0,\infty)$, with
$$\lim_{R\to\infty}\frac{| \mathcal G^R (\Lambda_\omega) \cap I|}{N(\omega, R)}=\int_I f(x)\,dx$$
for almost every doubled slit torus $\omega$ with respect to an  $SL_2(\R)$-invariant probability measure on $\E$ that is in the same measure class as Lebesgue measure.

The density function is continuous, piecewise differentiable with 3 points of non-differentiability, and each piece is expressed in terms of elementary functions.
\end{thm}

\begin{figure}[H]
\centering
\begin{subfigure}{.5\textwidth}
  \centering
  \includegraphics[width=.6\linewidth]{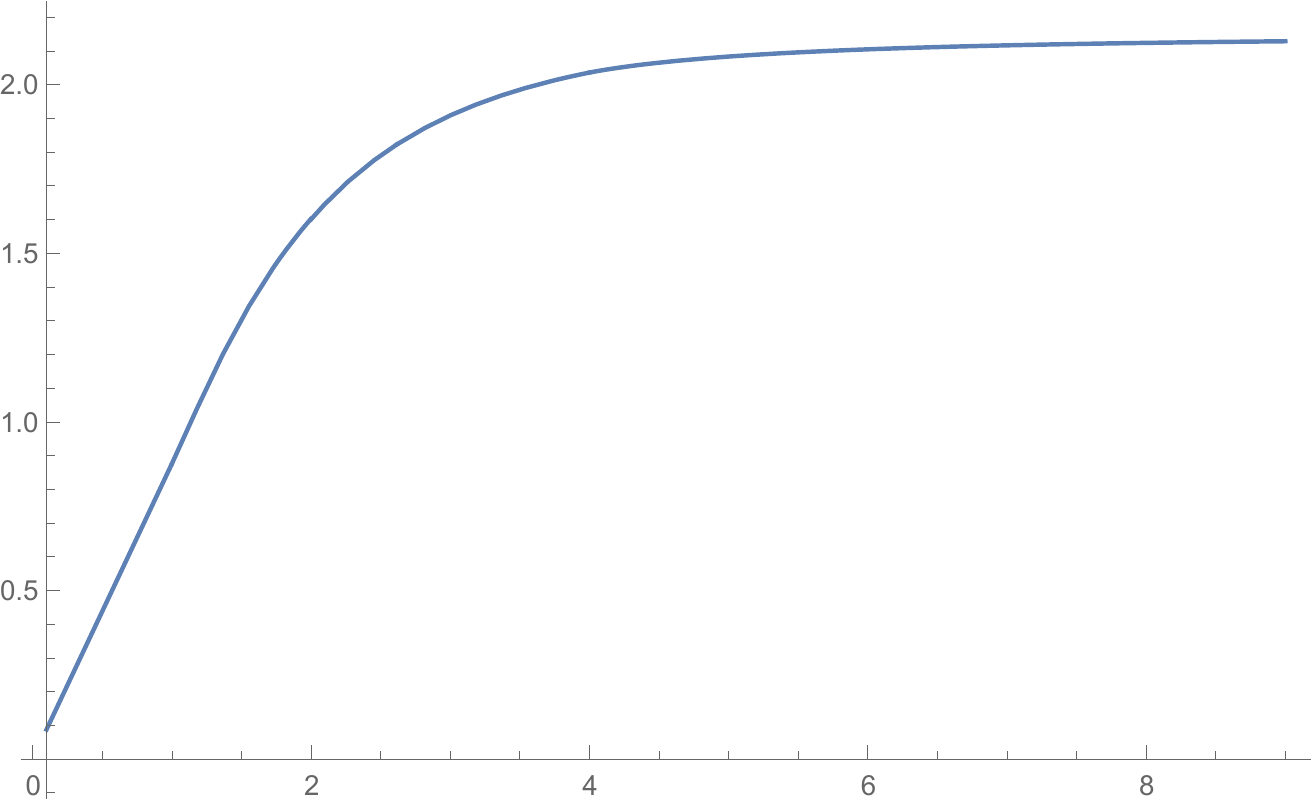}
  \caption{The cummulative distribution function.}
  \label{CummDist}
\end{subfigure}%
\begin{subfigure}{.5\textwidth}
  \centering
  \includegraphics[width=.6\linewidth]{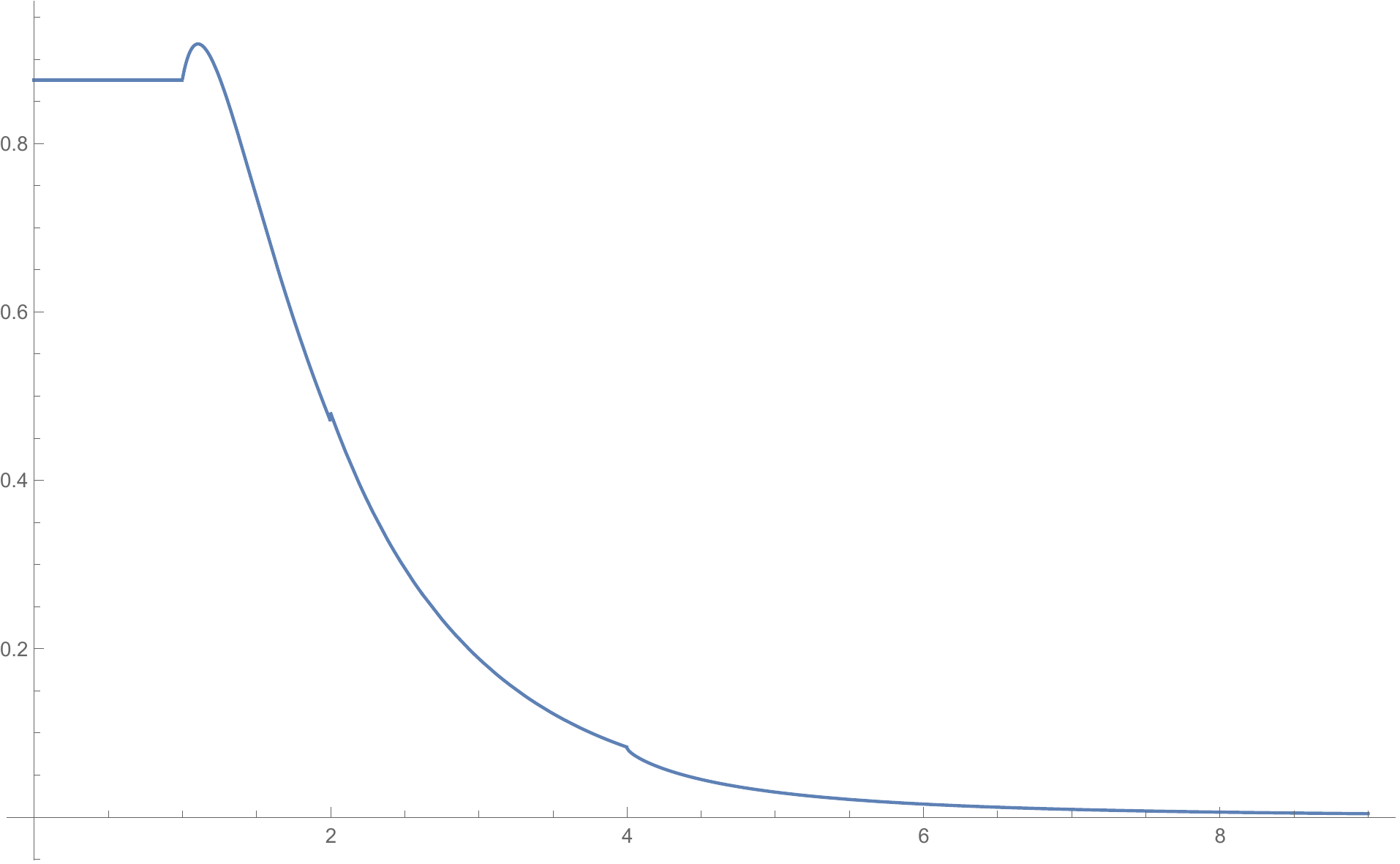}
  \caption{The density function}
  \label{Density}
\end{subfigure}
\caption{Proportion of gaps for doubled slit tori}
\label{GapFunctions}
\end{figure}

The proof of Theorem \ref{ThmGapSlit} can be found in section 4 and an explicit description of the cumulative distribution function can be found in the appendix A.1.3. As a corollary we also show that the gap distribution has quadratic tail so that the gap distribution of doubled slit tori is exotic (Theorem \ref{Estimates} part 4).  We prove other results related to the gap distribution, but we wait until section 4 Theorem \ref{Estimates} before more precisely stating them.  %See Figure \ref{GapFunctions} (A) for the graph of the cumulative distribution function and Figure \ref{GapFunctions} (B) for the graph of the density function. 

Perhaps surprisingly, Theorem \ref{ThmGapSlit} will follow from the following dynamical result (relevant definitions will be given in Section 2). 

\begin{thm}\label{ThmWTransversal}
The set $\W$ consisting of doubled slit tori with a horizontal saddle connection of length less than or equal to 1 forms a transversal for the space of doubled slit tori $\E$ under Teichm\"{u}ller horocycle flow. There is a four-dimensional parameterization of $\W$ (see Proposition \ref{ParameterizationW}) and the return time map in these coordinates is given explicitly (see Theorem \ref{ReturnTimeW}).
\end{thm}
In fact, the connection between slope gap questions and return times on moduli space is relatively well understood, see Athreya \cite{MR3497256} or Section 2.3 where we recall this connection. The main contribution of this theorem is the explicit return time map (Theorem \ref{ReturnTimeW}) that requires new techniques arising from the high dimensionality of the parameter space. Another key ingredient in proving Theorem \ref{ThmWTransversal} is the study of unipotent flows on the space of \emph{affine lattices}. By an affine lattice we simply mean a set of the form $g\Z^2+v$ where $g\in SL_2(\R)$ and $v\in \R^2$.  We prove the following result for the space of affine lattices.

\begin{thm}\label{ThmOmegaTransversal}
The set $\Omega$ consisting of affine lattices with a horizontal vector of length less than or equal to 1 forms a transversal for the space of affine lattices  $AX_2$ under horocycle flow. There is a four-dimensional parameterization of $\Omega$ (see Proposition \ref{ParameterizationOmega}) and the return time map in these coordinates is given explicitly (see Theorem \ref{ReturnTimeMap}). 
\end{thm}

In addition to using the above theorem to prove Theorem \ref{ThmWTransversal} we also use it to prove a slope gap question for affine lattices that we phrase in Section 2.5.

\subsection{Why doubled slit tori are important}
The space $\E$ of doubled slit tori has a long history of being studied to provide insight into general questions on the dynamics and geometry of translation surfaces. That the space and the surfaces within it have very concrete descriptions make it an excellent testing ground for conjectures. 
 
For example, the space $\E$ provided the first examples of higher genus surfaces with a minimal but non-uniquely ergodic directional flow by Masur (see~\cite{MR1928530}). Cheung~\cite{MR2018932} used these surfaces to show that an upper bound on Hausdorff dimension of the directions which are non-ergodic was sharp.  A dichotomy on the Hausdorff dimension of such directions for doubled slit tori was later given by Cheung-Hubert-Masur~\cite{MR2772084}.

The space $\E$ is also interesting from the point of view of the $SL_2(\R)$ action on translation surfaces. In the remarkable paper \cite{MR2299738}, McMullen completely classifies orbit closures (under $SL_2(\R)$) of genus 2 translation surfaces. One consequence of his work is that there are translation surfaces that have non closed $SL_2(\R)$ orbits but are not dense. These intermediate orbit closures consist of an infinite family called the eigenform loci. The space $\E$ is a part of this family and often called the \emph{Eigenform locus of discriminant four} in the literature. 

\subsection{Some history of gap distributions}
This paper makes use of  unipotent flows on the space of affine lattices to prove a gap distribution result for doubled slit tori. Elkies-McMullen~\cite{MR2060024} famously used unipotent flows on the space of affine lattices to compute the distribution of gaps of the sequence $(\sqrt{n})_{n\in \N}$  on the unit circle.  Similarly, Marklof-Str\"{o}mbergsson~\cite{MR2726104} used dynamics on the space of affine lattices to study the gap distribution for the angles of visible affine lattice points. 

Through personal communication with J. Athreya, we learned that the paper of  Marklof-Str\"{o}mbergsson provided much inspiration for the use of homogenous dynamics in the context of gap distributions of translation surfaces. Athreya-Chaika~\cite{MR3000496} showed that the gap distribution for generic surfaces exists and have support at zero. Gap distributions of lattice surfaces were considered shortly thereafter. Athreya-Chaika-Lelievre~\cite{MR3330337} computed the explicit gap distribution for a genus $2$ lattice surface (the \emph{Golden L}) and Uyanik-Work~\cite{MR3567252} for the genus $2$ surface arising from the regular octagon. In some sense, these papers were generalizations of Athreya-Cheung~\cite{MR3214280} which used homogenous dynamics to compute gap distribution of Farey fractions (which can be interpreted as slopes of saddle connections of the torus with a marked point). In all the cases above, the parameter space is three dimensional whereas the parameter space for doubled slit tori and affine lattice is five dimensional.

 The present paper and \cite{MR3330337,MR3567252,MR3214280} are examples of a general strategy developed in Athreya \cite{MR3497256}. This strategy involves turning the question of gap distributions of slopes of saddle connections to a dynamical question of return times to a transversal under horocycle flow on an appropriate moduli space.  We recall this strategy in detail in the next section. Examples of explicit cross sections for horocycle flow can also be found in \cite{dia} and \cite{gh2}.

\subsection{Organization of the paper} In section 2 we outline the general strategy from \cite{MR3497256} which reduces the gap distribution question to a dynamical question. We also introduce twice marked tori and construct transversals needed for our paper. In section 3 we give a parameterization of the transversals for twice marked tori (Proposition \ref{ParameterizationOmega}) and doubled slit tori (Proposition \ref{ParameterizationW}). We also compute the first return time under horocycle flow for each  (Theorem \ref{ReturnTimeMap} and Theorem \ref{ReturnTimeW}).  In section 4 we consider measures on our transversals, classify measures for the transversal on twice marked tori (Theorem \ref{MeasureClassificationOmega}), and give precise descriptions of our results (Theorem \ref{Estimates}). In Corollary \ref{dSlitTorus} we construct examples of translation surfaces in any genus $d>1$ that have the same gap distribution as doubled slit tori. In the appendix we prove our estimates for the various gap distributions we consider.

\subsection{Acknowledgments}
This research was supported by the National Science Foundation Graduate Research Fellowship under Grant No.  1122374.  This paper is part of my thesis and I would like to give many thanks to my advisor Jayadev Athreya for his continuous encouragement and endless patience. Thanks to Dami Lee and Sunrose Shrestha for useful comments on the paper and to Aaron Calderon and Jane Wang for helpful discussions.

%%%%%%%%%%%%%%%%%%%%%%%%

%%%%%%%%%%%%%%%%%%%%%%%%%%%%%%%%%%%%%%%%%%%%%%%%%
\section{Setting the stage} In this section we will outline the general strategy from  \cite{MR3497256}, introduce the relevant spaces, and define our transversals.

 \subsection{On translation surfaces} We review the basics of translation surfaces needed for this article. For a detailed treatment of these topics, we refer the reader to\cite{MR1928530,MR2186247,MR2261104}.

Recall that a translation surface is a pair $(X,\omega)$ where $X$ is a compact, connected Riemann surface of genus $g$ and $\omega$ is a non-zero holomorphic 1-form on $X$. We can also, equivalently, view a translation surface as a union of finitely many polygons $P_1\cup P_2\cup\cdots\cup P_n$ in the Euclidean plane with gluings of parallel sides by translations such that for each edge there exists a parallel edge
of the same length and these pairs are glued together by a Euclidean translation. In this case, the total angle about each vertex of the polygons is necessarily an integer multiple of $2\pi$. The vertices where the total angle $\theta$ is greater than $2\pi$ are called \emph{cone points} and correspond to the zeros of the 1-form. In fact, a zero of order $k$ is a cone point of total angle $2\pi(k+1)$.

By the Riemann-Roch theorem the sum of order of the zeros is $2g-2$ where $g$ denotes the genus of $X$. Thus, the space of genus $g$ translation surfaces can be stratified by integer partitions of $2g-2$. If $\underline k = (k_1,\ldots,k_s)$ is an integer partition of $2g-2$, we denote by $\mathcal H(\underline k)$ the moduli space of translation surfaces $\omega$ such that the multiplicities of the zeroes are given by $k_1,\ldots,k_s.$ As an example, any doubled slit torus belongs in the stratum $\mathcal H(1,1)$. In the case that there are no zeros of the differential, but we want to mark points, it is common to use a vector $\underline k$ with all zeros. For example, the space of genus 1 translation surfaces have no cone points, but it is convenient to mark a single point and denote it as $\mathcal H(0)$. 

There is a natural action by $SL_2(\R)$ on the space of translation surfaces. This is most easily seen via the polygon definition: Given a translation surface $(X,\omega)$ that is a finite union of polygons $\{P_1,\ldots,P_n\}$ and $A\in SL_2(\R)$, we define $A\cdot (X,\omega)$ to be the translation surface obtained by the union of the polygons  $\{AP_1,\ldots,AP_n\}$ with the same side gluings as for $\omega$. This action makes sense since the linear action of $A$ preserves the notion of parallelism. Notice that the action of $SL_2(\R)$ preserves the multiplicities and number of zeros so that it induces an action on each strata  $\mathcal H(\underline k)$. Furthermore,  $SL_2(\R)$ acts equivariantly on the holonomy vectors of translation surfaces in that $\Lambda_{g\cdot \omega} = g\cdot \Lambda_{\omega}.$

We will consider the action of the one-parameter family of matrices
$$\bigg\{h_u:= \left(
  \begin{array}{ c c }
     1 & 0 \\
     -u & 1
  \end{array} \right)\in SL_2(\R):u\in \R\bigg\}.$$ The action of these matrices on $\mathcal H(\underline k)$ is called the \emph{(Teichm\"{u}ller) horocycle flow}.

%%%%%%%%%%%%%%%%%%%%%%%
\subsection{Reducing to gaps in a vertical strip}
We now phrase a parallel gap question. We will see that this question will turn out to be equivalent to our original one. Let $V$ denote the vertical strip $V = \{(x,y)\in \R^2: 0<x\le 1, y>0\}$. For a translation surface $(X,\omega)$, we can consider the set of slopes of saddle connections contained in the vertical strip $V$
$$\mathcal S_V (\Lambda_\omega) = \{0<s_1<\cdots<s_N<\cdots\}.$$
Let $\mathcal S_V ^N$ denote the first $N$ slopes. We can associate to these slopes the set of \emph{gaps}
\begin{equation}\label{UnnormalizedGapsSlit}
\mathcal G^N _V(\Lambda_\omega) = \left\{s_i-s_{i-1}|i = 1,\ldots, N \right\}
\end{equation} 

Notice that the slopes tend to infinity and so we do not need to normalize. Then the gap distribution for slopes in the vertical strip given by
\begin{equation}\label{SlopeGapSlit}
\lim_{N\to\infty}\frac{| \mathcal G^N _V(\Lambda_\omega) \cap I|}{N}
\end{equation}
for an interval $I$. We prove

\begin{thm}\label{ThmGapSlitV} 
There is a limiting probability density function $ f : [0,\infty) \to [0,\infty)$, with
$$\lim_{N\to\infty}\frac{| \mathcal G^N _V(\Lambda_\omega) \cap I|}{N}=\lim_{R\to\infty}\frac{| \mathcal G^R (\Lambda_\omega) \cap I|}{N(\omega, R)} = \int_I f(x)\,dx$$
for almost every doubled slit torus $\omega$ with respect to an  $SL_2(\R)$-invariant probability measure on $\E$ that is in the same measure class as Lebesgue measure. 
The density function is continuous, piecewise differentiable with 3 points of non-differentiability, and each piece is expressed in terms of elementary functions.
\end{thm}

The density function $f$ is the same for both gap distribution questions. For the remainder of the paper we focus on proving the above theorem since it implies Theorem \ref{ThmGapSlit}. Indeed, the two gap distribution questions are related by applying the diagonal matrix 
$$\gamma_R= \left(
  \begin{array}{ c c }
     R^{-1} & 0 \\
     0 & R
  \end{array} \right)$$
 which takes vectors in the first quadrant with (max norm)  norm less than or equal to $R$ and sends them to vectors in the vertical strip. 
Notice $\gamma_R$ changes slopes of vectors in the first quadrant by a factor of $1/R^2$ and so the renormalized gaps from Equation \ref{NormalizedGapsSlit} become the unnormalized gaps from Equation \ref{UnnormalizedGapsSlit}. Thus, 
$$\lim_{R\to\infty}\frac{| \mathcal G^R(\Lambda_\omega) \cap I|}{N(\omega, R)} = \lim_{R\to\infty}\frac{| \mathcal G^{N(\omega, R)} _V(\gamma_R\cdot\Lambda_\omega) \cap I|}{N(\omega, R)}$$
and we have reduced proving Theorem \ref{ThmGapSlit} to proving Theorem \ref{ThmGapSlitV}, a gap distribution for the gaps coming from slopes of vectors in the vertical strip $V.$

%%%%%%%%%%%%%%%%%%%%%%%
\subsection{From gaps to transversals}

We show how to implement the strategy from \cite{MR3497256} where they translated the question of gaps between slopes of vectors into return times under the horocycle flow on an appropriate moduli space to a specific transversal. This method was utilized in  \cite{MR3330337},\cite{MR3214280}, and \cite{MR3567252}. By \emph{transversal} to horocycle flow we mean a subset $T$ so that almost every orbit under horocycle flow intersects $T$ in a non-empty, countable, discrete set of times. 

We consider the transversal on $\E$ under horocycle flow given by the set of doubled slit tori that contain a non-zero horizontal saddle connection with horizontal component less than or equal to 1.  Denote this by $\W$. That is,
 $$\mathcal W = \{\omega\in \mathcal E: \Lambda_\omega \cap (0,1]\ne \emptyset \}$$
where $\Lambda_\omega$ denotes the set of holonomies of saddle connections.
We will call a horizontal saddle connection with a horizontal component less than or equal to one a \emph{short} saddle connection and an element in $\W$ a \emph{short} doubled slit torus.  Occasionally the phrase ``cross section" or ``Poincar\'{e} section" will be used instead of transversal, but these all mean the same thing in this paper.

 Let us denote the return time map to $\W$ under horocycle flow by $\mathcal R$ and the return map by $\mathcal T$. Then, $\mathcal R:\W\to (0,\infty)$ is given by 
$$\mathcal R(\omega)=\min\{u>0: h_u (\omega)\in \W\}$$
and  $T:\W\to\W$ is given by
$$\mathcal T(\omega) = h_{\mathcal R(\omega)}(\omega).$$

A simple, but key, observation is that the horocycle flow preserves slope differences of vectors $\vec{v}$ since 
$$\text{slope}(h_u\vec{v})=\text{slope}(\vec{v})-u.$$
Consequently this holds for holonomy vectors of saddle connections.  This observation links slopes with $\W$ since if we apply horocycle flow to an element in $\W$, we see that the next vector to become short is the vector with short horizontal component and smallest slope, and the return time to the Poincar\'{e} section is exactly its slope. Similarly, the second vector to become short will be the vector with short horizontal component and second smallest slope and the return time will be its slope minus the return time of the first which is a \emph{slope difference}. Continuing in this fashion we see that the $i^{th}$-return time is the is the difference of the $i^{th}+1$ slope and the $i^{th}$. This is formalized with the following 
$$\mathcal R(\mathcal  T^i(\omega)) = s_{i+1}-s_i.$$ 
 Hence, we can relate our gap distribution to this induced dynamical system via 
\begin{equation}\label{slopetoerg}
\frac{1}{N}|\mathcal G^N _V(\Lambda_\omega)\cap I|=\frac{1}{N}\sum_{i=0} ^{N-1}\chi_{\mathcal  R^{-1}(I)}(\mathcal  T^i(\Lambda_\omega))
\end{equation}
where $\mathcal  R^{-1}(I)$ is the set of slope differences in an interval $I$. Thus, the gap distribution reduces to understanding return times of $\W$.  The remainder of this paper will be concerned with computing the return times.

\subsection{Moduli space}
We introduce the moduli space of twice marked tori and explain its connection to the space of doubled slit tori $\E$.

By a \emph{twice marked torus}, we simply mean a (unit area) torus with a preferred vertical direction and two marked points. We denote the set of all twice marked tori by $\mathcal H{(0,0)}$. Concretely, a twice marked torus is given by the data of a torus $\C/g\Z^2$, the flat metric $\,dz$ inherited from the complex plane, and two district marked points $v_1+\Z^2,v_2+\Z^2.$ We will always assume that the first marked point is at the origin and denote the second marked point by $v$. This data is equivalent to an \emph{affine lattice} $$\Lambda = g\Z^2+v$$
and we denote the space of all affine lattices by $AX_2.$ By considering the affine action of $SL_2(\R)\ltimes \R^2$ on $AX_2$, we may identify $AX_2$ (and hence $\mathcal H(0,0)$) as the homogenous space $$\mathcal H(0,0)=AX_2=SL_2(\R)\ltimes \R^2/ SL_2(\Z)\ltimes \Z^2.$$
We denote an element in any of the above spaces (and hence all) by $(g,v)$ and leave the coset implicitly defined. Denote by $X_2$ the \emph{space of lattices} which we can identify with the homogeneous space
$$X_2 = SL_2(\R)/SL_2(\Z).$$
There is the projection map $\mathcal H(0,0)=AX_2\to X_2$ given by $(g,v)\mapsto g$. This will be a torus bundle since the fiber of a point $g\Z^2$ in $X_2$ is a vector $ v$ that is only defined up to $g\Z^2$ and hence we think of $v$ as living in the torus $\R^2/g\Z^2$. 

Recall, that $\E$ is the family of translation surfaces obtained by gluing two identical tori along a slit. Every such surface projects to a twice marked torus by restricting to one torus and forgetting the slit. By reversing this process we can obtain a doubled slit torus from a twice marked torus, although some care must be taken. Indeed, suppose we have an element in $\mathcal H(0,0)$ which depends on the data $g\in SL_2(\R)$ and $v\in \C/g\Z^2$. We can take two copies of it and then consider a slit between the marked points to construct an element $\omega = \omega_{(g,v)}$ in $\E.$ There are four oriented trajectories between the marked points that we can choose corresponding to the ``corners" of the fundamental domain of $\C/g\Z^2$ spanned by the columns of $g$. Hence, each pair $(g,v)\in \mathcal H(0,0)$  gives rise to 4 elements in $\E$ which is to say there is a degree 4  map
$$\Pi:\E\to\mathcal H(0,0)$$
 given by ``forgetting the slit". Any time $\Pi(\omega)=(g,v)$ we use the notation $\omega = \omega_{(g,v)}$. Since the degree of $\Pi$ is four, then $\E$ can be identified as four copies of $\mathcal H(0,0)$, one copy for each oriented trajectory between the marked points.

Ultimately we are interested in saddle connections of elements in $\E$ and one way to find saddle connections of a doubled slit torus is to look for trajectories between marked points of the twice marked torus. This explains our interest in $\mathcal H (0,0)$.

\subsection{Gap distribution for twice marked tori}

We also consider the gap distribution of twice marked tori $\mathcal H(0,0)$. However, instead of considering all trajectories between marked points we will only be interested in those between \emph{distinct} marked points. The reason for this is because these trajectories will define an affine lattice. 

Let $\Lambda_\omega ^d$ denote the set of  all trajectories between marked points between distinct marked points.. Let  $\mathcal S_V (\Lambda_\omega ^d) $
denote the set of slopes of holonomy vectors in the vertical strip $V$. Then as before, we can order them
 $$\mathcal S_V (\Lambda_\omega ^d) = \{0<s_1<\cdots<s_N<\cdots\}.$$
Let $\mathcal S_V ^N$ denote the first $N$ slopes. We can associate to these slopes the set of \emph{gaps}
$$
\mathcal G^N _V(\Lambda_\omega ^d) = \left\{s_i-s_{i-1}|i = 1,\ldots, N \right\}
$$and consider a gap distribution question for the slopes in the vertical strip $V$
$$
\lim_{N\to\infty}\frac{| \mathcal G^N _V(\Lambda_\omega ^d) \cap I|}{N}
$$
for an interval $I.$
We prove

\begin{thm}\label{ThmGapTori}
There is a limiting probability density function $ g : [0,\infty) \to [0,\infty)$, with
$$\lim_{N\to\infty}\frac{| \mathcal G^N _V(\Lambda_\omega ^d) \cap I|}{N}=\int_I g(x)\,dx$$
for almost every twice-marked torus $\omega \in \mathcal H(0,0)$ with respect to an  $SL_2(\R)$-invariant probability measure on $\H(0,0)$ that is in the same measure class as Lebesgue measure. 
\end{thm}

In fact, we prove that the slope gap distribution for twice-marked tori \emph{always} exists, but is not always the same. This is done by proving a measure classification result for the transversal under the map induced by horocycle flow (see Section 4.1). The gap distribution density is the same for twice-marked tori that are in the support of the same measure.

Moreover if $\omega\in\mathcal H(0,0)$ corresponds to $(g,v)\in AX_2$, then $\Lambda_\omega ^d$ is the affine lattice $g\Z^2+v$. Hence, the above theorem can be interpreted as a theorem about gaps of slopes of an affine lattice. The distribution of gaps of angles of affine lattices was studied in~\cite{MR2726104}. They show that the gap distribution agrees with the distribution of gaps of the sequence $(\sqrt{n})_{n\in \N}$  on the unit circle found by Elkies-McMullen~\cite{MR2060024}.

A consequence of Theorem \ref{ThmGapTori} is that the gap distribution of affine lattices has a quadratic tail and hence is not exotic (Theorem \ref{Estimates} part 1). As with the gap distribution of doubled slit tori, we  prove much finer results, but we wait until section 4 Theorem \ref{Estimates} before more precisely stating them.
Our proof technique is the same as for doubled slit tori, namely, we translate the gap distribution question to one of return times to a transversal on the moduli space of twice marked tori $\mathcal H(0,0).$ The explicit computation of these return times addresses a question from \cite[Section 6.2]{MR3497256}.

The transversal we consider is given by the set of twice marked tori that have a horizontal saddle connection from \emph{distinct} marked points of length less than or equal to 1. Denote this set by  $\Omega$. Under the identification of $\mathcal H(0,0)$ and $AX_2$, this is the same as the set of affine lattices that contain a non-zero horizontal vector with horizontal component less than or equal to 1. That is,
 $$\Omega=\left\{\Lambda  \in AX_2: \Lambda\cap (0,1]\ne \emptyset\right\}.$$ 
We will call a horizontal vector with a horizontal component less than or equal to one a \emph{short} vector and an element in $\Omega$ a \emph{short} affine lattice. Henceforth we will only think about elements $(g,v)$ as affine lattices.

%%%%%%%%%%%%%%%%%%%%%%%%%%%%%%%%%%%%%%%%%%%%%%%%%%%%%%

%%%%%%%%%%%%%%%%%%%%%%%%%%%%%%%%%%%%%%%%%%%%%%%%%%%%%%%%%
 \section{Parameterization of the transversal and return times}
We give coordinates to the transversals from the last section and compute the first return time under these coordinates. 
\subsection{On $\mathcal H(0,0)$}
In this section we parameterize our Poincar\'{e} section $\Omega$ and compute the first return time to $\Omega$. We also recall some results from \cite{MR3214280} that will aid us.
\subsubsection{A Poincar\'{e} section for $X_2$}
Before parameterizing our section, we recall some results from \cite{MR3214280} where they considered a Poincar\'{e} section for $X_2$ which they used to obtain statistical properties of Farey fractions.  We fix the notation
$$ p_{a,b} :=  \left(
  \begin{array}{ c c }
     a & b \\
     0 & a^{-1}
  \end{array} \right)$$
and 
$$ g_{t} :=  \left(
  \begin{array}{ c c }
     e^{t} & 0 \\
     0 & e^{-t}
  \end{array} \right).$$
Let $\Delta$ denote the set of lattices that contain a short vector (recall that \emph{short} for us means length less than or equal to 1). It is shown in\cite{MR3214280} that this set is
$$\Delta = \left\{\Lambda_{a,b} :=  p_{a,b}\Z^2 = \left(
  \begin{array}{ c c }
     a & b \\
     0 & a^{-1}
  \end{array} \right) \Z^2 : 0<a\le1, 1-a<b\le1\right\}.$$ There is a bijection between $\Delta$ with the actual triangle  
$$\{(a,b):0<a\le1, 1-a<b\le1\}\subset \R^2$$ 
via $(a,b)\mapsto \Lambda_{a,b}$
 and thus we give an element in $\Lambda_{a,b}\in\Delta$ the coordinates $(a,b)$. Because of the bijection we abuse notation and use $\Delta$ for the subset of $X_2$ and the actual triangle.

We now restate the main theorem  \cite[Theorem 1.1]{MR3214280} with slightly different notation.
\begin{thm} The subset  $\Delta\subset X_2$ is a Poincar\'{e} section for the horocycle action on $X_2$. More precisely, every horocycle orbit $\{h_s \Lambda\}_{s\in \R}$ (outside a codimension 1 set of lattices) intersects $\Delta$ in a nonempty, countable, discrete set of times.

Moreover, the first return time map $r:\Delta\to [0,\infty)$ is given by
$$r(\Lambda_{a,b})=r(a,b)=\min\{s>0: h_s\Lambda_{a,b}\in \Delta\}=(ab)^{-1}$$
and the return map $S:\Delta\to\Delta$ is 
$$S(\Lambda_{a,b}):=h_{r(a,b)}\Lambda_{a,b}:=\Lambda_{S(a,b)}$$
where $S(a,b)=(b,-a+\lfloor\frac{1+a}{b}\rfloor b)$. 
\end{thm}

They also identify the codimension 1 set as those lattices that contain a short vertical vector which geometrically correspond to embedded closed horocycles that foliate the cusp.  Each such horocycle can be parameterized by an $0< a\le1$. Dynamically, each horocycle is a periodic orbit under horocycle flow $h_s$ with period $a^{2}$ and orbit
$$\text{Per}(a):=\left\{ h_sg_{\log(a)} \Z^2= \left(
   \begin{array}{ c c }
     0 & -a^{-1} \\
     a & sa^{-1}
  \end{array} \right)\Z^2:s\in[0,a^2)\right\}.$$

Notice above we implicitly used that
$$h_sg_{\log(a)} \Z^2=h_sg_{\log(a)} \left(
  \begin{array}{ c c }
     0 & -1 \\
     1 & 0
  \end{array} \right) \Z^2 = \left(
   \begin{array}{ c c }
     0 & -a^{-1} \\
     a & sa^{-1}
  \end{array} \right)\Z^2 $$
so as to further drive the point that such elements are outside of the generic set of lattices. Of course, any $a>0$ gives rise to a periodic orbit $\text{Per}(a)$, but only the $0<a\le 1$ give rise to points the never enter the transversal $\Delta.$

As a consequence of their theorem we have that every $g\in X_2$ is of two forms:
\begin{enumerate}
\item $g=h_sp_{a,b}$ where $0<a\le1$, $1-a<b\le1$, and $0\le s<(ab)^{-1}$. These are the generic elements in $X_2$.
\item $g = h_sg_{\log(a)}$ where $0\le a<1$ and  $0\le s<a^2$. These are the elements in the codimension 1 set of lattices with a short vertical vector.
\end{enumerate}

\subsubsection{A Poincar\'{e} section for $AX_2$}

Beginning in this section we will use the identification between twice marked tori  $\mathcal H(0,0)$ and affine lattices $AX_2$ and mention our results in the language of affine lattices.
Observe that since the dimension of the space of affine lattices $AX_2$ is five that this cross section will reduce the dimension to 4. Thus, we aim find a four dimensional parameterization.

We prove a lemma that says that any affine lattice $(g,v)\in \Omega$ can be expressed as an affine lattice where the affine piece is horizontal.
\begin{lem}
 $$\Omega=\left\{(g,v)\in AX_2: v= \left(
  \begin{array}{ c c }
     \a  \\
     0 
  \end{array} \right) \text{ and }\a\in(0,1]\right\}.$$ 
\end{lem}
\begin{proof}
Suppose that $(g,v)\in\Omega$ i.e. that the affine lattice $(g,v)$ contains a short horizontal vector.
Recall that $v$ is only defined up to elements in $g\Z^2$. Hence, if
$ w\in g\Z^2$, then as an affine lattice we have 
$$g\Z^2 + v =g\Z^2 + ( v +w) .$$

In particular, if we assume our affine lattice is short, then there is some  $ w\in g\Z^2$ so that 
$w+ v$ is a short horizontal vector. 
In particular, it is of the form
$$w+v = \left(
  \begin{array}{ c c }
     \a  \\
     0 
  \end{array} \right) $$
for some $\a\in(0,1]$. Thus,
$$g\Z^2 + v =g\Z^2 + v +w  = g\Z^2 +\left(
  \begin{array}{ c c }
     \a  \\
     0 
  \end{array} \right).$$\end{proof}

We use this lemma along with the coordinates on the space of lattices developed in  \cite{MR3214280} to parameterize $\Omega$.

\begin{prop}\label{ParameterizationOmega}
The set of of short affine lattices is given by the union of 
$$\left\{\left(h_sp_{a,b}, \left(
  \begin{array}{ c c }
     \a  \\
     0 
  \end{array} \right)\right) :a\in(0,1], b\in(1-a,1], s\in\left[0,(ab)^{-1}\right),\alpha\in(0,1])\right\}$$
and 
$$\left\{\left(h_sg_{\log(a)}, \left(
  \begin{array}{ c c }
     \a  \\
     0 
  \end{array} \right)\right) :a\in(0,1],s\in(0,a^2],\alpha\in(0,1])\right\}.$$
\end{prop}
\begin{proof}
First notice that any affine lattice from the union is a short affine lattice.

Now suppose $(g,v)\in \Omega$. By the last lemma we can assume $v$ is a short vector. Combining this with the coordinates developed in \cite{MR3214280}, which we recalled in the last section, we can finish the proof of our parameterization. By the results of \cite{MR3214280}, a generic lattice $g\Z^2\in X_2$ (off a codimension 1 set) has the form
$$g\Z^2=h_s p_{a,b}\Z^2=\left(
  \begin{array}{ c c }
     a & b \\
     -sa & a^{-1}-sb
  \end{array} \right)\Z^2$$ with $0<a\le1$, $1-a<b\le1$, and $0<s<(ab)^{-1}$.

Thus, we can give such an element in $(g,v)\in\Omega$  the coordinates $(a,b,s,\alpha)$ where 
$$g\Z^2=h_s p_{a,b}\Z^2 \text{ and }v=\left(
  \begin{array}{ c c }
     \a  \\
     0 
  \end{array} \right).$$
 
If $g$ is in the codimension 1 set, then it is of the form
$$g\Z^2=h_sg_{\log(a)}\Z^2=\left(
  \begin{array}{ c c }
    0  & -a^{-1} \\
     a & sa^{-1}
  \end{array} \right)\Z^2$$ with $0<a\le1$, and $0<s<a^2$ and we give it the coordinates $(a,s,\a).$
\end{proof}

We note that the affine lattice $(a,0,s,\a)$ is \emph{not} the same as the affine lattice $(a,s,\a)$. 

%%%%%%%%%%%%%%%%%%%%%%%%%%%%%%%%%%%%%%%%%%%%%%%%%%%%%%%%%%%%%%%%55

%%%%%%%%%%%%%%%%%%%%%%%%%%%%%%%%%%%%%%%%%%%%%%%%%%%%%%%%%%%%%%%%%%%%

\subsubsection{First return time to $\Omega$}
As mentioned above, our goal is to find the return map to our transversal which will calculate the gap distribution for slopes in the vertical strip $V$. We outline the strategy of how to do this: If we had the point $(g,v)\in\Omega$ i.e.  an affine lattice with a short horizontal vector, then we have  $v=\left(
  \begin{array}{ c c }
     \a  \\
     0 
  \end{array} \right)$ for some $0<\alpha \le 1$. Choose any vector $\left(
  \begin{array}{ c c }
     x  \\
     y 
  \end{array} \right)\in g\Z^2$. Then the horocycle flow action on the vectors of the affine lattice $(g,v)$ is given by
$$h_u\left(
  \begin{array}{ c c }
     x \\
    y
  \end{array} \right) + h_u\left(
  \begin{array}{ c c }
     \a \\
    0
  \end{array} \right)=\left(
  \begin{array}{ c c }
     x+\alpha \\
    y-s(x+\alpha)
  \end{array} \right).$$ 
Thus, if we want a short vector, then we need  the horizontal $x+\a$ between 0 and 1, $y>0$ (so that $(x+\a,y)$ is in the vertical strip), and $$u=\frac{y}{x+\alpha}$$
so that the vertical is zero. As there are several lattice points that will meet this criteria and since we are seeking the first return time, we will minimize $u>0$ over all lattice points in $g\Z^2$ and the vertical strip $V$.  Which is to say we want to compute
$$u=\min_{(x,y)^T\in g\Z^2\cap V}\frac{y}{x+\alpha}.$$

%%%%%%%%%%%%%%%%%%%%%%%%%%%%%%%%%%%%%%%%%%%%%%%%%%%%%%%%%%%%%%%%%%%%%%

We can unpack a little more by observing that an element in a generic lattice has the form $g=h_sp_{a,b}$ where $0<a\le1$, $1-a<b\le1$, $0\le s<{(ab)}^{-1},$ and $0<\alpha\le 1$.

Thus, we are trying to find the minimum of the slopes of these vectors  
$$R(a,b,s,\a):= \min_{m,n\in \Z}\frac{(a^{-1}-sb)n - sam}{ma+nb + \alpha} $$
subject to the constraints:
\begin{itemize}
\item[1.] $0<ma+nb + \alpha\le 1$  so that the horizontal is short.
\item[2.] $ (a^{-1}-sb)n - sam$ is positive so that we are sampling points in the vertical strip $V$.
\end{itemize}
because this minimum is the return time. Our main result in this part is an explicit return time computation for generic affine lattices.

\begin{thm} \label{ReturnTimeMap}Let $j=\left\lfloor \frac{1+a-\alpha}{b}\right\rfloor$. If $\a>a$, then
$$R(a,b,s,\a) = \begin{cases}
 \frac{sa}{\a-a}, &\text{ if } s<\frac{\a-a}{ab\a}\\
 \frac{j(a^{-1}-sb)+sa}{\a-a+jb}, &\text{ if } s>\frac{\a-a}{ab\a}.
\end{cases}$$

If $\a\le a$, then
$$R(a,b,s,\a) = \begin{cases}
 \frac{a^{-1}-sb}{b+\a}, &\text{ if } b+\a<1\\
 \frac{j(a^{-1}-sb)+sa}{\a-a+jb}, &\text{ if } b+\a>1.

\end{cases}$$
\end{thm}
%%%%%%%%%%%%%
\begin{proof}
%%%%
The main idea is to notice that for any fixed $n$, there is a unique $m=m_n$ for which the affine lattice point
$$\left(
  \begin{array}{ c c }
     ma+nb + \alpha \\
    (a^{-1}-sb)n - sam
  \end{array} \right)$$ 
may be a candidate to minimize slope. From there we can single out a candidate vector that provides an upper bound to the minimum slope and then use that candidate along with the structure of the affine lattice to find the actual minimum. 

Fix $n\in\Z$, if we have a candidate affine lattice point, then the horizontal must satisfy
$$0<ma+nb + \alpha\le 1\iff \frac{-(nb+\a)}{a}<m \le \frac{1-(nb+\a)}{a}.$$
Since we require $m\in \Z$ we have that $m$ must be one of the following:
$$\bigg\lceil \frac{-(nb+\a)}{a}\bigg\rceil, \bigg\lceil \frac{-(nb+\a)}{a}\bigg\rceil +1, \bigg\lceil \frac{-(nb+\a)}{a}\bigg\rceil +2, \ldots, \bigg\lfloor \frac{1-(nb+\a)}{a} \bigg\rfloor. $$ 
Now notice that to minimize the slope
$$\frac{(a^{-1}-sb)n - sam}{ma+nb + \alpha} $$
we should maximize $m$. Hence, we choose  $m_n = \bigg\lfloor \frac{1-(nb+\a)}{a} \bigg\rfloor$. In particular, for every $n\in\Z$, there is a single $m=m_n$ for which the associated affine lattice point may be a minimizer.

Using this fact, we will single out a candidate and using the structure of the lattice to find the true minimum.

We proceed by cases:
\begin{itemize}

\item Case 1: $\a<a$. Consider the candidate minimizers when $n=1$. Since their horizontal must be short we deduce that
$$0<ma + b + \a \le 1 \iff \frac{-(b+\a)}{a}<m\le \frac{1-(b+\a)}{a} $$
Since we are trying to minimize the slope $\frac{a^{-1}-sb - sam}{ma+b+\a}$, then we want to maximize $m$ and so 
$$m=\left\lfloor \frac{1-(b+\a)}{a}\right\rfloor.$$ We claim that the only possible values for this floor are 0 and 1.  Indeed, notice by our choice of coordinates we have
$$1-a<b\le1 \iff-\a \le 1-(b+\a)<a-\a.$$
An obvious upper bound is $a$ and due to our assumption that $\a<a$ we have that a lower bound is $-a$. Hence,
 $$-a \le 1-(b+\a)<a\iff -1 \le \frac{1-(b+\a)}{a}<1.$$
This proves the claim that $m\in\{0,1\}$. Moreover, we see that 
$$m=\left\lfloor \frac{1-(b+\a)}{a}\right\rfloor =  \begin{cases}
0 &\text{ if } b+\a<1\\
-1 &\text{ if } b+\a>1.\\
\end{cases}$$
Hence, in the case that $\a<a$, the candidate when $n=1$ is
$$\gamma_1:= \begin{cases}
\left(
  \begin{array}{ c c }
    b+\a \\
    a^{-1}-sb
  \end{array} \right) &\text{ if } b+\a<1\\
\left(
  \begin{array}{ c c }
     b + \a -a \\
    a^{-1} -sb+sa
  \end{array} \right) &\text{ if } b+\a>1.\\
\end{cases}$$
Now, the minimum slope is bounded above by slope$(\gamma_1)$ and that to get any other candidate vector we must add/subtract multiples of the basis vectors of $\left(
  \begin{array}{ c c }
     a &b \\
    -sa & a^{-1}-sb
  \end{array} \right)\Z^2.$ Recall we found $\gamma_1$ by fixing $n=1$, this means that we cannot obtain any other candidates by adding/ subtracting the first basis vector. First notice that 
$$\gamma_1 -\left(
  \begin{array}{ c c }
     b \\
    a^{-1} -sb
  \end{array} \right)$$
has vertical zero so that it is not a candidate to minimize slope. Hence, to find other candidates to minimize slope, our only option to is to add the second basis vector $\left(
  \begin{array}{ c c }
     b \\
    a^{-1} -sb
  \end{array} \right)$ to $\gamma_1$ to get the other candidates at different levels (as long as the horizontal less than 1). For example, the level 2 candidate $$\gamma_2=\gamma_1 + \left(
  \begin{array}{ c c }
     b \\
    a^{-1} -sb
  \end{array} \right).$$
 More generally, all of our candidates are$$\gamma_n=\gamma_1 + (n-1)\left(
  \begin{array}{ c c }
     b \\
    a^{-1} -sb
  \end{array} \right)$$
where $n\ge1$ and so long as the horizontal is less than 1. Of course, once the horizontal becomes bigger than 1, we subtract the other basis vector, but since we are decreasing the horizontal, we are increasing the slope so that we do not need to look at these candidates.

 Now notice that 
$$ \begin{cases}
\text{ if } b+\alpha<1, \text{ adding } \left(
  \begin{array}{ c c }
     b \\
    a^{-1} -sb
  \end{array} \right) \text{ to $\gamma_1 $ increases slope.}  \\
 \text{ if } b+\alpha>1,\text{  adding } \left(  
  \begin{array}{ c c }
     b \\
    a^{-1} -sb
  \end{array} \right) \text{ to $\gamma_1$ decreases slope}.\\
\end{cases}$$
  We conclude that the return time is

$$\begin{cases}
 \text{ slope }(\gamma_1), \text{ if } b+\alpha<1\\
 \text{ slope }(\gamma_j), \text{ if } b+\alpha>1 \\
\end{cases},$$
where $j$ is the maximum number of times we can add $\left(
  \begin{array}{ c c }
     b \\
    a^{-1} -sb
  \end{array} \right)$ to $\gamma_1$ so that the horizontal of $\gamma_j$ is smaller than 1. We compute this value of $j$:
This is the largest number so that
$$\text{ horizontal }(\gamma_j) = jb+\alpha -a  \le 1<(j+1)b+\alpha -a \iff j=\left \lfloor \frac{a+1-\alpha}{b}\right\rfloor.$$

\item Case 2: $a<\a$. This case is slightly easier since we have a candidate at $n=0$ given by
$$\gamma_0 = \left(
  \begin{array}{ c c }
     \a \\
    0
  \end{array} \right) - \left(
  \begin{array}{ c c }
     a \\
   -sa
  \end{array} \right)  = \left(
  \begin{array}{ c c }
     \a - a \\
    sa
  \end{array} \right).$$
 This candidate is in the vertical strip $V$ by our assumption $a< \a$ and does not exists in case one. Notice that subtracting the second basis vector will yield a negative horizontal. 

Now, as before, to obtain any other candidate all we can do is add the second basis vector yielding the candidates
$$\gamma_n = \gamma_0 + n\left(
  \begin{array}{ c c }
     b \\
    a^{-1} -sb
  \end{array} \right)$$
where $n\ge0.$ Since we are considering a minimum, then slope$(\gamma_0)$ is an upper bound. When we compare the slope of $\gamma_0$ and 
$$\gamma_1 = \gamma_0 + \left(
  \begin{array}{ c c }
     b \\
    a^{-1}-sb
  \end{array} \right)$$
we see that the slope of $\gamma_1$ is 
$$ \begin{cases}
\text{ greater  if }s<\frac{\a-a}{ab\a} \\
 \text{ smaller if } s>\frac{\a-a}{ab\a}
\end{cases}.$$
In the second situation we continue to add the second basis vector as many times as possible so that the horizontal of the new vector is less than or equal to 1.  
We deduce,
$$R(a,b,s,\a) = \begin{cases}
 \text{ slope }(\gamma_0)=\frac{sa}{\a-a}, &\text{ if } s<\frac{\a-a}{ab\a}\\
 \text{ slope }(\gamma_j)=\frac{j(a^{-1}-sb)+sa}{\a-a+jb}, &\text{ if } s>\frac{\a-a}{ab\a} \\
\end{cases},$$
where $j$ is the maximum number of times we can add $\left(
  \begin{array}{ c c }
     b \\
    a^{-1}-sb
  \end{array} \right)$ to $\gamma_0$ so that the horizontal of $\gamma_j$ is smaller than 1. We compute this value of $j$:
This is the largest number so that
$$\text{ horizontal }(\gamma_j) = jb+\alpha -a  \le 1<(j+1)b+\alpha -a \iff j=\bigg\lfloor \frac{1+a-\alpha}{b}\bigg\rfloor.$$
\end{itemize}

\end{proof}
Notice the above theorem was for a generic affine lattice i.e. one such that the lattice part has the form$g=h_sp_{a,b}$. On the codimension 1 set of lattices the proof is even easier.  
\begin{prop}Suppose $a\in(0,1]$. The return map for the the codimension 1 set of lattices with a short vertical vector is given by
$$R\left(\left(h_sg_{\log(a)}, \left(
  \begin{array}{ c c }
     \a \\
    0
  \end{array} \right) \right)\right)=R(a,s,\a) = \frac{a}{\a}.$$
\end{prop}
\begin{proof}
All of the points inside of a short vertical lattice $(a,s,\a)$ are of the form
$$ \left(
  \begin{array}{ c c }
     \a-a^{-1}n \\
    sa^{-1}n +am
  \end{array} \right)$$
for $m,n\in\Z.$ Since we want a slope of such a point in the vertical strip, then we have the constraints $0<\a-a^{-1}n \le 1$ and
$ sa^{-1}n +am>0$. We have
$$0<\a-a^{-1}n \le 1\iff a(\a-1)\le n <a\a.$$
Since $a,\a<1$, then we must have $n=0$. Hence, the slope of the points we are intersted in reduces to minimizing 
$$\frac{ sa^{-1}n +am}{  \a-a^{-1}n }=\frac{am}{\a}$$
subject to $am+sa^{-1}n=am$ positive. Since $a$ is positive, then this is minimized when $m=1.$
\end{proof}

\begin{remark}
We establish some notation suggested by the partition from the theorem. This notation will be used frequently in later sections.
$$\Omega = \Omega_1\sqcup \Omega_2\sqcup \Omega_3\sqcup \Omega_4\sqcup \Omega_{VL}$$
where
\begin{align*}
\Omega_1 &:= \left\{(a,b,s,\alpha)\in \Omega: a<\alpha \text{ and } s<\frac{\a-a}{ab\a}\right\},\\
\Omega_2 &:= \left\{(a,b,s,\alpha)\in \Omega: a<\alpha \text{ and } \frac{\a-a}{ab\a}<s\right\},\\
\Omega_3 &:= \left\{(a,b,s,\alpha)\in \Omega: \a<a \text{ and } b<1-\a\right\}, \\
\Omega_4 &:= \left\{(a,b,s,\alpha)\in \Omega: \a<a \text{ and } 1-\a<b\right\},\\
\Omega_{VL} &:=\left\{(a,s,\a) \in\Omega:a\in(0,1],s\in(0,a^2],\text{ and } \alpha\in(0,1])\right\}.
\end{align*}

Furthermore, we have the return time map is given by
$$R(a,b,s,\a) = \begin{cases}
 \frac{sa}{\a-a}, &\text{ if } (a,b,s,\a)\in\Omega_1\\
 \frac{j(a^{-1}-sb)+sa}{\a-a+jb}, &\text{ if } (a,b,s,\a)\in\Omega_2\\
 \frac{a^{-1}-sb}{b+\a}, &\text{ if } (a,b,s,\a)\in\Omega_3\\
 \frac{j(a^{-1}-sb)}{\a-a+jb}, &\text{ if } (a,b,s,\a)\in\Omega_4\\
 \frac{a}{\a}, &\text{ if } (a,s,\a)\in\Omega_{VL}.
\end{cases}$$
where $j=\left\lfloor \frac{1+a-\alpha}{b}\right\rfloor$. 

We remark that in the next sections we will mostly be uninterested in the piece $\Omega_{VL}$ since this is a lower dimensional subset of $\Omega.$
\end{remark}

\subsection{On $\E$}
\subsubsection{A Poincare section for $\E$}

We describe a parameterization of $\W$. Before doing this, we describe more explicitly saddle connections of a doubled slit torus.

 For an element $\omega=\omega_{(g,v)}\in\E$, the cone points correspond to either the origin of the torus $\C/g\Z^2$ or at the affine piece $v\in\C/g\Z^2$. Hence,  any saddle connection begins at  the origin or to the affine piece $v$ and end at the origin or the affine piece $v$.  Thus, the saddle connections of $\omega=\omega_{(g,v)}$ are
$$\Lambda_{(g,v)}:=\Lambda_{\omega_{(g,v)}} = g\Z^2 _\text{prim}\cup (g\Z^2 + v)\cup (-g\Z^2 - v).$$
Here, $g\Z^2 _\text{prim}$ is the set of relatively prime vectors in $\Z^2$. This set represents the saddle connections beginning and ending at themselves while $g\Z^2 + v$ represents those saddle connections beginning at one of the singularities and ending at the other one. The piece $-g\Z^2 - v$ corresponds to saddle connections starting at the affine piece and ending at the origin.  Our description above is only true under our standing assumption that $v$ has irrational slope. 

The explicit description of saddle connections of doubled slit tori allow us to describe $\mathcal W$ more explicitly as well. Essentially, what we will deduce is that for a doubled slit tori $\omega=\omega_{(g,v)}$ to have a short horizontal saddle connection (i.e. to be in $\W$), then the short horizontal saddle connection must come from the primitive vectors  $g\Z^2 _\text{prim}$ or from the affine lattice $(g\Z^2 + v)$.

 Recall, $\Delta = \{p_{a,b}\Z^2:(a,b)\in \Delta\}$ and 
$$\Omega = \{h_sp_{a,b}\Z^2 + \left(
  \begin{array}{ c c }
     \a \\
    0
  \end{array} \right): (a,b)\in\Delta, s\in [0,(ab)^{-1}), \alpha\in (0,1]\}.$$
We write $\Omega$ more succinctly as 
$$\Omega =\left\{ (g,v) = g\Z^2 +v: g\in SL_2(\R),v=\left(
  \begin{array}{ c c }
     \a \\
    0
  \end{array} \right)\right\}.$$
Denote the projection map from $AX_2 \to X_2$ by $\pi$.  Then, define
$$\mathcal W _{sl} = \{\omega\in \mathcal E: \Pi(\omega)\in \pi^{-1}(\Delta)\}$$
and 
$$\mathcal W _{sa} = \{\omega\in \mathcal E : \Pi(\omega)  \in \Omega\}.$$
Of course both of these sets are preimages of $\Pi$ but we use this notation to avoid notation overload. It is not hard to see that $\mathcal W _{sl}$ consists of those doubled slit tori so that the image under $\Pi$ is an affine lattice where the lattice part is short. Thus, it is a \emph{short lattice} and that is where the subscript ``$sl$" comes from. Similarly, $\mathcal W _{sa}$ consists of the doubled slit tori so that the image under $\Pi$ is an affine lattice where the affine piece is short. Thus, it has a \emph{short affine} vector and that is where the subscript ``$sa$" comes from. Unwrapping a bit more we see,
$$\mathcal W _{sl}=\{\omega\in\mathcal E: \Pi(\omega)= (p_{a,b},v)\} \text{ and }\mathcal W _{sa} =\left\{\omega \in \mathcal E:\Pi(\omega)= \left(g,  \left(
  \begin{array}{ c c }
     \a \\
    0
  \end{array} \right)\right)\right\}$$
for where $(a,b)\in \Delta$ and $\alpha\in(0,1]$.

This discussion proves that we have the following description of $\mathcal W$: if $\omega$ is short ($\omega\in \mathcal W$), then the short vector is either in the lattice part $g\Z^2_{prim}$ or in the affine lattice part $g\Z^2 +v$. We state this as a proposition.
\begin{prop}\label{ParameterizationW} The transversal $\mathcal W$ is
$$\mathcal W = \mathcal W _{sl}\cup \mathcal W _{sa}.$$
\end{prop}

For doubled slit tori we will remove lower dimensional subsets such as the affine lattices that have a periodic lattice.
%%%%%%%%%%%%%First Return map%%%%%%%%%%%%%%%%%%%%%%%%%%%%%%%%%%%%%%%%%%%%%%%%
\subsubsection{First return time to $\W$}
We compute the time it takes for an element in $\mathcal W$ to return to $\mathcal W$ under the horocycle flow. This return time will require the return map on $AX_2$ computed in the last section.

Our strategy to do this is to consider the return time for an element $\omega$ on the partition $\mathcal W_{sl}\cup \mathcal W_{sa}$ of $\mathcal W$. Then $\Pi(\omega) = (g,v)$ can be decomposed into a lattice piece $g$ and an affine vector piece $v$. We can check exactly how much time needs to pass under horocycle flow before returning $\mathcal W$  using the results of \cite{MR3214280} and the results on affine lattices. Then the minimum of the two times will be the return time to $\Omega$.

\begin{thm}\label{ReturnTimeW}
Suppose $\omega \in \mathcal W$ and let $u$ denote the return time to $\mathcal W$ under the horocycle flow. Then we have two cases, either $\omega\in \mathcal W_{sa}$ or $\omega \in \mathcal W_{sl}$.
\begin{itemize}
\item Suppose $\omega\in \mathcal W_{sl}$. Then $\Pi(\omega)=(p_{a,b},v)$ for some $a,b\in\Delta$ where $v=(v_1,v_2)^T$ and the return time to $\mathcal W$ under horocycle flow is 
$$\mathcal R (\omega)= \begin{cases}
   v_2/v_1, &\text{ if } b+v_1 \le 1 \text{ and }\\
 (ab)^{-1}, &\text{ if } b+v_1 > 1.\\
\end{cases}$$

\item Suppose $\omega\in \mathcal W_{sa}$. Then $\Pi(\omega)=(h_sp_{a,b},(\a,0)^T)$ for $(a,b,s,\a)\in \Omega$ and the return time to $\mathcal W$ under horocycle flow is 
$$\mathcal R(\omega) = \begin{cases}
   (ab)^{-1} -s, &\text{ if } \Pi(\omega)\in\Omega_2\cup\Omega_4,\\
 \frac{sa}{\a-a}, &\text{ if } \Pi(\omega)\in\Omega_1 \text{ and }\\
\frac{a^{-1}-sb}{b+\a}, &\text{ if } \Pi(\omega)\in\Omega_3,
\end{cases}$$
where $\Omega_i$, $i=1,\ldots,4$ form the atoms of the partition of $\Omega$.

\end{itemize}
\end{thm}
%%%%%%%%%%%%

We first prove a lemma that computes the time it would take to travel from $\mathcal W _{sl}$ to $\mathcal W _{sl}$ under horocycle flow.

\begin{lem} Let $\rho = \rho (a,b,v_1,v_2)$ denote the time it would take to travel from $\mathcal W _{sl}$ to $\mathcal W _{sl}$ under horocycle flow. Then
$$ \rho (a,b,v_1,v_2) =  \begin{cases}
  \frac{v_2}{v_1}= \text{ slope }(v) , &\text{ if } b+v_1 \le 1 \text{ and }\\
   \frac{v_2 + ja^{-1}}{v_1+jb-a}, &\text{ if } b+v_1 > 1\\
\end{cases}$$

\end{lem}

\begin{proof}
Suppose that  $\omega\in\mathcal W _{sl}$. Then $\Pi(\omega) =(p_{a,b},v)$ for some $(a,b)\in\Delta.$ Then the only way $(p_{a,b}, v)$  will have a short horizontal affine vector under the horocycle flow will be when $h_u v$ is short. Recall, that $v$ is actually only defined up to $p_{a,b}\Z^2$. However, when we write ``$v$" we mean the representative inside of the standard fundamental domain of $\R^2/p_{a,b}\Z^2$.
 \begin{comment}In particular, 
$$v = (v_1, v_2)^T, 0<v_1<a,0<v_2<a^{-1}, \text{ and slope}(v)<\frac{a^{-1}}{b} .$$
We will use this shortly.
\end{comment}

Then $h_{\rho(a,b,v_1,v_2)} v$ is short when $\rho(a,b,v_1,v_2) = \text{slope}(v) = \frac{v_2}{v_1}.$ Since $v$ is actually only defined up to $p_{a,b}\Z^2$ we see the next value $u$ so that $h_u v$ is short is 
$$\rho(a,b,v_1,v_2) = \min \frac{v_2 + a^{-1}n}{v_1 + am + bn}$$
where the minimum is taken over $(m,n)\in \Z$ such that 
\begin{itemize}
\item $\frac{v_2 + a^{-1}n}{v_1 + am + bn}>0$ since we want positive slope/return time and  
\item$v_1 + am + bn \in(0,1]$ since we want a short vector.
\end{itemize}
The first condition says that $v_2 + a^{-1}n>0\iff n>-v_2 a>-1$. Since $n\in \Z$, then $n\ge 0$. 

Now fix $n\ge0$. Then the second condition of $v_1 + am + bn \in(0,1]$ tells us that 
$$\frac{-(nb+v_1)}{a}<m \le\frac{1-(nb+v_1)}{a}.$$
Since we are trying to minimize the slope, then we want to maximize $m$, so that we take
$$m= m_n = \left\lfloor\frac{1-(nb+v_1)}{a}\right\rfloor.$$

Hence, for every $n\ge0$, we have the candidate
$$\gamma_n = \left(
  \begin{array}{ c c }
     v_1+am_n + bn \\
    v_2+a^{-1}n
  \end{array} \right).$$
Now observe that since we are looking at a minimum, over $n$ of the $\gamma_n$ that $\gamma_1$ is an upper bound. We are interested in $\gamma_1$, because we can describe it explicitly. In particular,
 $$m_1 = \begin{cases}
 0, &\text{ if } b+v_1 \le 1\\
 -1, &\text{ if } b+v_1 < 1\\
\end{cases}$$
so that
$$\gamma_1 = \begin{cases}
  \left(
  \begin{array}{ c c }
     v_1 + b \\
    v_2+a^{-1}
  \end{array} \right), &\text{ if } b+v_1 \le 1 \text{ and }\\
  \left(
  \begin{array}{ c c }
     v_1 + b -a \\
    v_2+a^{-1}
  \end{array} \right), &\text{ if } b+v_1 < 1\\
\end{cases}$$
 We are interested in explicitly knowing $\gamma_1$ because this allows us to compute all the other candidates $\gamma_n$. For example, to go to the candidate $\gamma_1$ to the candidate at level $2$, all we have to do is add the basis vector $(b, a^{-1})^T$ of the lattice $p_{a,b}\Z^2$:
 $$\gamma_2 = \gamma_1 + \left(
  \begin{array}{ c c }
      b \\
    a^{-1}
  \end{array} \right).$$
  Similarly, 
$$\gamma_3 = \gamma_1 + 2\left(
  \begin{array}{ c c }
      b \\
    a^{-1}
  \end{array} \right)\text{ and } \gamma_0 = \gamma_1 - \left(
  \begin{array}{ c c }
      b \\
    a^{-1}
  \end{array} \right).$$
  Of course, if we add the basis vector $(b, a^{-1})^T$ too many times, the horizontal will become larger than 1 and hence we cease to have a candidate. This is amended by subtracting the other basis vector $(a,0)^T$, but this makes the slope larger since we are decreasing the horizontal. In other words, the only candidates $\gamma_n$ that we are interested are the ones that we can add the basis vector $(b, a^{-1})^T$ to and so that they still have horizontal less than or equal to 1. In particular, this is a finite list. We can then obtain the candidate with minimum slope by observing how adding the basis vector $(b, a^{-1})^T$ changes the slope compared to $\gamma_1$. Now notice that 
$$ \begin{cases}
\text{ if } b+v_1<1, \text{ adding } \left(
  \begin{array}{ c c }
     b \\
    a^{-1}
  \end{array} \right) \text{ increases slope . }\\
 \text{ if } b+v_1>1,\text{  adding } \left(  
  \begin{array}{ c c }
     b \\
    a^{-1}
  \end{array} \right) \text{ decreases slope. }\\
\end{cases}$$
  We conclude that
  \begin{itemize}
  %%%%%%%%%%%
  \item If  $b+v_1<1$, then the minimum of the slopes is $\min\{\text{ slope }(\gamma_0), \text{ slope } (\gamma_1)\}$. A direct calculation shows the minimum of these two is $\text{ slope }(\gamma_0) = v_2/v_1$.
  %%%%%%%%%%%%%%%%%%%%%
  \item  If  $b+v_1>1$, then the minimum of the slopes is slope $(\gamma_j)$, where $j$ is the maximum number of times we can add $\left(
  \begin{array}{ c c }
     b \\
    a^{-1}
  \end{array} \right)$ to $\gamma_1$ so that the horizontal of $\gamma_j$ is smaller than 1. We compute this value of $j$:
This is the largest number so that
$$\text{ horizontal }(\gamma_j) = jb+v_1 -a  \le 1<(j+1)b+v_1 -a \iff j=\left\lfloor \frac{a+1-v_1}{b}\right\rfloor.$$
  \end{itemize}
Hence, for an element $\omega\in\mathcal W _{sl}$, with $\Pi(\omega)=(p_{a,b},v)$, to return to $\mathcal W _{sa}$ under the horocycle flow, it needs to travel $\rho$ units of time where
$$\rho(a,b,v_1,v_2) = \begin{cases}
   \text{ slope }(\gamma_0) = \text{ slope }(v) = v_2/v_1, &\text{ if } b+v_1 \le 1 \text{ and }\\
  \text{ slope }(\gamma_j) = \frac{v_2 + ja^{-1}}{v_1+jb-a}, &\text{ if } b+v_1 > 1\\
\end{cases}$$
\end{proof}

Now we give the proof of the return time to $\mathcal W$ under horocycle flow.
\begin{proof}[Proof of Theorem \ref{ReturnTimeW}]

\textbf{(Step 1: Computing time from $W_ {sa}/\mathcal W_{sl}$ to $W_{sa}/\mathcal W_{sl}$)}
We implement the first part of the strategy outlined at the beginning of this section and compute how much time is needed for an element in $\mathcal W$ to return to either $\mathcal W_{sl}$ or $\mathcal W_{sa}$.

\textbf{1. Time from $\mathcal W _{sa}$ to $\mathcal W _{sa}$:} Suppose that $\omega \in \mathcal W _{sa}$ so that $\Pi(\omega)=(h_sp_{a,b},(\a,0)^T)$ for $(a,b,s,\a)\in \Omega$. Consequently,  the return time to $\mathcal W _{sa} \subset \mathcal W$ is exactly the return time from the previous section on affine lattices. Namely, the return time is $R(a,b,s,\a)$.

\textbf{2. Time from $\mathcal W _{sa}$ to $\mathcal W _{sl}$:} Suppose that $\omega \in \mathcal W _{sa}$ so that $\Pi(\omega)=(h_sp_{a,b},(\a,0)^T)$ for $(a,b,s,\a)\in \Omega$. Consequently,  under the horocycle flow, the next time we get a short vector in the lattice part is when the lattice part contains a short vector. Hence, the return time to $\mathcal W _{sl}\subset \mathcal W$ is $(ab)^{-1}-s$ by the results of \cite{MR3214280}.

\textbf{3. Time from $\mathcal W _{sl}$ to $\mathcal W _{sl}$:} Suppose that  $\omega\in\mathcal W _{sl}$ so that $\Pi(\omega) =(p_{a,b},v)$ for some $(a,b)\in\Delta.$ Thus, under the horocycle flow, the next time we get a short vector in the lattice part is the return time obtained in \cite{{MR3214280}}. Hence, the return time to $\mathcal W _{sl}\subset \mathcal W$ is $(ab)^{-1}$ by the results of \cite{MR3214280}.

\textbf{4. Time from $\mathcal W _{sl}$ to $\mathcal W _{sa}$:}  The only difficult return time $\rho(a,b,v_1,v_2)$ to compute is for an element $\omega\in\mathcal W _{sl}$ to return to $\mathcal W _{sa}$.  This was computed in the previous lemma to be
$$\rho(a,b,v_1,v_2) = \begin{cases}
 \frac{v_2}{v_1}, &\text{ if } b+v_1 \le 1 \text{ and }\\
 \frac{v_2 + ja^{-1}}{v_1+jb-a}, &\text{ if } b+v_1 > 1\\
\end{cases}.$$

We collect the times collected above now to deduce that for $\omega \in \mathcal W$ the return time $u$ to $\mathcal W$ under the horocycle flow is given by the following two cases.

 (Case 1: $\omega \in \mathcal W_{sl}$)  For $\omega\in \mathcal W_{sl}$ the return time is 
$$u =   \min\{(ab)^{-1}, \rho(a,b,v_1,v_2)\}$$
where $\rho$ is the minimum time it takes for an affine lattice $(p_{a,b},v)$ to have a short affine vector in the affine piece.

 (Case 2: $\omega \in \mathcal W_{sa}$)   For $\omega\in \mathcal W_{sa}$ the return time is
$$u  = \min \{(ab)^{-1}-s,R(a,b,s,\a)\}$$
where $R$ is the return map for affine lattices.

\textbf{(Step 2: Computing the minimum)}
 We now implement the second step of the strategy and compute the minimums of the times calculated from step 1. We have two cases given by the partition of $\mathcal W$.
\begin{itemize}
\item  (Case 1: $\omega \in \mathcal W_{sl}$)  For $\omega\in \mathcal W_{sl}$ the return time is 
$$u =   \min\{(ab)^{-1}, \rho(a,b,v_1,v_2)\}$$
where we have $\Pi(\omega)=(p_{a,b},v)$ for some $(a,b)\in\Delta$ and $v=(v_1,v_2)^T\in p_{a,b}\Z^2$. The return time is 
$$u = \begin{cases}
   \min \{(ab)^{-1},v_2/v_1\}, &\text{ if } b+v_1 \le 1 \text{ and }\\
  \min\{ (ab)^{-1}, \frac{v_2 + ja^{-1}}{v_1+jb-a}\}, &\text{ if } b+v_1 > 1.\\
\end{cases}.$$
By our choice of $v$ as the representative lying in the standard fundamental domain of $\R^2/p_{a,b}\Z^2$ spanned by the columns of $p_{a,b}$ we have 
$$u = \begin{cases}
   \min \{(ab)^{-1},v_2/v_1\} = v_2/v_1, &\text{ if } b+v_1 \le 1 \text{ and }\\
  \min\{ (ab)^{-1}, \frac{v_2 + ja^{-1}}{v_1+jb-a}\} = (ab)^{-1} , &\text{ if } b+v_1 > 1.\\
\end{cases}.$$
\item  (Case 2: $\omega \in \mathcal W_{sa}$)  For $\omega\in \mathcal W_{sa}$ the return time is
$$u  = \min \{(ab)^{-1}-s,R(a,b,s,\a)\}$$
where $R$ is the return map for affine lattices and $(a,b,s,\a)\in\Omega$. Recall the map is defined piecewise as
$$R(a,b,s,\a) = \begin{cases}
 \frac{sa}{\a-a}, &\text{ if } (a,b,s,\a)\in\Omega_1\\
 \frac{j(a^{-1}-sb)+sa}{\a-a+jb}, &\text{ if } (a,b,s,\a)\in\Omega_2\\
 \frac{a^{-1}-sb}{b+\a}, &\text{ if } (a,b,s,\a)\in\Omega_3\\
 \frac{j(a^{-1}-sb)}{\a-a+jb}, &\text{ if } (a,b,s,\a)\in\Omega_4.
\end{cases}$$
where $j=\left\lfloor \frac{1+a-\alpha}{b}\right\rfloor$. We now compute the return time
$$u  = \min \{(ab)^{-1}-s,R(a,b,s,\a)\}.$$
\begin{enumerate}
\item If $(a,b,s,\a)\in\Omega_1$, then
$$u=\frac{sa}{\a-a}$$
because $$\frac{sa}{\a-a}<(ab)^{-1}-s\iff s<\frac{\a-a}{ab\a}$$ and this inequality is one of the defining ones of $\Omega_1.$
\item If $(a,b,s,\a)\in\Omega_2$, then
$$u=(ab)^{-1}-s$$
because $$(ab)^{-1}-s<\frac{j(a^{-1}-sb)+sa}{\a-a+jb}\iff s>\frac{\a-a}{ab\a}$$ and this inequality is one of the defining ones of $\Omega_2.$
\item If $(a,b,s,\a)\in\Omega_3$, then
$$u=\frac{a^{-1}-sb}{b+\a}$$
because $$\frac{a^{-1}-sb}{b+\a}<(ab)^{-1}-s\iff s<\frac{1}{ab}$$ and this inequality is one of the defining ones of $\Omega.$
\item If $(a,b,s,\a)\in\Omega_4$, then
$$u=(ab)^{-1}-s$$
because $$(ab)^{-1}-s< \frac{j(a^{-1}-sb)}{\a-a+jb}\iff s<\frac{1}{ab}$$ and this inequality is one of the defining ones of $\Omega.$
\end{enumerate}
\end{itemize}
\end{proof}

%%%%%%%%%%%%%%%%%%%%%%%%%%%%%%%%%%%%%%%%%%%%%%%%%%%%%%%%%%%%%%%%%%%%%%%%%%%%%%%%%%%%%%%%%%%%%%%%%%%%%%%%%%%%%%%%%%%%%%%%%%%%%%%%%%%%%%%%%%%%55%%%%%

\section{Invariant measures and gap distributions}
In this section we prove the existence of a density function for the slope gap question for twice marked tori and for doubled slit tori. We will classify measures on the transversal $\Omega$ and show various properties of the cumulative distribution function of the gaps such as support at 0 and the decay rate. We will also consider an invariant measure on the transversal  $\mathcal W$.
%%%%%%%%%%%%%%%%%%%%%%%%%%%%%%%%%%%%%%%%%%%%%
\subsection{Invariant Measures on $\Omega$} In this section we review some Ratner theory and use it classify invariant measures on $\Omega$.
\subsubsection{Ratner Theory for $AX_2$}
In the early 90's Marina Ratner published a series of influential works \cite{MR1075042}, \cite{MR1062971}, and \cite{MR1135878} that gave a mostly complete understanding of ergodic probability measures for unipotent dynamical systems on homogenous spaces. We briefly review her work in the context of horocycle flow on $AX_2$. Our exposition borrows ideas from \cite{MR2582821}, \cite{MR2648693}, and \cite{MR3332893}.

Consider the vertical translation $V_\beta:AX_2\to AX_2$
$$V_\beta(g,v) = \left(g,v +  \left(
  \begin{array}{ c c }
     0 \\
    \beta
  \end{array} \right)\right)$$
for some $\beta\in \R$. Notice that this flow always commutes with $h_s$. Hence, if $\mu$ is an ergodic probability measure invariant under horocycle flow, then so is the measure obtained by pre-composing with $V_\beta$.  This makes it interesting to consider the vertical translation if one is interested in measures invariant under the horocycle flow.

By disintegrating, any measure $\,d\mu (g,v)$ on $AX_2$  can be decomposed into the measure 
$$\,d\mu (g,v) =\,d\lambda_g(v) \,d \nu (g)$$
 where  $\,d \nu (g)$ is a measure  on the lattice part and $\,d\lambda_g(v)$ is a measure on the fiber. Notice that the vertical translation only acts on the fiber piece of an affine lattice $(g,v)$. (That is, it acts on the torus $\R^2/g\Z^2.$)  We can use previous results of \cite{MR0393339} (in the cocompact case) and \cite{MR629475}  (in the general lattice case) to understand possibilities for $\nu (g)$. These results show that either $\nu$ is Haar measure on $X_2$ or supports a periodic measure. Given these two possible measures on the lattice part we can the try to understand how vertical translation acts on the fiber $\R^2/g\Z^2$ to gain a complete understanding of possibilities for $\mu$. This is carried out in \cite{MR2648693}. Before being more precise we make some definitions. 

The \emph{denominator} of an affine lattice $(g,v)$  is the minimal $d>0$ so that $v\in d^{-1}g\Z^2$ whenever this is defined. We denote the denominator by $d(g,v).$ Let $X[q]$ denote the set of affine lattices with denominator $q>0.$ That is,
$$X[q]=\left\{ (g,v)\in AX_2: d(g,v)=q    \right\}.$$

In the literature this is also referred to as the set of \emph{torsion} affine lattices since  the vector $v$ for any $(g,v)\in X[q]$ is a torsion point under the group structure of $\C/g\Z^2$.      The set of torsion affine lattices forms an $SL_2(\R)$-invariant subset of $AX_2$ of dimension 3. In particular, it is a homogenous space and can be identified as the quotient of $SL_2(\R)$ with a certain congruence subgroup depending on $q$. Since it is a homogenous space, then we can think of $X[q]$ as an $SL_2(\R)$ orbit of \emph{any} point in $X[q]$ e.g.

$$X[q]=SL_2(\R)\left(Id, \left(
  \begin{array}{ c c }
     1/q \\
    0
  \end{array} \right)\right).$$

To get an idea of the properties of these types of affine lattices, notice that $X[1]$ is just the space of lattices. Also notice that $\left(Id, \left(
  \begin{array}{ c c }
     1/2 \\
    1/2
  \end{array} \right)\right)$ is $X[2]$, but not in $X[4]$ by our minimality condition on the denominator. 

We now state Ratner's measure classification for horocycle flow on $AX_2$. 
\begin{thm}
If $\mu$ is an $h_u$-invariant ergodic probability measure, then $\mu$ is one of the following:
\begin{itemize}
\item $\mu=m_{AX_2}$, the Haar measure on $AX_2$,
\item $\mu$ is a \emph{torsion measure}. That is there is $q>0$ so that the support of $\mu$ is contained in the set of torsion affine lattices $X[q]$,
\item  there is $a>0$ and $ 0\le \a<a^{-1}$ so that the support of $\mu$  is
$$\text{Per}(a,\a):= \left\{h_sV_\beta \left(g_{\log(a)}, \left(
  \begin{array}{ c c }
     \a \\
    0
  \end{array} \right)\right) = \left(h_sg_{\log(a)}, \left(
  \begin{array}{ c c }
     \a \\
    \beta
  \end{array} \right)\right):s\in[0,a^2),\beta\in[0,a)\right\},$$

\item $\mu$ is supported on a $h_u$-periodic orbit.

\end{itemize}
\end{thm}
The first two measure above are those whose lattice measure is Haar measure on $X_2$, while the remaining two are those for which the lattice measure is supported on a $h_u$-periodic orbit.

Lastly, we observe that any measure in the list above can be precomposed by a vertical translation $V_\beta$ to produce more ergodic probability measures. We will mainly concern ourselves with the above list.

\subsubsection{Invariant measures for $\Omega$}

To classify ergodic probability measures invariant under the return map on the transversal $\Omega$, we use Ratner's theorem which classifies  ergodic probability measures for unipotent dynamical systems on homogenous spaces. Since we have a transversal, then there is a bridge between measures on $(AX_2,h_s)$ and measures on $(\Omega,T)$. While this bridge is well known by now, the interested reader can find details in \cite{MR3214280} .

As a corollary to Ratner's theorem we can classify $T$-invariant measures on the transversal $\Omega$ by disintegration.

\begin{thm}\label{MeasureClassificationOmega}
If $\nu$ is a $T$-invariant ergodic probability measure, then 

\begin{itemize}
\item $\nu$ is induced by the Haar measure on $AX_2$ i.e. $d \nu =  2dsdbdad\a $, 
\item $\nu$ is induced by a torsion measure on $AX_2$ i.e. there is an integer $q>0$ so that the support of $\nu$ is the set of affine lattices 
$$\left \{  \left  (p_{a,b},\left (
  \begin{array}{ c c }
     a/q \\
    0
  \end{array} \right) \right ):(a,b)\in\Delta\right\}.$$
\item $\nu$ is supported on $$\Omega\cap \text{Per}(a,\a)= \left\{ \left(h_sg_{\log(a)}, \left(
  \begin{array}{ c c }
     \a \\
    0
  \end{array} \right)\right):s\in[0,a^2)\right\}$$
for some  $a>0$ and $ 0\le \a<\min\{a^{-1},1\}$,
\item $\nu$ is supported on a $T$-periodic orbit.
\end{itemize}
\end{thm}
\begin{proof}
By our discussion about disintegration in the last section, it is clear that the disintegration of Haar measure on $AX_2$ is given by the product of 2-dimensional Lebesgue measure on the torus $\R^2/g\Z^2$ and the 3-dimensional Haar measure on $X_2.$ This measure is $SL_2(\R)$-invariant and hence horocycle invariant. On $\Omega$, we must consider only those affine lattices with a short horizontal vector so that the disintegration of Haar measure on $\Omega$ will be
$$\,dm_{\Omega} = 2\,ds\,db\,da\,d\a .$$
Now suppose that $\nu$   is a $T$-invariant ergodic probability measure on $\Omega$, different than $m_\Omega$. Consider $\, d\nu \, du$ on $AX_2$. By Ratner's theorem, we know the support of this measure on $AX_2$ can only be 3, 2, or 1 dimensional.

In case it is 3-dimensional, then it must be supported on $X[q]$ for some $q>0$. Hence, $\,d\nu$ is necessarily  2-dimensional and supported on $X[q]\cap\Omega$. We claim that 
$$\Omega\cap X[q] = \left \{  \left  (p_{a,b},\left (
  \begin{array}{ c c }
     a/q \\
    0
  \end{array} \right) \right ):(a,b)\in\Delta\right\}.$$

Indeed, if $(g,v)\in \Omega\cap X[q]\subseteq \Omega$, then it has the form $(g,v)=\left(h_sp_{a,b},\left (
  \begin{array}{ c c }
     \a \\
    0
  \end{array} \right)\right)$ where $(a,b)\in\Delta$, $s\in[0,\frac{1}{ab}) $, and $\a\in(0,1]$. On the other hand, notice that 
$$X[q] = SL_2(\R)\cdot\left(Id,\left(
  \begin{array}{ c c }
     q^{-1} \\
    0
  \end{array} \right)\right) $$ i.e. it is an orbit under $SL_2(\R)$ for any affine piece of denominator $q$. Then, we conclude that
$$(g,v)=\left(h_sp_{a,b},\left (
  \begin{array}{ c c }
     \a \\
    0
  \end{array} \right)  \right)= h_sp_{a,b}\cdot \left(Id,\left (
  \begin{array}{ c c }
     (a^{-1}-sb)\a \\
    sa
  \end{array} \right) \right)$$ so that
$$\left (
  \begin{array}{ c c }
     (a^{-1}-sb)\a \\
    sa
  \end{array} \right) = \left (
  \begin{array}{ c c }
    q^{-1} \\
    0
  \end{array} \right) .$$
The equality of the vertical component shows that $s=0$. Plugging this into the horizontal component shows that $\a =a/q$. This proves one containment. The other containment is obvious.

In case $\,d\nu \, du$ is 2 dimensional, then it is supported on $\text{Per}(a,\a)$ for some  $a>0$ and $ 0\le \a<\min\{a^{-1},1\}$ i.e. those affine lattices for which the lattice part is periodic and the affine piece has arbitrary vertical component.  Hence, $\,d\nu$ is supported on 
$$\Omega\cap \text{Per}(a,\a)= \left\{ \left(h_sg_{\log(a)}, \left(
  \begin{array}{ c c }
     \a \\
    0
  \end{array} \right)\right):s\in[0,a^2)\right\}.$$
  Lastly, if $\,d\nu \, du$ is 1-dimensional, then on $\Omega$ it becomes zero dimensional. This simply means that it is supported on a periodic orbit for $T$.
\end{proof}

We continue to refer to the measure $m_\Omega$ as the Haar measure on $\Omega$ and $m_q$ as a torsion measure. We obtain the following.
\begin{thm}\label{ThmMeasure} The measure $dm_{\Omega} = 2\,ds\,db\,da\,d\a $ is unique $T$-invariant ergodic probability measure on $\Omega$ that is absolutely continuous with respect to Lebesgue measure. Hence, for every $m_\Omega$-a.e. point $(a,b,s,\a)$ and every $f\in L^1(\Omega)$,
$$\frac{1}{N}\sum_{i=0} ^{N-1}f(T^i(a,b,s,\a))\to \int_\Omega f dm_{\Omega}.$$
\end{thm}
In coordinates $(a,b,s,\a)$ integrating with respect to $dm_{\Omega}$ takes the form
$$\int _\Omega f dm_\Omega = \int _{0} ^1 \int_{0} ^1 \int_{1-a} ^ 1 \int _{0} ^{(ab)^{-1}} f(a,b,s,\a)\,ds\,db\,da\,d\a$$
and integrating with respect to one of the torsion measures $m_{q}$ takes the form
$$\int _\Omega f dm_{q} =\int_{(a,b)\in\Delta}f\left(p_{a,b},\left(
  \begin{array}{ c c }
     a/q \\
  0
  \end{array} \right)\right )\,da\,db=  \int_{a=0} ^1 \int_{b=1-a} ^ 1  f(a,b,0,a/q)\,db\,da.$$
%%%%%%%%%%%%%%%%%%%%%
\subsection{An invariant measure on $\mathcal W$}
Recall that $\mathcal H(0,0)=AX_2$ is a homogenous space so that it comes equipped with a Haar measure $m_{AX_2}$. Since $\W$ is a four fold cover of $AX_2$, we can endow it with a natural measure that is four copies of $m_{AX_2}$.  Denote this measure by $m_{\E}.$ This measure is $SL_2(\R)$-invariant and finite.  In particular, this measure is also invariant under the horocycle flow.

 There is a correspondence of measures invariant under a flow and measures invariant under the first return map of that flow to a transversal. Specializing to our situation, there is an induced measure on our transversal $\mathcal W$ which we denote by $m_\mathcal W$ that is invariant under the first return map. 

It can be written as the sum of a Lebesgue measure on $W_{sl}$ and $W_{sa}$ which we denote as $m_{sa}$ and $m_{sl}$. Explicitly, on $\W_{sa}$ we have $\Pi(\omega)= \Lambda_{(a,b,s,\a)}\in \Omega $ and so for $f\in L^1(\W_{sa})$, we have 
$$ \int_{\W_{sa}} f \,dm_{sa}=  4\int_{\Omega} f(a,b,s,\a)\,ds \,db \,da \,d\a.$$
where the 4 reflects the fact that the  $\Pi:\mathcal E\to \mathcal H(0,0)=AX_2$ is a 4-to-1 map.  Similarly, on $\W_{sl}$ we have $\Pi(\omega)= (p_{a,b},v)$  where $(a,b)\in \Delta$ and $v\in \R^2/p_{a,b}\Z^2$ and so for $f\in L^1(\W_{sl})$, we have 
$$ \int_{\W_{sl}} f \,dm_{sl}=  4 \iint\limits_{(a,b)\in\Delta}\int_{v\in \R^2/p_{a,b}\Z^2} f(a,b,v)\,dv \,db \,da.$$
Thus, the flat measure on $\W$ is given by the sum,
$$\int _\W f \,dm_{\W} = \int_{\W_{sl}} f \,dm_{sa} +\int_{\W_{sl}} f \,dm_{sl}.$$

%%%%%%%%%%%%%%%%%%%%%%%%%%%%%%%%%%%%%%%%%%%%%
\subsection{Gap distribution results}
In this section we prove the existence of the slope gap limit for twice marked tori and doubled slit tori and state various properties about their distribution and density functions. 

\begin{proof}(of Theorem \ref{ThmGapTori})

We prove the slope gap limit in Theorem \ref{ThmGapTori} of twice marked tori exists. Recall, from (\ref{slopetoerg}) we have
$$\frac{1}{N}|\mathcal G^N _V(\Lambda)\cap I|=\frac{1}{N}\sum_{i=0} ^{N-1}\chi_{ R^{-1}(I)}(T^i(\Lambda)).$$

Now suppose that $\Lambda$ is $\mu$-generic for some ergodic probablity measure $\mu$. We classified all such measures in Theorem \ref{ThmMeasure}. Then by definition of a $\mu$-generic point, we have that 
$$\frac{1}{N}\sum_{i=0} ^{N-1}\chi_{ R^{-1}(I)}(T^i(\Lambda))\to \mu(R^{-1}(I)).$$

Since the cumulative distribution function is given by the total area bounded by the return time hyperbolas, it is piecewise
real analytic. In particular, it is continuous and differentiable.
The probability density associated to the slope gap is simply the derivative of the cumulative distribution function. For example, for $m_\Omega$-generic points, we have the probability density is given by
$$g(t) := \frac{d}{dt}m_\Omega(R^{-1}(0,t)).$$

Hence, the proportion of gaps in an interval $I$ of an $m_\Omega$-generic affine lattice $\Lambda\in\Omega$ is given by
$$\lim_{N\to\infty}\frac{1}{N}|\mathcal G^N _V(\Lambda)\cap I|=\lim_{N\to\infty}\frac{1}{N}\sum_{i=0} ^{N-1}\chi_{ R^{-1}(I)}(T^i(\Lambda))= m_\Omega(R^{-1}(I))=\int_I g(t) \,dt.$$
Similarly, for $m_q$-generic points, we have the probability density is given by
$$g_q(t) := \frac{d}{dt}m_q(R^{-1}(0,t))$$ and the proportion of gaps given by
$$\lim_{N\to\infty}\frac{1}{N}|\mathcal G^N _V(\Lambda)\cap I|=\int_I g_q(t) \,dt.$$\end{proof}

\begin{proof}(of Theorem \ref{ThmGapSlitV})

A similar argument works to prove the slope gap limit (\ref{SlopeGapSlit}) for doubled slit tori exists for any $m_\mathcal W$-generic point of $\mathcal W$. Indeed,
$$\lim_{N\to\infty}\frac{1}{N}|\mathcal G^N _V(\Lambda_\omega)\cap I|=\lim_{N\to\infty}\frac{1}{N}\sum_{i=0} ^{N-1}\chi_{\mathcal R^{-1}(I)}(\mathcal T^i(\Lambda_\omega)) =m_{\W}(\mathcal R^{-1}(I)).$$
This proves Theorem \ref{ThmGapSlit}.
The probability density associated to the slope gap is simply the derivative of the cumulative distribution function. We denote this density by $f$. That is,
$$f(t) := \frac{d}{dt}m_\W(\mathcal R^{-1}(0,t))$$
where $\mathcal R$ is the first return map to $\W.$
Hence, the proportion of gaps of saddle connections in an interval $I$ of an $m_\W$-generic translation surface $\omega\in\W$ is given by
$$\lim_{N\to\infty}\frac{1}{N}|\mathcal G^N _V(\Lambda_\omega)\cap I|=\int_I f(t) \,dt.$$
\end{proof}

We state more precisely the conclusions of the above theorems now as well. Recall, that $F\sim G$ means that, for large enough $t$, the ratio is a positive constant, $F\ll G$ means there is a positive constant $c$ so that  for large enough $t$, $|F(t)|\le c|G(t)|$,  and that $F\asymp G$ indicates that, for large enough $t$, the ratio is bounded between two positive constants.

\begin{thm}\label{Estimates}\begin{enumerate}

\item  For any affine lattice in $ \Omega$ that is generic for $m_{\Omega}$, we have that the proportion of gaps larger than $t$,
$$\int_t ^\infty g(t)dt$$
satisfies
$$t^{-2}\ll\int_t ^\infty g(t)dt\ll t^{-1}.$$
Here $g(t)$ is the density function for affine lattices from the proof of Theorem \ref{ThmGapTori}.

\item There is support at zero for any affine lattice in $ \Omega$ that is generic for $m_{\Omega}$. That is, for any $\varepsilon>0$,
$$ \int_0 ^\varepsilon g(t)dt>0.$$

\item For  any affine lattice in $ \Omega$ that is generic for a torsion measure $m_q$, we have that the proportion of gaps larger than $t$ has,
$$\int_t ^\infty g_q(t)dt\asymp t^{-2}.$$
Here $g_q(t)$ is the density function for affine lattices from the proof of Theorem \ref{ThmGapTori}.

\item For any doubled slit torus in $\W$ that is generic with respect to the $m_\W$, we have  that the proportion of gaps of saddle connections larger than $t$ decays quadratically. More precisely,
$$\int_t ^\infty f(t)dt\sim t^{-2}.$$
Here $f(t)$ is the density function for doubled slit tori from the proof of Theorem \ref{ThmGapSlit}.

\item There is support at zero for any doubled slit torus in $ \W$ that is generic with respect to the $m_\W$. That is, for any $\varepsilon>0$,
$$ \int_0 ^\varepsilon f(t)dt>0.$$
\end{enumerate}
\end{thm}

\begin{remark}
While the statements in part 1 of Theorem \ref{Estimates}  are for gaps of affine lattices that are already in the transversal $\Omega$ and generic for $m_\Omega$, we can obtain information about the gaps of an affine lattice $\Lambda$ that is generic for $m_{AX_2}$. Let $\Lambda'\in\Omega$ be such that $h_u\Lambda = \Lambda'$ and $u$ is the smallest time for which this occurs. Then,
$$\lim_{R\to\infty}\frac{| \mathcal G^R _V(\Lambda )\cap (t,\infty)|}{N(R)}= \lim_{R\to\infty}\frac{| \mathcal G^R _V(\Lambda' )\cap (t,\infty)|}{N(R)}=\int_t ^\infty g(t)dt.$$
Similar arguments can be used for the rest of Theorem \ref{Estimates} to turn statements about the transversal to statements about the ambient space.
\end{remark}

\subsection{$d$-Symmetric Torus Covers}
We can apply our results of doubled slit tori to a more general class of translation surfaces. Let $d>1$ and consider $d$-copies of a torus $X$ that we label $X_1,\ldots,X_d$. Now glue the copies along a slit and identify opposite sides of the slits according to the permutation $(1,2,\ldots, d)$. We assume that the slit has irrational slope. Following \cite{MR2180242} we call a surface constructed in this way a \emph{$d$-symmetric torus cover}. Notice that a doubled slit torus simply corresponds to this construction with $d=2$. 

Any $d$-symmetric torus cover has two cone points of angle $2d\pi$ and has genus $d$.  Let $\E^d$ denote the class $d$-symmetric torus covers. We have $\E^d\subset \mathcal H(d-1,d-1)$. See Figure \ref{fig: dSlitTorus} for an example of a $d$-symmetric torus cover.

\begin{figure}[H]
	\centering
    \includegraphics[width=4.5in]{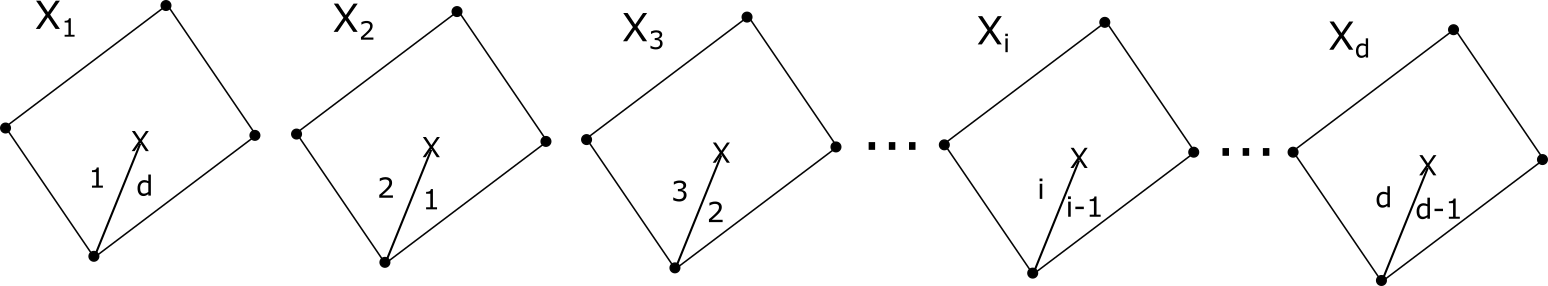}
    \caption{An example of a $d$-symmetric torus cover.} 
    \label{fig: dSlitTorus}
\end{figure}

We remark that we can do a similar construction with any cycle $\sigma\in S^d$ (cycle so that the resulting surface is connected). By relabeling the $d$-copies there is no loss in generality in assuming $(1,2,\ldots, d)$.

The results for doubled slit tori allow us to compute the gap distribution of a typical element in $\E^d$. In particular this yields an explicit gap distribution for a translation surface in any genus $d>1$.

\begin{cor}\label{dSlitTorus}
For $d>1$, the gap distribution of a $d$-symmetric torus cover can be computed using the density function for doubled slit tori. That is,
$$\lim_{R\to\infty}\frac{| \mathcal G^R (\Lambda_\omega) \cap I|}{N(\omega, R)}=\int_I f(x)\,dx$$
for almost every $\omega\in\E^d$ with respect to an  $SL_2(\R)$-invariant probability measure on $\E^d$ that is in the same measure class as Lebesgue measure. Here $f(t)$ is the density function for doubled slit tori from the proof of Theorem \ref{ThmGapSlit}.
\end{cor}

\begin{proof}
Consider a typical element $\omega_1\in\E^d$. We obtain a twice marked torus $(g,v)\in\mathcal H(0,0)$ by projecting to a single copy and forgetting the slit. Use this twice marked torus to construct a doubled slit torus $\omega_2$. Each has the same set of saddle connections ($\Lambda_{\omega_1} = \Lambda_{\omega_2})$ so that their respective gap distributions are the same as well.
\end{proof}

\appendix
\section{Measure computations and estimates}
%%%%%%%%%%%%%Properties of the gap distributio%%%%%%%%%%%%%%%%%%%%%%%%%
%%%%%%%%%%%%%%%%%%%%%%%%%%
In the appendix we compute an explicit limiting gap distribution  for slopes of saddle connections on a doubled slit torus. We also prove estimates on the decay of gaps for affine lattices.

\subsection{Haar measure computations and estimates for $\W$}
In this section we explicitly calculate the gap distribution function given by 
$$m_{\W}(\omega\in \W| \mathcal R(\omega)>t)$$
where $t>0$. As a corollary we compute the decay rate of the tail and show there is support at zero.

We compute the distribution by considering the partition on $\W$ given by surfaces with projection a short affine lattice $\W_{sa}$ and surfaces with projection a short lattice $\W_{sl}$. On either part of the partition, we will have the return time map $\mathcal R$ explicitly. The basic idea of proof will be to use the condition of there being a gap larger than $t$ (i.e that the return time $\mathcal R$ is larger than $t$) and see what conditions are imposed on the coordinates where the return time is defined. This will yield explicit integrals whence we can compute the integrals and make arguments with Taylor series to show their decay. 

\subsubsection{Haar measure computations and estimates for $\W_{sa}$}

We compute the contribution on the doubled slit tori coming from a short affine piece. That is,
$$m_{\mathcal W}(\{\omega\in \W_{sa}| \mathcal R(\omega)>t\}) .$$

The return time map $\mathcal R$ is a piece-wise map on three pieces depending on $\Pi(\omega)$. In each piece, we know explicitly the return map and so we can compute the contribution.
%%%%%%%%%%%%%%%%%%%%%%%%%%%%%%%%%%%%%

 \textbf{Contribution when $\Pi(\omega)\in \Omega_1.$}

Suppose $\Pi(\omega)\in \Omega_1.$ Notice that $\Omega_1$ can be written as
$$\Omega_1 = \left\{(a,b,s,\a)\in \Omega: b\in(0,1], \a\in(1-b,1], a\in(1-b,\a], s\in\left[0,\frac{\a-a}{ab\a}\right)\right\}$$
We have,
$$m_{\W}(\{\omega\in \W_{sa}| \mathcal R(\omega)>t, \Pi(\omega)\in \Omega_1\}) = 4 m_{\Omega}(\{(a,b,s,\a)\in\Omega_1:R(a,b,s,\a)>t\}) .$$
Thus, we now compute $m_{\Omega}(\Omega_1\cap\{R>t\}) .$

The condition on the return time says
 $$t<R(a,b,s,\alpha) = \frac{sa}{\a-a}\iff \frac{t(\a-a)}{a}<s.$$
Notice that this is valid lower bound for $s$ since $s$ is non-negative. On $\Omega_1$, $s$ has a upper bound of $\frac{\a-a}{ab\a}$ so in order for the return map $R$ to be larger than $t$ we must also satisfy  the cross-term condition $ \frac{t(\a-a)}{a}<s< \frac{(\a-a)}{ab\a}\iff t<\frac{1}{b\a}$. We rewrite this as $\a<\frac{1}{bt}$. Hence, we  all we need to understand is how the hyperbola $\a=1/bt$ intersects the $(b,\a)$-triangle  $\{(b,\a): b\in(0,1], \a\in(1-b,1]\} )$. See Figure \ref{fig: ContributionOmega1} for an illustration for the $(b,a)$ slice of the contribution of $\Omega_1$ when $t>4$.\begin{figure}[H]
	\centering
    \includegraphics[width=4in]{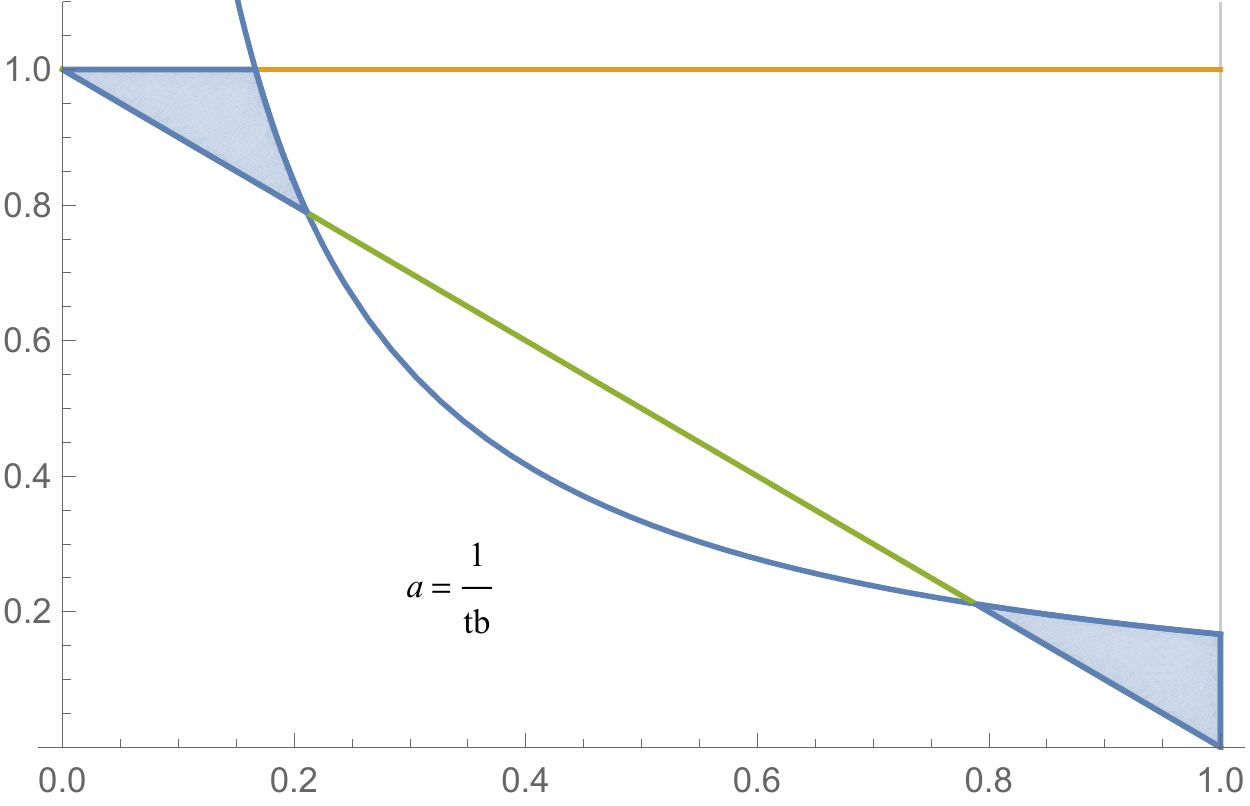}
    \caption{The contribution on $\Omega_1$ when $t>4$.} % needs citation
    \label{fig: ContributionOmega1}
\end{figure}

If $t\in(0,1)$, then everything lies below the hyperbola $\a=\frac{1}{bt}$ and our contribution is given by 
$$m_{\Omega}(\{R>t\}\cap \Omega_1)=\int _{b=0} ^{1} \int_{\a=1-b} ^{1} \int_{a=1-b} ^ \a \int _{s=\frac{t(\a-a)}{a}} ^{\frac{\a-a}{ab\a}} 1\,ds\,da\,d\a\,db$$
If $t\in[1,4)$, then 
\begin{align*}
m_{\Omega}(\{R>t\}\cap \Omega_1)&= \int _{b=0} ^{1/t} \int_{\a=1-b} ^{1} \int_{a=1-b} ^ \a \int _{s=\frac{t(\a-a)}{a}} ^{\frac{\a-a}{ab\a}} 1\,ds\,da\,d\a\,db\\
&+ \int _{b=1/t} ^{1} \int_{\a=1-b} ^{1/bt} \int_{a=1-b} ^ \a \int _{s=\frac{t(\a-a)}{a}} ^{\frac{\a-a}{ab\a}} 1\,ds\,da\,d\a\,db.
\end{align*}
If $t\in[4,\infty)$, then
\begin{align*}
m_{\Omega}(\{R>t\}\cap \Omega_1)& =\int_{b=0}^{1/t}\int_{\a=1-b}^{1}\int_{a=1-b}^{\a}\int _{s=\frac{t(\a-a)}{a}} ^{\frac{\a-a}{ab\a}}1\,ds\,da\,d\a\,db  \\
&+ \int _{b=1/t} ^{\frac{1-\sqrt{1-(4/t)}}{2}} \int_{\a=1-b} ^{1/bt} \int_{a=1-b} ^ \a \int _{s=\frac{t(\a-a)}{a}} ^{\frac{\a-a}{ab\a}} 1\,ds\,da\,d\a\,db\\
&+ \int _{\frac{1+\sqrt{1-(4/t)}}{2}}^1 \int_{\a=1-b} ^{1/bt} \int_{a=1-b} ^ \a \int _{s=\frac{t(\a-a)}{a}} ^{\frac{\a-a}{ab\a}} 1\,ds\,da\,d\a\,db.
\end{align*}
By evaluating each integral and then approximating the result by its Taylor series, we can understand how it decays. As an example, we do this for the first integral.
\begin{align*}
&\int_{b=0}^{1/t}\int_{\a=1-b}^{1}\int_{a=1-b}^{\a}\int _{s=\frac{t(\a-a)}{a}} ^{\frac{\a-a}{ab\a}}1\,ds\,da\,d\a \,db\\
&= \int_{b=0}^{1/t}\frac{ b(b (2 + b) t-8) + 2 ( b (2 + t)-4) \log(1 - b)}{4b}\,db\\
&\approx \int_{b=0}^{1/t}\frac{ b( b (2 + b) t-8) + 2 ( b (2 + t)-4)(-b -\frac{ b^2}{2} -\frac{ b^3}{3})}{4b}\,db \\
&=- \frac{1}{6}  \int_{b=0}^{1/t}b^2 ( b (2 + t)-1)\,db = -\frac{1}{12 t^4} + \frac{1}{72 t^3}
\end{align*}
Hence, on this piece we have cubic decay. Doing this for the rest of the integrals we have
\begin{align*}
& \int_{b=0}^{1/t}\int_{\a=1-b}^{1}\int_{a=1-b}^{\a}\int _{s=\frac{t(\a-a)}{a}} ^{\frac{\a-a}{ab\a}}1\,ds\,da\,d\a \,db\sim t^{-3},  \\
& \int _{b=1/t} ^{\frac{1-\sqrt{1-(4/t)}}{2}} \int_{\a=1-b} ^{1/bt} \int_{a=1-b} ^ \a \int _{s=\frac{t(\a-a)}{a}} ^{\frac{\a-a}{ab\a}} 1\,ds\,da\,d\a \,db\sim t^{-5},\\
& \int _{\frac{1+\sqrt{1-(4/t)}}{2}}^1 \int_{\a=1-b} ^{1/bt} \int_{a=1-b} ^ \a \int _{s=\frac{t(\a-a)}{a}} ^{\frac{\a-a}{ab\a}} 1\,ds\,da\,d\a \,db\sim t^{-2}.
\end{align*}
Hence, the decay rate on the tail of $\Omega_1$ is quadratic. 
By multiplying each integral above by 4 we have that contribution of 
$$m_{\mathcal W}(\{\omega\in \W_{sa}| \mathcal R(\omega)>t,\Pi(\omega)\in\Omega_1\}) $$
and also the decay rate
$$m_{\mathcal W}(\{\omega\in \W_{sa}| \mathcal R(\omega)>t,\Pi(\omega)\in\Omega_1\}) \sim t^{-2}.$$
%%%%%%%%%%%%%%%%%%%%%%%%%%%%
 \textbf{Contribution when $\Pi(\omega)\in \Omega_3.$}

Suppose $\Pi(\omega)\in \Omega_3.$ Notice that $\Omega_3$ can be written as
$$\Omega_3 = \left\{(a,b,s,\a)\in \Omega: b\in(0,1], \a\in(0,1-b], a\in(1-b,1], s\in\left[0,\frac{1}{ab}\right)\right\}.$$
We have,
$$m_{\W}(\{\omega\in \W_{sa}| \mathcal R(\omega)>t, \Pi(\omega)\in \Omega_3\}) = 4 m_{\Omega}(\{(a,b,s,\a)\in\Omega_3:R(a,b,s,\a)>t\}) .$$
Thus, we now compute $m_{\Omega}(\Omega_3\cap\{R>t\}) .$

On $\Omega_3$ we have
$$t<R(a,b,s,\alpha) =\frac{a^{-1}-sb}{b+\a}\iff \frac{a^{-1}-t(b+\a)}{b}>s.$$
Notice that the left-hand side is always bounded above by $(ab)^{-1}.$ In order for there to be any valid values we need the left-hand side to be positive which is equivalent to
$$a<\frac{1}{t(b+\a)}.$$
Hence, we seek the minimum between $\frac{1}{t(b+\a)}$ and $a$. 
\begin{itemize}
\item\textbf{ Case I}  $a<\min\left\{\frac{1}{t(b+\a)},1\right\} =1$. This happens when
$$\frac{1}{t(b+\a)}>1 \iff \a<\frac{1}{t}-b.$$
Notice now that $\a<1-b$ on $\Omega_3$  and that 
$$\a<\min\{1/t-b,1-b\}=\begin{cases}
1-b, &\text{ if } t<1\\
  1/t-b, &\text{ if } t>1.
\end{cases}$$
Hence, if $t<1$, the integral is
$$\int_{b=0} ^{1} \int_{\a=0} ^{1-b}\int_{a=1-b} ^1 \int _{s=0} ^{\frac{a^{-1}-t(b+\a)}{b}}1dsdad\a db = \frac{\pi^2}{6} -1 - \frac{ t}{3}.$$
On the other hand, if $t>1$, then $\a<1/t-b$ and since $\a>0$, we must have $1/t>b$. Thus, for large  $t$ (larger than 1), the integral is
$$\int_{b=0} ^{1/t} \int_{\a=0} ^{1/t-b}\int_{a=1-b} ^1 \int _{s=0} ^{\frac{a^{-1}-t(b+\a)}{b}}1\,ds\,da\,d\a \,db.$$
A Taylor series argument shows that this decays quadratically.

\item \textbf{ Case II} $a<\min\left\{\frac{1}{t(b+\a)},1\right\} = \frac{1}{t(b+\a)}$. This happens when
$$\frac{1}{t(b+\a)}<1 \iff \a>\frac{1}{t}-b.$$ 
Since $1-b<\a$, then combining the inequalities we obtain 
$$\a<\frac{b^2-b+1/t}{1-b}.$$

By understanding how these functions intersects the $(b,\a)$-plane $$\{(b,\a)|0<b\le1\text{ and }0<\a<1-b\},$$
then we can understand how this function grows in $t$.

We obtain the following:
If $t\in (0,1)$, there is no contribution.

If $t\in[1,2)$, the contribution is 
\begin{align*}
m_{\Omega}(\{R>t\}\cap \Omega_3)=&\int_{b=0} ^{1-1/t} \int_{\a=1/t - b} ^{\frac{b^2-b+1/t}{1-b}} \int_{a=1-b} ^{1/(t(b+\a))} \int _{s=0} ^{\frac{a^{-1}-t(b+\a)}{b}} 1\,ds\,da\,d\a\,db\\
&\int_{b=1-1/t} ^{1/t} \int_{\a=1/t - b} ^{1-b} \int_{a=1-b} ^{1/(t(b+\a))} \int _{s=0} ^{\frac{a^{-1}-t(b+\a)}{b}} 1\,ds\,da\,d\a\,db\\
&\int_{b=1/t} ^{1} \int_{\a=0} ^{1-b} \int_{a=1-b} ^{1/(t(b+\a))} \int _{s=0} ^{\frac{a^{-1}-t(b+\a)}{b}} 1\,ds\,da\,d\a\,db.
\end{align*}
If $t\in[2,4)$, the contribution is
\begin{align*}
m_{\Omega}(\{R>t\}\cap \Omega_3)=&\int_{b=0} ^{1/t} \int_{\a=1/t - b} ^{\frac{b^2-b+1/t}{1-b}} \int_{a=1-b} ^{1/(t(b+\a))} \int _{s=0} ^{\frac{a^{-1}-t(b+\a)}{b}} 1\,ds\,da\,d\a\,db\\
&\int_{b=1/t} ^{1-1/t} \int_{\a=0} ^{\frac{b^2-b+1/t}{1-b}} \int_{a=1-b} ^{1/(t(b+\a))} \int _{s=0} ^{\frac{a^{-1}-t(b+\a)}{b}} 1\,ds\,da\,d\a\,db\\
&\int_{b=1-1/t} ^{1} \int_{\a=0} ^{1-b} \int_{a=1-b} ^{1/(t(b+\a))} \int _{s=0} ^{\frac{a^{-1}-t(b+\a)}{b}} 1\,ds\,da\,d\a\,db.
\end{align*} 
If $t>4$ the tail on this piece of $\Omega_3$ is given by
\begin{align*}
m_{\Omega}(\{R>t\}\cap \Omega_3)=&\int_{b=0} ^{1/t}\int_{\a=t^{-1}-b} ^{\frac{b^2-b+1/t}{1-b}}\int_{a=1-b} ^{1/t(b+\a)}\int_{s=0} ^{\frac{a^{-1}-t(b+\a)}{b}}1\,ds\,da\,d\a \,db \\
&+ \int_{b=1/t} ^{\frac{1-\sqrt{1-4/t}}{2}}\int_{\a=0} ^{\frac{b^2-b+1/t}{1-b}}\int_{a=1-b} ^{1/t(b+\a)}\int_{s=0} ^{\frac{a^{-1}-t(b+\a)}{b}}1\,ds\,da\,d\a\, db\\
&+\int_{b=\frac{1+\sqrt{1-4/t}}{2}} ^{1-1/t} \int_{\a=0} ^\frac{b^2-b+1/t}{1-b}\int_{a=1-b} ^{1/t(b+\a)}\int_{s=0} ^{\frac{a^{-1}-t(b+\a)}{b}}1\,ds\,da\,d\a \,db\\
&+\int_{b=1-1/t} ^{1} \int_{\a=0} ^{1-b} \int_{a=1-b} ^{1/t(b+\a)}\int_{s=0} ^{\frac{a^{-1}-t(b+\a)}{b}}1\,ds\,da\,d\a \,db
\end{align*}
By evaluating each integral and then approximating by its Taylor series, we see that 
\begin{align*}
&\int_{b=0} ^{1/t}\int_{\a=t^{-1}-b} ^{\frac{b^2-b+1/t}{1-b}}\int_{a=1-b} ^{1/t(b+\a)}\int_{s=0} ^{\frac{a^{-1}-t(b+\a)}{b}}1\,ds\,da\,d\a \,db \sim t^{-4},\\
& \int_{b=1/t} ^{\frac{1-\sqrt{1-4/t}}{2}}\int_{\a=0} ^{\frac{b^2-b+1/t}{1-b}}\int_{a=1-b} ^{1/t(b+\a)}\int_{s=0} ^{\frac{a^{-1}-t(b+\a)}{b}}1\,ds\,da\,d\a\, db\sim t^{-5},\\
&\int_{b=\frac{1+\sqrt{1-4/t}}{2}} ^{1-1/t} \int_{\a=0} ^\frac{b^2-b+1/t}{1-b}\int_{a=1-b} ^{1/t(b+\a)}\int_{s=0} ^{\frac{a^{-1}-t(b+\a)}{b}}1\,ds\,da\,d\a \,db\sim t^{-5},\\
&\int_{b=1-1/t} ^{1} \int_{\a=0} ^{1-b} \int_{a=1-b} ^{1/t(b+\a)}\int_{s=0} ^{\frac{a^{-1}-t(b+\a)}{b}}1\,ds\,da\,d\a \,db\sim t^{-2}
\end{align*}
\end{itemize}
By adding the contributions from both cases, we conclude that
$$m_{\Omega}(\{R>t\}\cap \Omega_3)\sim t^{-2}.$$
By multiplying each integral above by 4 we have that contribution of 
$$m_{\mathcal W}(\{\omega\in \W_{sa}| \mathcal R(\omega)>t,\Pi(\omega)\in\Omega_3\}) $$
and also the decay rate
$$m_{\mathcal W}(\{\omega\in \W_{sa}| \mathcal R(\omega)>t,\Pi(\omega)\in\Omega_3\}) \sim t^{-2}.$$
%%%%%%%%%%%%%%%%%%%%%%%%%%%%%%%%%%%%%%
\textbf{Contribution when $\Pi(\omega)\in \Omega_2 \cup \Omega_4.$}

 Suppose $\Pi(\omega)\in \Omega_2$.   Notice that $\Omega_2$ can be written as
$$\Omega_2=\left\{b\in(0,1],\a\in(1-b,1],a\in(1-b,\a),s\in\left[(ab)^{-1}\left(\frac{\a-a}{\a}\right),(ab)^{-1}\right)\right\}.$$
We have,
$$m_{\W}(\{\omega\in \W_{sa}| \mathcal R(\omega)>t, \Pi(\omega)\in \Omega_2\}) = 4 m_{\Omega}(\{(a,b,s,\a)\in\Omega_2: (ab)^{-1}-s>t\}) .$$

Unpacking a little we see
$$ (ab)^{-1}-s>t \iff  (ab)^{-1}-t>s$$
and this is always the minimum upper bound for $s$ in $\Omega_2$ since $t>0$. On the other hand, to satisfy the cross term condition, we need
$$(ab)^{-1}\left(\frac{\a-a}{\a}\right)<(ab)^{-1}-t \iff\a<1/bt.$$
Hence, we want to understand how the hyperbola $\a=1/bt$ intersects the triangle $$\left\{(b,\a)|b\in(0,1],\a\in(1-b,1]\right\}.$$ 

If $t\in[0,1)$, then everything lies below the hyperbola $\a=\frac{1}{bt}$ and the contribution is
$$m_{\Omega}(\{\mathcal R>t\}\cap\Omega_2)=\int _{b=0} ^{1} \int_{\a=1-b} ^{1} \int_{a=1-b} ^ \a \int _{s=(ab)^{-1}\left(\frac{\a-a}{\a}\right)} ^{(ab)^{-1}-t} 1\,ds\,da\,d\a\,db$$

If $t\in[1,4)$, then  the hyperbola $\a=\frac{1}{bt}$ intersects once at $b=1/t$ and the contribution is
\begin{align*}
m_{\Omega}(\{\mathcal R>t\}\cap\Omega_2)&= \int _{b=0} ^{1/t} \int_{\a=1-b} ^{1} \int_{a=1-b} ^ \a \int _{s=(ab)^{-1}\left(\frac{\a-a}{\a}\right)} ^{(ab)^{-1}-t} 1\,ds\,da\,d\a\,db \\
&+ \int _{b=1/t} ^{1} \int_{\a=1-b} ^{1/bt} \int_{a=1-b} ^ \a \int _{s=(ab)^{-1}\left(\frac{\a-a}{\a}\right)} ^{(ab)^{-1}-t} 1\,ds\,da\,d\a \,db.
\end{align*}
If $t\in[4,\infty)$, then  the hyperbola $\a=\frac{1}{bt}$  intersects twice and the contribution is
\begin{align*}
m_{\Omega}(\{\mathcal R>t\}\cap\Omega_2)& =\int_{b=0}^{1/t}\int_{\a=1-b}^{1}\int_{a=1-b}^{\a}\int _{s=(ab)^{-1}\left(\frac{\a-a}{\a}\right)} ^{(ab)^{-1}-t} 1\,ds\,da\,d\a \,db  \\
&+ \int _{b=1/t} ^{\frac{1-\sqrt{1-(4/t)}}{2}} \int_{\a=1-b} ^{1/bt} \int_{a=1-b} ^ \a \int _{s=(ab)^{-1}\left(\frac{\a-a}{\a}\right)} ^{(ab)^{-1}-t} 1\,ds\,da\,d\a \,db\\
&+ \int _{\frac{1+\sqrt{1-(4/t)}}{2}}^1 \int_{\a=1-b} ^{1/bt} \int_{a=1-b} ^ \a \int _{s=(ab)^{-1}\left(\frac{\a-a}{\a}\right)} ^{(ab)^{-1}-t} 1\,ds\,da\,d\a \,db.
\end{align*}
By evaluating each integral and approximating the result with a Taylor series, we see that
\begin{align*}
& \int_{b=0}^{1/t}\int_{\a=1-b}^{1}\int_{a=1-b}^{\a}\int _{s=(ab)^{-1}\left(\frac{\a-a}{\a}\right)} ^{(ab)^{-1}-t} 1\,ds\,da\,d\a \,db\sim t^{-2},  \\
&\int _{b=1/t} ^{\frac{1-\sqrt{1-(4/t)}}{2}} \int_{\a=1-b} ^{1/bt} \int_{a=1-b} ^ \a \int _{s=(ab)^{-1}\left(\frac{\a-a}{\a}\right)} ^{(ab)^{-1}-t} 1\,ds\,da\,d\a \,db\sim t^{-4},\\
& \int _{\frac{1+\sqrt{1-(4/t)}}{2}}^1 \int_{\a=1-b} ^{1/bt} \int_{a=1-b} ^ \a \int _{s=(ab)^{-1}\left(\frac{\a-a}{\a}\right)} ^{(ab)^{-1}-t} 1\,ds\,da\,d\a \,db\sim t^{-2}.
\end{align*}

In total, we have a quadratic tail when $\Pi(\omega\in\Omega_2)$
$$m_{\Omega}(\{\mathcal R>t\}\cap\Omega_2)\sim t^{-2}.$$

Now suppose $\Pi(\omega)\in \Omega_4.$ Notice that $\Omega_4$ can be written as
$$\Omega_4=\left\{b\in(0,1],a\in(1-b,1],\a\in(1-b,a),s\in[0,(ab)^{-1})\right\}.$$
We have,
$$m_{\W}(\{\omega\in \W_{sa}|\mathcal  R(\omega)>t, \Pi(\omega)\in \Omega_4\}) = 4 m_{\Omega}(\{(a,b,s,\a)\in\Omega_4: (ab)^{-1}-s>t\}) $$

The condition on the return time says
$$ (ab)^{-1}-s>t \iff  (ab)^{-1}-t>s$$
and this is always the minimum upper bound for $s$ in $\Omega_4$ since $t>0$. On the other hand, to satisfy the cross term condition, we need
$$0<(ab)^{-1}-t \iff a<1/bt.$$
Hence, we are interested in the hyperbola $a =1/bt$ and how it intersects the triangle  $\left\{b\in(0,1],a\in(1-b,1]\right\}.$

If $t\in[0,1)$, then everything lies below the hyperbola $\a=\frac{1}{bt}$ and the contribution is
$$m_{\Omega}(\{\mathcal R>t\}\cap\Omega_4)=\int _{b=0} ^{1} \int_{a=1-b} ^{1} \int_{\a=1-b} ^ a \int _{s=0} ^{(ab)^{-1}-t} 1\,ds\,da\,d\a\,db$$

If $t\in[1,4)$, then  the hyperbola $\a=\frac{1}{bt}$ intersects once at $b=1/t$ and the integral is
\begin{align*}
m_{\Omega}(\{\mathcal R>t\}\cap\Omega_4)&= \int _{b=0} ^{1/t} \int_{a=1-b} ^{1} \int_{\a=1-b} ^ a \int _{s=0} ^{(ab)^{-1}-t} 1\,ds\,da\,d\a\,db \\
&+ \int _{b=1/t} ^{1} \int_{a=1-b} ^{1/bt} \int_{\a=1-b} ^ a \int _{s=0} ^{(ab)^{-1}-t} 1\,ds\,da\,d\a\,db .
\end{align*}

If $t\in[4,\infty)$, then the hyperbola $\a=\frac{1}{bt}$  intersects twice and the integral is
\begin{align*}
m_{\Omega}(\{\mathcal R>t\}\cap\Omega_4)& =\int_{b=0}^{1/t}\int_{a=1-b}^{1} \int_{\a=1-b} ^ a \int _{s=0} ^{(ab)^{-1}-t} 1\,ds\,da\,d\a\,db   \\
&+ \int _{b=1/t} ^{\frac{1-\sqrt{1-(4/t)}}{2}} \int_{a=1-b} ^{1/bt} \int_{\a=1-b} ^ a \int _{s=0} ^{(ab)^{-1}-t} 1\,ds\,da\,d\a\,db \\
&+ \int _{\frac{1+\sqrt{1-(4/t)}}{2}}^1 \int_{a=1-b} ^{1/bt} \int_{\a=1-b} ^ a \int _{s=0} ^{(ab)^{-1}-t} 1\,ds\,da\,d\a\,db .
\end{align*}
By evaluating each integral and approximating the result with a Taylor series, we see that
\begin{align*}
& \int_{b=0}^{1/t}\int_{a=1-b}^{1} \int_{\a=1-b} ^ a \int _{s=0} ^{(ab)^{-1}-t} 1\,ds\,da\,d\a\,db   \sim t^{-2}\\
& \int _{b=1/t} ^{\frac{1-\sqrt{1-(4/t)}}{2}} \int_{a=1-b} ^{1/bt} \int_{\a=1-b} ^ a \int _{s=0} ^{(ab)^{-1}-t} 1\,ds\,da\,d\a\,db \sim t^{-4}\\
&\int _{\frac{1+\sqrt{1-(4/t)}}{2}}^1 \int_{a=1-b} ^{1/bt} \int_{\a=1-b} ^ a \int _{s=0} ^{(ab)^{-1}-t} 1\,ds\,da\,d\a\,db \sim t^{-2}.
\end{align*}

$$m_{\Omega}(\{\mathcal R>t\}\cap\Omega_4)\sim t^{-2}.$$

\textbf{Tail contribution of $\W_{sa} $}

Hence, combining all the above contributions we have that the tail on $\W_{sa}$ has quadratic decay,
$$m_{\W}(\{\omega\in \W_{sa}:\mathcal R(\omega)>t\})\sim t^{-2}.$$
%When accounting for normalizations in the measures we see that the constant above is $16/3.$

\subsubsection{Haar measure computations and estimates for $\W_{sl}$}

We now compute the contribution on the doubled slit tori coming from a short lattice piece. That is,
$$m_{\mathcal W}(\{\omega\in \W_{sl}| \mathcal R(\omega)>t\}) .$$
for  $\omega=\omega_{(a,b,v_1,v_2)}\in\W_{sl}$,  we have a piece-wise map on two pieces depending on the length of $b+v_1$. In each piece, we know explicitly the return map and thus can compute 
$$m_{\W}(\omega\in \W_{sl}|\mathcal  R(\omega)>t)= 4 \iint\limits_{(a,b)\in\Delta}\int_{v\in \R^2/p_{a,b}\Z^2}\chi_{\{ \mathcal R>t\}}(a,b,v)\,dv \,db \,da.$$

It will be useful to parameterize $\R^2/p_{a,b}\Z^2$ as the set
$$\{(v_1,v_2)=(ax+by,a^{-1}y)|(a,b)\in\Delta,(x,y)\in[0,1)^2\}.$$ 
The condition that determines what the return map is depends on whether $b+v_1$ is larger than or smaller than 1. In the coordinates $(v_1,v_2)=(ax+by,a^{-1}y)$ we have
$$b+v_1 >1\iff x\ge \frac{1-b}{a}-\frac{b}{a}y=:l(y).$$
Hence, there is natural interest in the behavior of $l$. Notice
\begin{itemize}
\item $l(0)=\frac{1-b}{a} $ is always less than one for any $(a,b)\in \Delta.$
\item and $l(1)<0\iff b>1/2$. 
\end{itemize}

Hence we can represent the two cases of our return map through the $(y,x)$ slice as in Figure \ref{fig:2Cases}.
\begin{figure}[H]
\centering
\begin{subfigure}{.5\textwidth}
  \centering
  \includegraphics[width=.6\linewidth]{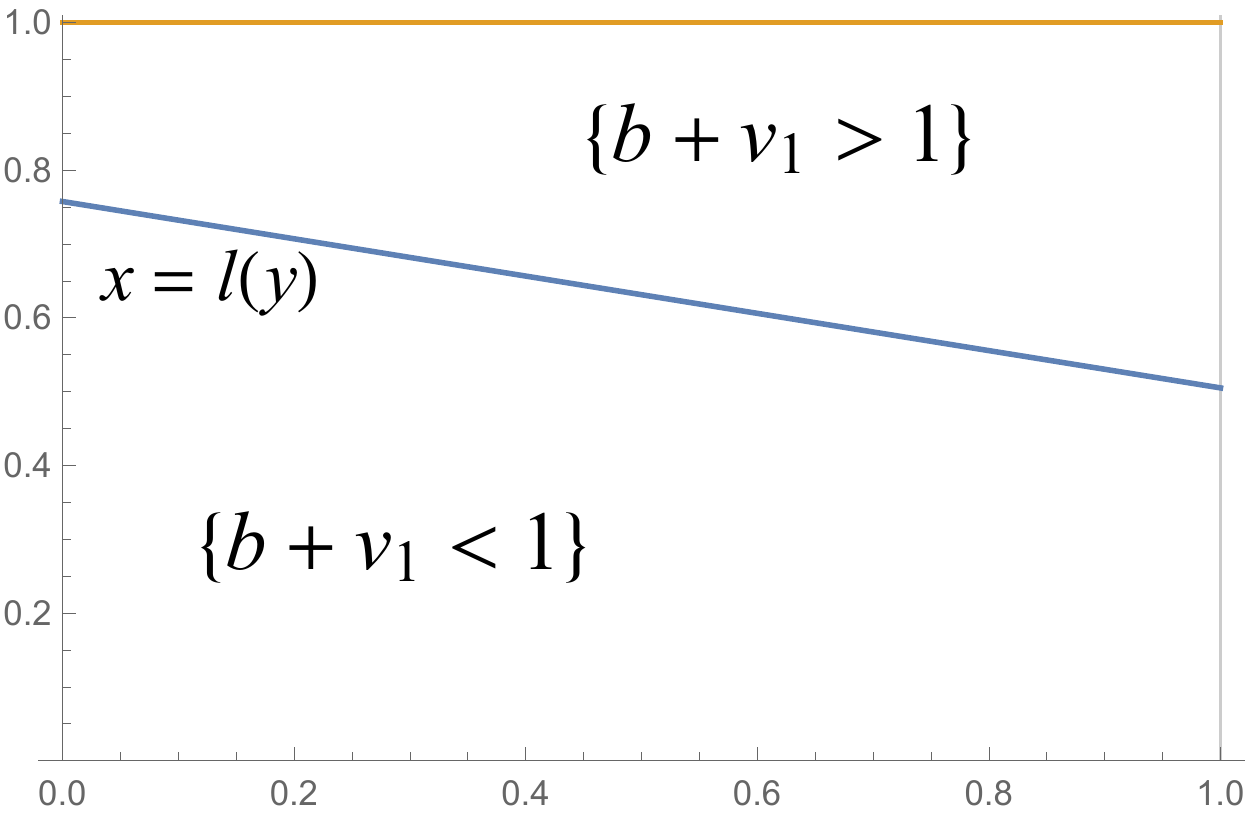}
  \caption{A generic picture when $b<\frac{1}{2}$}
  \label{fig:sub1}
\end{subfigure}%
\begin{subfigure}{.5\textwidth}
  \centering
  \includegraphics[width=.6\linewidth]{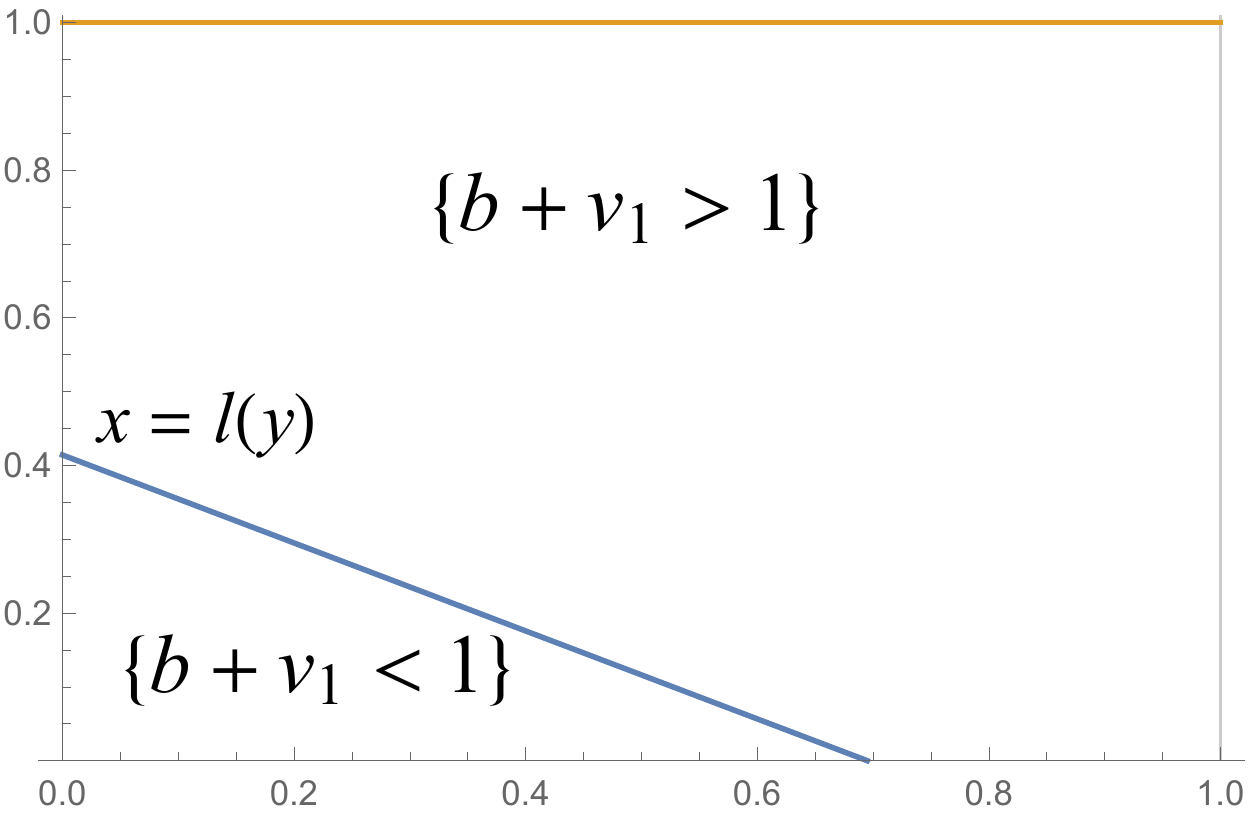}
  \caption{A generic picture when $b>\frac{1}{2}$}
  \label{fig:sub2}
\end{subfigure}
\caption{A generic picture of the two cases that arise for the return map}
\label{fig:2Cases}
\end{figure}

\textbf{ (Case 1: $b+v_1 >1$)} 

Notice that $l(1)<0\iff b>1/2$.  Moreover, $l$ hits the $y$-axis at $y=b^{-1}-1$
As indicated at the start of this subsection we break our computation up into the case where $b>1/2$ and $b<1/2$. Notice that $l$ hits the $y$-axis at $y=b^{-1}-1$.

\textbf{ (Case 1A: $b+v_1 >1$ and $b<1/2$)} 

If $b<1/2$, then our region $b+v_1>1$ can be described as
$$\{(b,a,x,y)|b\in(0,1/2],a\in(1-b,1),y\in[0,1),x\in[l(y),1)\}.$$

See Figure \ref{fig:2Cases} (A).

On this region, 
$$\mathcal R(\omega)=(ab)^{-1}>t\iff a<\frac{1}{bt}.$$ 
Hence, all we have to do is understand how this hyperbola intersects the triangle $\{(b,a)\in \Delta:b\in(1,1/2),a\in(1-b,1)\}$.

If $t\in[0,2)$, the everything lies below the hyperbola and the contribution is given by
$$ \int_{b=0} ^{1/2} \int_{a=1-b} ^1 \int _{y=0} ^{1} \int_{x=l(y)} ^1 1\,dx\,dy\,da\,db.$$

If $t\in[2,4)$, then the contribution is given by
$$\int_{b=0} ^{1/t} \int_{a=1-b} ^1  \int _{y=0} ^{1} \int_{x=l(y)} ^1 1\,dx\,dy\,da\,db 
+\int_{b=1/t} ^{1/2} \int_{a=1-b} ^{1/bt}  \int _{y=0} ^{1} \int_{x=l(y)} ^1 1\,dx\,dy\,da\,db$$

If $t\in[4,\infty)$, then the contribution is given by
$$\int_{b=0} ^{1/t} \int_{a=1-b} ^{1}  \int _{y=0} ^{1} \int_{x=l(y)} ^1 1\,dx\,dy\,da\,db +\int_{b=1/t} ^{\frac{1-\sqrt{1-(4/t)}}{2}} \int_{a=1-b} ^{1/bt}  \int _{y=0} ^{1} \int_{x=l(y)} ^1 1\,dx\,dy\,da\,db
$$

By evaluating each integral and approximating the result with a Taylor series we see that
\begin{align*}
&\int_{b=0} ^{1/t} \int_{a=1-b} ^{1}  \int _{y=0} ^{1} \int_{x=l(y)} ^1 1\,dx\,dy\,da\,db \sim t^{-3},\\
& \int_{b=1/t} ^{\frac{1-\sqrt{1-(4/t)}}{2}} \int_{a=1-b} ^{1/bt}  \int _{y=0} ^{1} \int_{x=l(y)} ^1 1\,dx\,dy\,da\,db\sim t^{-4}
\end{align*}
Hence, the tail on this region is less than quadratic. More precisely, we have
$$m_{\W}\left(\{\mathcal R>t\}\cap\{b+v_1>1\}\cap\{b<1/2\}\right) \sim t^{-3}.$$
%with constant $5/3$.

\textbf{ (Case 1B: $b+v_1 >1$ and $b>1/2$)} 

If $b>1/2$, then our region is the union of
$$\{(b,a,x,y)|b\in(1/2,1),a\in(1-b,1),y\in[0,b^{-1}-1),x\in[l(y),1)\}$$
with
$$ \{(b,a,x,y)|b\in(1/2,1),a\in(1-b,1),y\in[b^{-1}-1,1),x\in[0,1)\}.$$
See Figure \ref{fig:2Cases} (B). 

Now we can implement the return time condition. On this region, 
$$R(\omega)=(ab)^{-1}>t\iff (bt)^{-1}>a.$$ 
Hence, all we have to do is understand how this hyperbola intersects the triangle $\{(b,a)\in\Delta|b\in(1/2,1),a\in(1-b,1)\}$.

If $t\in(0,1)$, the everything lies below the hyperbola and the integral is 
$$\int_{b=1/2} ^1 \int_{a=1-b} ^1 \left[\int _{y=0} ^{b^{-1}-1} \int_{x=l(y)} ^1 1\,dx\,dy + \int _{y=b^{-1}-1} ^{1} \int_{x=0} ^1 1\,dx\,dy     \right] \,da\,db$$
If $t\in[1,2)$, then the integral is 
\begin{align*}
&\int_{b=1/2} ^{1/t} \int_{a=1-b} ^1 \left[\int _{y=0} ^{b^{-1}-1} \int_{x=l(y)} ^1 1\,dx\,dy + \int _{y=b^{-1}-1} ^{1} \int_{x=0} ^1 1\,dx\,dy     \right] \,da\,db \\
&+\int_{b=1/t} ^{1} \int_{a=1-b} ^{1/bt} \left[\int _{y=0} ^{b^{-1}-1} \int_{x=l(y)} ^1 1\,dx\,dy + \int _{y=b^{-1}-1} ^{1} \int_{x=0} ^1 1\,dx\,dy     \right] \,da\,db 
\end{align*}
If $t\in[2,4)$, then the integral is 
$$\int_{b=1/2} ^{1} \int_{a=1-b} ^{1/bt} \left[\int _{y=0} ^{b^{-1}-1} \int_{x=l(y)} ^1 1\,dx\,dy + \int _{y=b^{-1}-1} ^{1} \int_{x=0} ^1 1\,dx\,dy     \right] \,da\,db $$
If $t\in[4,\infty)$, then the integral is 
$$\int_{b=\frac{1+\sqrt{1-(4/t)}}{2}} ^{1} \int_{a=1-b} ^{1/bt} \left[\int _{y=0} ^{b^{-1}-1} \int_{x=l(y)} ^1 1\,dx\,dy + \int _{y=b^{-1}-1} ^{1} \int_{x=0} ^1 1\,dx\,dy     \right] \,da\,db $$
By evaluating this integral and approximating the result with a Taylor series we see that
$$\int_{b=\frac{1+\sqrt{1-(4/t)}}{2}} ^{1} \int_{a=1-b} ^{1/bt} \left[\int _{y=0} ^{b^{-1}-1} \int_{x=l(y)} ^1 1\,dx\,dy + \int _{y=b^{-1}-1} ^{1} \int_{x=0} ^1 1\,dx\,dy     \right] \,da\,db \sim t^{-2}.$$
%with constant $2.$

Hence,
$$m_{\W}\left(\{\mathcal R>t\}\cap\{b+v_1>1\}\cap\{b>1/2\}\right) $$
decays quadratically.

Combining this with the part where $b<1/2$, we see that

$$m_{\W}\left(\{\mathcal R>t\}\cap\{b+v_1>1\}\right) \sim t^{-2}.$$
%with constant $2$.

\textbf{ (Case 2: $b+v_1 <1$)} 

We are interested in the situation underneath $l(y)$ in Figure \ref{fig:2Cases}. Now we impose the return time condition. On this region (i.e. when $b+v_1\le 1$), we have
$$R(\omega) = \frac{v_2}{v_1}>t\iff x \le \left(\frac{1}{a^2t}-\frac{b}{a}\right)y=:L(y).$$
Notice that $L$ is a line through the origin and since it is an upper bound for $x$, then we are only concerned when it has positive slope. That is, when
$$\frac{1}{a^2t}-\frac{b}{a}>0\iff \frac{1}{bt}>a.$$

Lastly, we notice to parameterize the $(y,x)$-plane, we need to understand how the two lines $l$ and $L$ intersect each other. Since they are each lines, it suffices to understand how the interact when $y=1.$ We obtain 3 cases:
\begin{enumerate}
\item $b>1/2$ and $L(1)\ge l(1)$
\item $b<1/2$ and $L(1)\ge l(1)$ 
\item $b<1/2$ and $L(1)< l(1)$
\end{enumerate}
We do not have the case $b>1/2$ and $L(1)< l(1)$ because when $b>1/2$ then we have $l(1)$ is negative which forces $L$ to be negative which is not possible since it is an upper bound for $x$. Lastly, we note that $L$ and $l$ intersect at the point $y=(1-b)at.$ We illustrate the three cases in Figure \ref{fig:3Cases}.

\begin{comment}
\begin{figure}[!htb]

\begin{minipage}{.5\linewidth}
\centering
\subfloat[$b>1/2$ and $L(1)\ge l(1)$]{\includegraphics[scale=.35]{Case1WSLFigure}}
\end{minipage}%
\begin{minipage}{.35\linewidth}
\centering
\subfloat[$b<1/2$ and $L(1)\ge l(1)$ ]{\includegraphics[scale=.35]{Case2WSLFigure}}
\end{minipage}\par\medskip
\centering
\subfloat[ $b<1/2$ and $L(1)< l(1)$]{\includegraphics[scale=.35]{Case3WSLFigure}}

\caption{A generic picture of the three cases that arise for the return map}
\label{fig:3Cases}
\end{figure}
\end{comment}

\begin{figure}[ht!]
    \centering
    \begin{subfigure}{.5\linewidth}
        \includegraphics[scale=0.35]{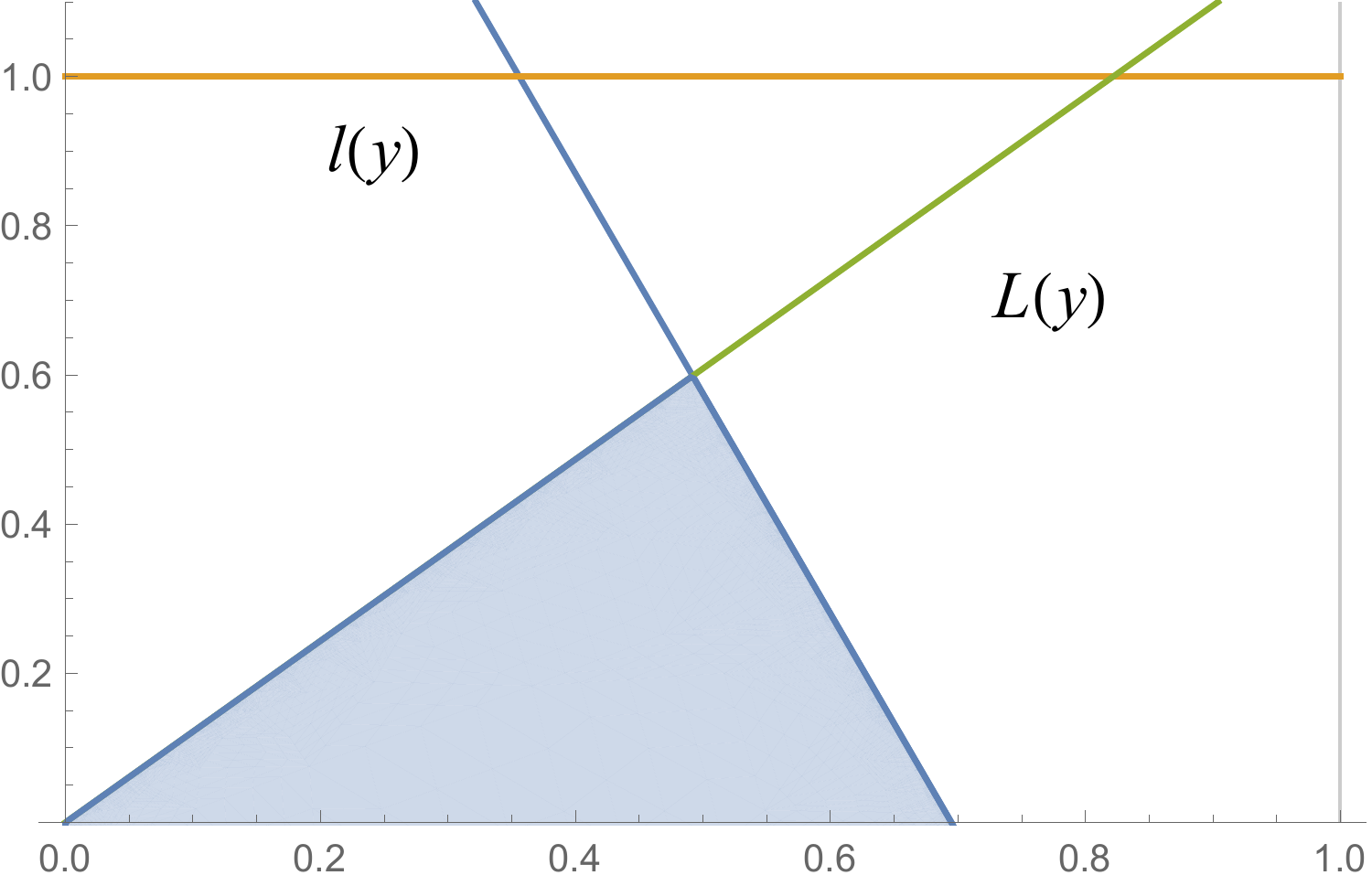}
        \caption{$b>1/2$ and $L(1)\ge l(1)$}
    \end{subfigure}
    \hskip2em
    \begin{subfigure}{.5\linewidth}
        \includegraphics[scale=0.35]{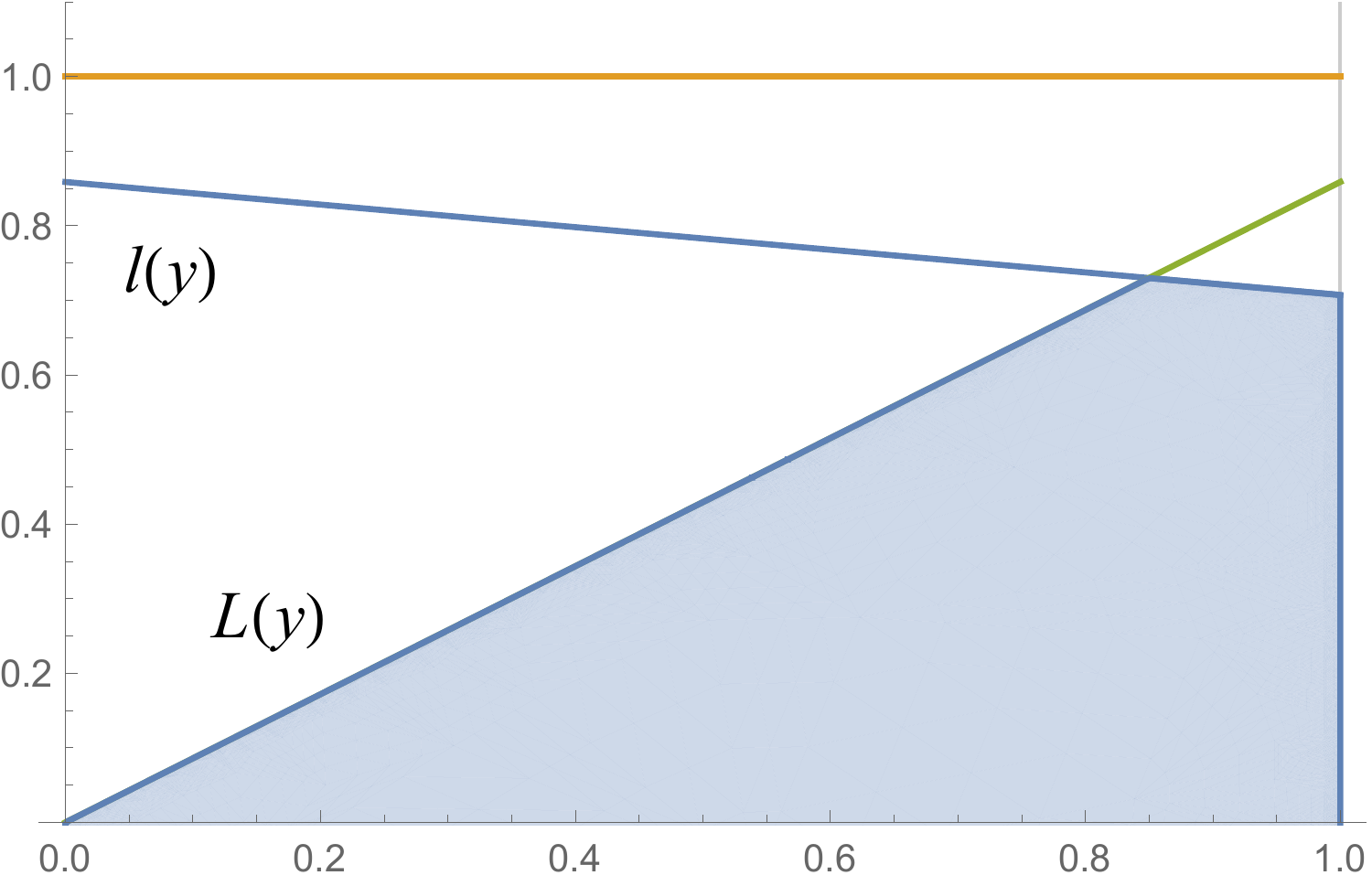}
        \caption{$b<1/2$ and $L(1)\ge l(1)$ }
    \end{subfigure}
    \begin{subfigure}{.5\linewidth}
        \includegraphics[scale=0.35]{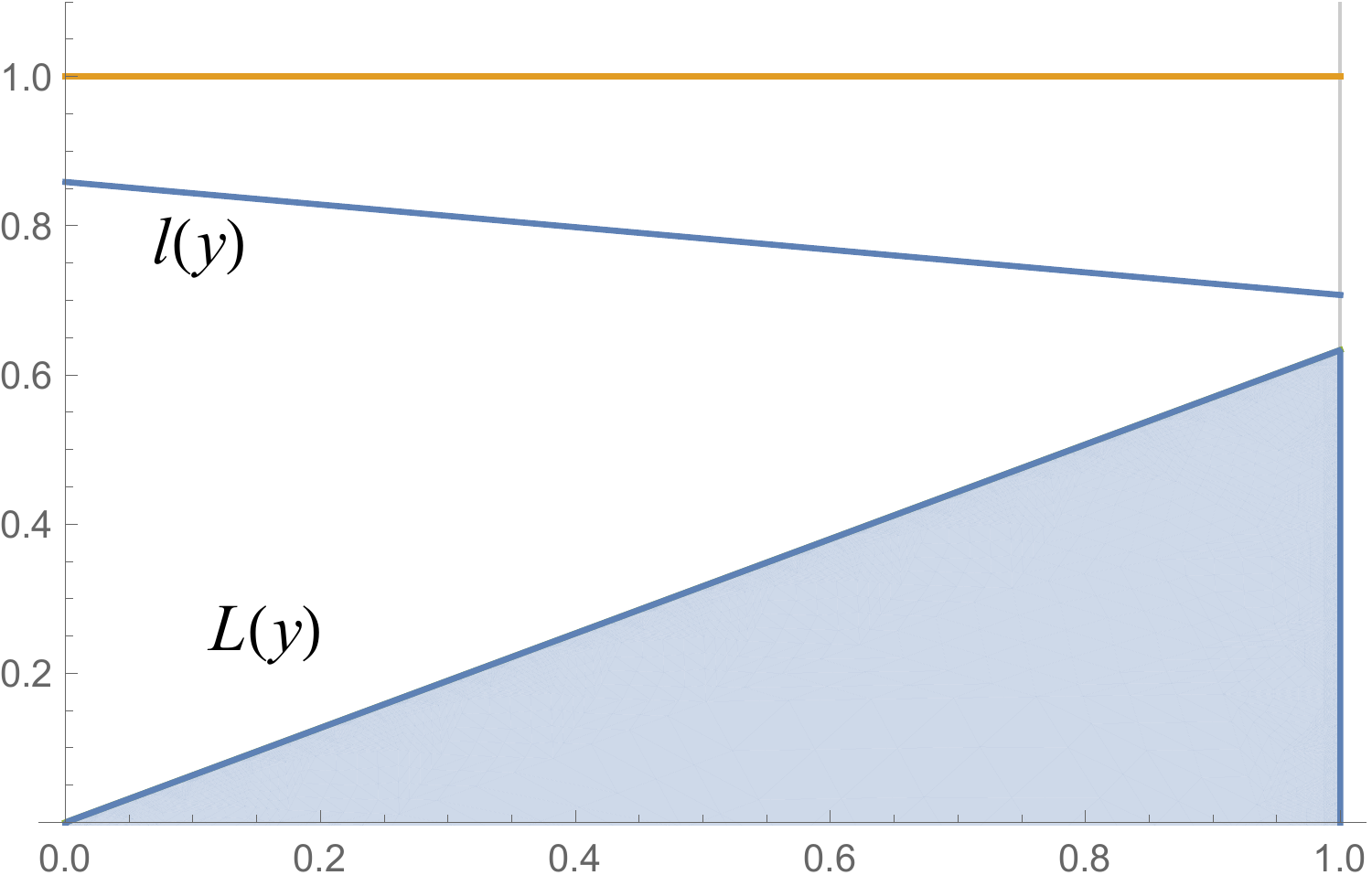}
        \caption{$b<1/2$ and $L(1)< l(1)$}
    \end{subfigure}
\caption{A generic picture of the three cases that arise for the return map}
\label{fig:3Cases}
\end{figure}

We now compute the contribution on each of the above cases.

\textbf{ (Case 2A: $b+v_1 <1$, $b>1/2$, and $L(1)\ge l(1)$)} 

If $b>1/2$ and $L(1)\ge l(1)$, then the region is the union of
$$\{(b,a,y,x): b\in(1/2,1], a\in(1-b,1],y\in(0,(1-b)at],x\in(0,L(y))\}$$
and 
$$\{(b,a,y,x): b\in(1/2,1], a\in(1-b,1],y\in((1-b)at,b^{-1}-1],x\in(0,l(y))\}.$$

The condition $L(1)\ge l(1)$ is equivalent to $a\le\frac{1}{t(1-b)}$. Combining this with the condition from the return map means that we are interested in 
$$a\le\min\left\{\frac{1}{bt},\frac{1}{t(1-b)}\right\}.$$
However,
$$\frac{1}{tb}<\frac{1}{t(1-b)}\iff1/2<b$$
so that actually we only care about how the hyperbola $a=\frac{1}{tb}$ intersects the region $\{(b,a):b\in(1/2,1],a\in(1-b,1]\}.$

If $t\in[0,1)$, then everything is under the hyperbola and the integral is
$$\int_{b=1/2} ^{1}\int_{a=1-b}^{1} \left[  \int_{y=0} ^{(1-b)at} \int_{x=0} ^{L(y)} 1\,dx\,dy +    \int_{y=(1-b)at} ^{b^{-1}-1} \int_{x=0} ^{l(y)} 1\,dx\,dy               \right]\,da\,db.$$

If $t\in[1,2)$, then the integral is
\begin{align*}
&\int_{b=1/2} ^{1/t}\int_{a=1-b}^{1} \left[  \int_{y=0} ^{(1-b)at} \int_{x=0} ^{L(y)} 1\,dx\,dy +    \int_{y=(1-b)at} ^{b^{-1}-1} \int_{x=0} ^{l(y)} 1\,dx\,dy               \right]\,da\,db\\
&+\int_{b=1/t} ^{1}\int_{a=1-b}^{1/bt} \left[  \int_{y=0} ^{(1-b)at} \int_{x=0} ^{L(y)} 1\,dx\,dy +    \int_{y=(1-b)at} ^{b^{-1}-1} \int_{x=0} ^{l(y)} 1\,dx\,dy               \right]\,da\,db
\end{align*}

If $t\in[2,4)$, then the integral is
\begin{align*}
\int_{b=1/2} ^{1}\int_{a=1-b}^{1/bt} \left[  \int_{y=0} ^{(1-b)at} \int_{x=0} ^{L(y)} 1\,dx\,dy +    \int_{y=(1-b)at} ^{b^{-1}-1} \int_{x=0} ^{l(y)} 1\,dx\,dy               \right]\,da\,db
\end{align*}

If $t\in[4,\infty)$, then the integral is
\begin{align*}
\int_{b=\frac{1+\sqrt{1-(4/t)}}{2}} ^{1}\int_{a=1-b}^{1/bt} \left[  \int_{y=0} ^{(1-b)at} \int_{x=0} ^{L(y)} 1\,dx\,dy +    \int_{y=(1-b)at} ^{b^{-1}-1} \int_{x=0} ^{l(y)} 1\,dx\,dy               \right]\,da\,db
\end{align*}

By evaluating the above integral and approximating it with its Taylor series we see that the decay is cubic. That is,
$$m_{\W}\left(\{\mathcal R>t\}\cap\{b+v_1<1\}\cap\{b>1/2\}  \cap\{L(1)\ge l(1)\}  \right) \sim t^{-3}.$$

\textbf{ (Case 2B: $b+v_1 <1$,  $b<1/2$, and $L(1)\ge l(1)$ )} 

 If $b<1/2$ and $L(1)\ge l(1)$, then the region is the union of
$$\{(b,a,y,x): b\in(0,1/2], a\in(1-b,1],y\in(0,(1-b)at],x\in(0,L(y))\}$$
and 
$$\{(b,a,y,x): b\in(0,1/2], a\in(1-b,1],y\in((1-b)at,1],x\in(0,l(y))\}.$$

As before, the condition $L(1)\ge l(1)$ is equivalent to $a\le\frac{1}{t(1-b)}$. Combining this with the condition from the return map means that we are interested in 
$$a\le\min\left\{\frac{1}{bt},\frac{1}{t(1-b)}\right\}.$$
We have,
$$\frac{1}{tb}>\frac{1}{t(1-b)}\iff1/2>b.$$
Consequently, we are interested in how the hyperbola $a=\frac{1}{t(1-b)}$ intersects the triangle $\{(b,a):b\in(0,1/2],a\in(1-b,1]\}$/

If $t\in[0,1)$, then everything is under the hyperbola and the integral is 
%%%%
$$\int_{b=0} ^{1/2}\int_{a=1-b}^{1} \left[  \int_{y=0} ^{(1-b)at} \int_{x=0} ^{L(y)} 1\,dx\,dy +    \int_{y=(1-b)at} ^{1} \int_{x=0} ^{l(y)} 1\,dx\,dy               \right]\,da\,db.$$

If $t\in[1,2)$, then the integral is
\begin{align*}
&\int_{b=1-\sqrt{t^{-1}}} ^{1-t^{-1}}\int_{a=1-b}^{\frac{1}{t(1-b)}} \left[  \int_{y=0} ^{(1-b)at} \int_{x=0} ^{L(y)} 1\,dx\,dy +    \int_{y=(1-b)at} ^{1} \int_{x=0} ^{l(y)} 1\,dx\,dy               \right]\,da\,db\\
&+\int_{b=1-t^{-1}} ^{1/2}\int_{a=1-b}^{1} \left[  \int_{y=0} ^{(1-b)at} \int_{x=0} ^{L(y)} 1\,dx\,dy +    \int_{y=(1-b)at} ^{1} \int_{x=0} ^{l(y)} 1\,dx\,dy               \right]\,da\,db
\end{align*}

If $t\in[2,4)$, then the integral is
\begin{align*}
\int_{b=1-\sqrt{t^{-1}}} ^{1/2}\int_{a=1-b}^{\frac{1}{t(1-b)}} \left[  \int_{y=0} ^{(1-b)at} \int_{x=0} ^{L(y)} 1\,dx\,dy +    \int_{y=(1-b)at} ^{1} \int_{x=0} ^{l(y)} 1\,dx\,dy               \right]\,da\,db
\end{align*}

If $t\in[4,\infty)$, then the hyperbola passes below the triangle and there is no contribution.

\textbf{ (Case 2C: $b+v_1 <1$,  $b<1/2$, and $L(1)< l(1)$ )} 

 If $b<1/2$ and $L(1)< l(1)$, then the region is 
$$\{(b,a,y,x): b\in(0,1/2], a\in(1-b,1],y\in(0,1],x\in(0,L(y))\}.$$

Using that $L(1)<l(1)$ is equivalent to $a>\frac{1}{t(1-b)}$. we have that we are interested in the condition:
$$\max\left\{1-b,\frac{1}{t(1-b)}\right\}<a<\min\left\{1,\frac{1}{bt}\right\}.$$
 By understanding this region in the $(b,a)$-slice we deduce the following.

If $t\in[0,1)$, the condition that $a>\frac{1}{t(1-b)}$ says there is no contribution here.

If $t\in[1,2)$, the integral is
$$\int_{b=0} ^{1-\sqrt{1/t}}\int_{a=1-b}^{1}   \int_{y=0} ^{1} \int_{x=0} ^{L(y)} 1\,dx\,dy \,da\,db +           \int_{b=1-\sqrt{1/t}} ^{1-(1/t)}\int_{a=\frac{1}{t(1-b)}}^{1}   \int_{y=0} ^{1} \int_{x=0} ^{L(y)} 1\,dx\,dy \,da\,db.$$

If $t\in\left[2,\frac{3+\sqrt{5}}{2}\right)$, the integral is
\begin{align*}
\int_{b=0} ^{1-\frac{1}{\sqrt{t}}}\int_{a=1-b}^{1}   \int_{y=0} ^{1} &\int_{x=0} ^{L(y)} 1\,dx\,dy \,da\,db +  \int_{b=1-\sqrt{1/t}} ^{1/t}\int_{a=\frac{1}{t(1-b)}}^{1/bt}   \int_{y=0} ^{1} \int_{x=0} ^{L(y)} 1\,dx\,dy \,da\,db \\
&+  \int_{b=1-\sqrt{1/t}} ^{1/2} \int_{a=\frac{1}{t(1-b)}}^{1/bt}   \int_{y=0} ^{1} \int_{x=0} ^{L(y)} 1\,dx\,dy \,da\,db
\end{align*}

If $t\in\left[\frac{3+\sqrt{5}}{2},4\right)$, the integral is
\begin{align*}
\int_{b=0} ^{1/t}\int_{a=1-b}^{1}   \int_{y=0} ^{1} &\int_{x=0} ^{L(y)} 1\,dx\,dy \,da\,db +  \int_{b=1/t} ^{1-\sqrt{1/t}}\int_{a=1-b}^{1/bt}   \int_{y=0} ^{1} \int_{x=0} ^{L(y)} 1\,dx\,dy \,da\,db \\
&+  \int_{b=1-\sqrt{1/t}} ^{1/2} \int_{a=\frac{1}{t(1-b)}}^{1/bt}   \int_{y=0} ^{1} \int_{x=0} ^{L(y)} 1\,dx\,dy \,da\,db
\end{align*}

Lastly, if $t\ge4$, the integral is 
$$\int_{b=0} ^{1/t}\int_{a=1-b}^{1}   \int_{y=0} ^{1} \int_{x=0} ^{L(y)} 1\,dx\,dy \,da\,db +           \int_{b=1/t} ^{\frac{1-\sqrt{1-(4/t)}}{2}}\int_{a=1-b}^{\frac{1}{bt}}   \int_{y=0} ^{1} \int_{x=0} ^{L(y)} 1\,dx\,dy \,da\,db.$$
By evaluating the integrals and approximating it with its Taylor series we see that the decay is cubic.

\subsubsection{Proof of Theorem \ref{Estimates} part 4 and 5}
Combining the results of the previous two subsections we see that we have quadratic decay for the slope gap distribution of doubled slit tori and support at zero.

An explicit description of the tail distribution $G$ is given below. Recall, that this is defined as  $G(t)=\frac{1}{6} \left(3+\pi ^2\right)-F(t)$ where $F(t)$ is the cummulative distribution function. 

Piecewise, we have:
$G(t) \cdot \chi_{[0,1]}=\frac{1}{6} \left(3+\pi ^2\right)-\frac{7 t}{8},$
\begin{align*}
G(t) \cdot \chi_{(1,2]}&=\frac{1}{24} \left(24 \left(\text{Li}_2\left(\frac{1}{t}\right)-\text{Li}_2\left(\frac{t-1}{t}\right)\right)-12 \log ^2(t)+24 \log (t-1) \log (t)+12 \log (t)\right)\\
&-\frac{5}{2}+\frac{1}{6 t^2}+\frac{1}{24} t (-24 \log (t-1)+24 \log (t)-4)+\frac{24 \log (t-1)+54 \log (t)+51}{24 t},
\end{align*}
\begin{align*}
G(t) \cdot \chi_{(2,\frac{3+\sqrt{5}}{2}]}&=\frac{1}{48} \left(-48 \left(\text{Li}_2\left(\frac{1}{t}\right)-\text{Li}_2\left(\frac{t-1}{t}\right)\right)-3 \log \left(t^3\right)-24 \log ^2(t)\right)+\frac{1}{2 t^{3/2}}\\
&+\frac{12 \log \left(1-\frac{1}{\sqrt{t}}\right)-36 \log (t-1)+48 \log (t)-12 \log \left(t-\sqrt{t}\right)}{48 t^2}\\
&+\frac{72 \coth ^{-1}(1-2 t)-18}{48 t^2}+\frac{24 \log \left(1-\frac{1}{\sqrt{t}}\right)+24 \log (t)-36}{48 \sqrt{t}}\\
&+\frac{1}{48} \left(-12 \log \left(1-\frac{1}{\sqrt{t}}\right)-3 (7+\log (256)) \log \left(\frac{4}{t}\right)\right)\\
&+\frac{1}{24} (-63-45 \log (2)-\log (8) \log (256))\\
&+\frac{1}{48} t (-48 \log (t-1)+48 \log (t)-16)\\
&+\frac{-12 \log \left(1-\frac{1}{\sqrt{t}}\right)+48 \log (t-1)+96 \log (t)+192}{48 t}\\
&+\frac{1}{48} (48
   \log (t-1) \log (t)+24 (3+\log (2)) \log (t)),
\end{align*}

\begin{align*}
G(t) \cdot \chi_{(\frac{3+\sqrt{5}}{2},4]}&=\frac{1}{48} \left(-48 \left(\text{Li}_2\left(\frac{1}{t}\right)-\text{Li}_2\left(\frac{t-1}{t}\right)\right)-3 \log \left(t^3\right)-24 \log ^2(t)\right)+\\
&\frac{-12 \log (1-t)-24 \log (t-1)+12 \log (-t)+24 \log (t)}{{48 t^2}}\\
&+\frac{72 \coth ^{-1}(1-2 t)-24}{48 t^2}\\
&+\frac{1}{48} \left(-12 \log \left(1-\frac{1}{\sqrt{t}}\right)-144-2 \log (8) (15+\log
   (256))\right)\\
&+\frac{1}{48} \left(12 \log \left(\sqrt{t}-1\right)+3 (7+\log (256)) \log \left(\frac{4}{t}\right)\right)\\
&+\frac{24 \log \left(1-\frac{1}{\sqrt{t}}\right)-24 \log \left(\sqrt{t}-1\right)+12 \log (t)}{48 \sqrt{t}}\\
&+\frac{1}{48} t (-48 \log (t-1)+48 \log (t)-16)+\\
&\frac{-12 \log \left(1-\frac{1}{\sqrt{t}}\right)+12 \log
   \left(\sqrt{t}-1\right)+24 \log (t-1)+24 \log \left(\frac{t-1}{\sqrt{t}}\right)}{48 t}\\
&+\frac{114 \log (t)+198}{48 t}\\
&+\frac{1}{48} (48 \log (t-1) \log (t)+6 (13+\log (16)) \log (t)).
\end{align*}
 
The formula for $G(t)$ when $t>4$ is extremely long so we do not write them down. The main characteristic is that it has quadratic tail decay as shown in Theorem \ref{Estimates} part 4.

%%%%%%%%%%%%%%%%%%%%%%%%%%%%%%%%%%%%%%%%%%%%%%%%%%%%%%%%%%%%%%
\subsection{Haar measure computations and estimates for $\Omega$}

We prove the estimate of Theorem \ref{Estimates} part 1. We also show that there is support at zero for the density function.

\begin{proof} (Of Theorem \ref{Estimates} part 1)

We compute the cumulative distribution function for the gaps of slopes $$m_{\Omega}(\{R>t\}\cap \Omega_i)$$ separately for each $i=1,\ldots,4$ where $\Omega_i$ are the regions where the return map is the same.

We notice that much of these computations were needed in computing the gap distribution on $\W$, specifically the ones for the tail of $\Omega_1$ and $\Omega_3$. As such, we do not reproduce them here.

 For $\Omega_1$ and $\Omega_3$ we can compute exact asymptotics, while we only compute upper and lower bounds for $\Omega_2$ and $\Omega_4$. This is because the return time function includes a floor function in these cases which make explicit computations difficult.

\textbf{(Tail on $\Omega_1$)} 
This computation was implicit in the section of measure estimates on $\W$. It was shown there that for $t>4$, we have

$$m_{\Omega}(\{R>t\}\cap \Omega_1) \sim t^{-2}.$$
%with constant $1/6.$

 Hence, the decay on $\Omega_1$ is quadratic.

%The first integral above is
%$$\frac{1 - 3 t (11 + 2 t) - 6 t (-2 + t + t^2) \log(1-t^{-1}) +  24 t^2 \D(t^{-1})}{12 t^2}$$
%%%%%%%%%%%%%%%%%%%%%

\textbf{(Tail in $\Omega_2$)} 

Due to the floor function in the formula for the return map on $\Omega$ we only compute lower and upper estimates on the decay.
 We first notice that $\Omega_2$ can be written as
$$\Omega_2=\left\{b\in(0,1],a\in(1-b,1],\a\in(a,1),s\in\left[(ab)^{-1}\left(\frac{\a-a}{\a}\right),(ab)^{-1}\right)\right\}.$$
%$$\Omega_2=\left\{b\in(0,1],\a\in(1-b,1],a\in(1-b,\a),s\in[(ab)^{-1}\left(\frac{\a-a}{\a}\right),(ab)^{-1})\right\}.$$

We compute an upper bound for the contribution on $\Omega_2$. Recall that $j = \lfloor\frac{1+a-\a}{b}\rfloor$ is in the formula for the return map on $\Omega_2$. An upper bound is found by
\begin{align*}
R|_{\Omega_2}&=\frac{j(a^{-1}-sb)+sa}{\a-a+jb}<\frac{(\frac{1+a-\a}{b})(a^{-1}-sb)+sa}{\a-a+jb}\\
&=\frac{(1+a-\a)(ab)^{-1}-s(1-\a)}{\a-a+jb}<\frac{2(ab)^{-1}}{\a-a+jb}.
\end{align*}
Using that 
$$\a-a+jb>\a-a+\left(\frac{1+a-\a}{b}-1\right)b=1-b$$
and continuing our estimate from before, we find
$$R|_{\Omega_2}<\frac{2(ab)^{-1}}{\a-a+jb}<\frac{2}{ab(1-b)}=:U.$$

To compute the contribution of $m_{\Omega}(\{t<U\})$, we notice that
$$t<U\iff a<\frac{2}{tb(1-b)}.$$
Thus, we need to understand how the curve $\frac{1}{tb(1-b)}$ intersects the triangle  $\{(b,a): b\in(0,1], a\in(1-b,1]\} )$.
Due to this, the following points will be of interest
$$b_k := \frac{2}{3}\left(\cos\left(\frac{1}{3}\arccos\left(\frac{54}{2t}-1\right)-\frac{2\pi *k}{3}\right)+1\right)$$
where $k=1,2.$ For $t$ large, we note that $b_1$ is close to 1 and $b_2$ is close to 0. These points correspond to where the curve $a=\frac{1}{tb(1-b)}$ intersect the line $a=1-b$.

Hence, 
\begin{align*}
m_{\Omega}(\{t<U\}) &= \int_{b=0} ^{\frac{1-\sqrt{1-(8/t)}}{2}} \int_{a=1-b} ^{1} \int_{\a=a} ^{1} \int _{s=\frac{\a-a}{ab\a}} ^{\frac{1}{ab}}1\,ds\,d\a\,da \,db\\
&+ \int _{b={\frac{1-\sqrt{1-(8/t)}}{2}}} ^{b_2} \int_{a=1-b} ^{\frac{1}{bt(1-b)}} \int_{\a=a} ^ 1  \int _{s=\frac{\a-a}{ab\a}} ^{\frac{1}{ab}}1\,ds\,d\a \,da \,db\\
&+ \int _{b=x_1} ^{\frac{1+\sqrt{1-(8/t)}}{2}} \int_{a=1-b} ^{\frac{1}{bt(1-b)}} \int_{\a=a} ^ 1  \int _{s=\frac{\a-a}{ab\a}} ^{\frac{1}{ab}}1\,ds\,d\a\,da \,db\\
&+ \int _{b=\frac{1+\sqrt{1-(8/t)}}{2}} ^1 \int_{a=1-b} ^{1} \int_{\a=a} ^ 1 \int _{s=\frac{\a-a}{ab\a}} ^{\frac{1}{ab}}1\,ds\,d\a\,da \,db
\end{align*}
 
By evaluating each integral and approximating the result with a Taylor series, we see that
\begin{align*}
& \int_{b=0} ^{\frac{1-\sqrt{1-(8/t)}}{2}} \int_{a=1-b} ^{1} \int_{\a=a} ^{1}  \int _{s=\frac{\a-a}{ab\a}} ^{\frac{1}{ab}}1\,ds\,d\a\,da \,db\sim t^{-2},\\
& \int _{b={\frac{1-\sqrt{1-(8/t)}}{2}}} ^{b_2} \int_{a=1-b} ^{\frac{1}{bt(1-b)}} \int_{\a=a} ^ 1 \int _{s=\frac{\a-a}{ab\a}} ^{\frac{1}{ab}}1\,ds\,d\a\,da \,db\sim t^{-1},\\
& \int _{b=b_1} ^{\frac{1+\sqrt{1-(8/t)}}{2}} \int_{a=1-b} ^{\frac{1}{bt(1-b)}} \int_{\a=a} ^ 1 \int _{s=\frac{\a-a}{ab\a}} ^{\frac{1}{ab}}1\,ds\,d\a\,da \,db\sim t^{-1},\\
& \int _{b=\frac{1+\sqrt{1-(8/t)}}{2}} ^1 \int_{a=1-b} ^{1} \int_{\a=a} ^ 1 \int _{s=\frac{\a-a}{ab\a}} ^{\frac{1}{ab}}1\,ds\,d\a\,da \,db\sim t^{-1}.
\end{align*}

Hence, an upper bound for the decay is
$$m_{\Omega}(\{t<R\}\cap\Omega_2) <m_{\Omega}(\{t<U\}) \sim t^{-1}.$$

We now compute a lower bound for the contribution on $\Omega_2$. Using that the denominator of $R$ is bounded above by 1 and that all the terms in the numerator are positive we have that 
$$R|_{\Omega_2}=\frac{j(a^{-1}-sb)+sa}{\a-a+jb}>\frac{sa}{1}.$$
Finally, using the lower bound of $s$ in $\Omega_2$ yields that
$$R|_{\Omega_2}>\frac{\a-a}{b\a}=:L.$$
To compute the contribution of $m_{\Omega}(\{t<L\})$, we notice that
$$t<L\iff a<\a(1-bt)$$
which indicates interest in how the curve $a=\a(1-bt)$ intersects the  $(\a,a)$ triangle. A computation confirms that we only have points in this triangle whenever $b<1/t$. of those points, we only get contribution whenever $1-tb>1$ which gives the integral
$$\int_{b=0} ^{1/t} \int_{\a=1-b} ^1 \int_{a=1-b} ^\a \int_{s=\frac{\a-a}{ab\a}} ^{(ab)^{-1}} 1\,ds\,da\,d\a \,db.$$
Evaluating and using Taylor series shows that this integral has quadratic decay i.e.
$$\int_{b=0} ^{1/t} \int_{\a=1-b} ^1 \int_{a=1-b} ^\a \int_{s=\frac{\a-a}{ab\a}} ^{(ab)^{-1}} 1\,ds\,da\,d\a \,db\sim t^{-2}.$$
%%%%%%%%%%%%%%%%%%
\textbf{ (Tail on $\Omega_3$)}

This computation was implicit in the section of measure estimates on $\W$. It was shown there that for $t>4$, we have

$$m_{\Omega}(\{R>t\}\cap \Omega_3) \sim t^{-2}.$$
%with constant $1/2$.

 Hence, the decay on $\Omega_3$ is quadratic.

\textbf{(Tail on $\Omega_4$) } 
 We first notice that $\Omega_4$ can be written as
$$\Omega_4 = \left\{b\in(0,1],a\in(1-b,1],\a\in(1-b,a], s\in[0,(ab)^{-1})\right\}.$$ 
%$$\Omega_4 = \left\{b\in(0,1],\a\in(1-b,1],a\in(\a,1], s\in[0,(ab)^{-1})\right\}.$$ 

Second, we notice that the upper estimate from $\Omega_2$ works. Namely,
$$R|_{\Omega_4}< \frac{2}{ab(1-b)}=:U.$$
Moreover, 
$$t<U\iff a<\frac{2}{tb(1-b)}.$$
Thus, we need to understand how the curve $\frac{1}{tb(1-b)}$ intersects the triangle  $$\{(b,a): b\in(0,1], a\in(1-b,1]\} ).$$
As before, we will be interested in
$$b_k := \frac{2}{3}\left(\cos\left(\frac{1}{3}\arccos\left(\frac{54}{2t}-1\right)-\frac{2\pi *k}{3}\right)+1\right)$$
where $k=1,2.$ For $t$ large, we note that $b_1$ is close to 1 and $b_2$ is close to 0. Lastly, these points correspond to where the curve $a=\frac{1}{tb(1-b)}$ intersect the line $a=1-b$.

An upper bound for the contribution is given by, 
\begin{align*}
m_{\Omega}(\{t<U\}) &= \int_{b=0} ^{\frac{1-\sqrt{1-(8/t)}}{2}} \int_{a=1-b} ^{1} \int_{\a=1-b} ^{a} \int _{s=0} ^{\frac{1}{ab}}1\,ds\,d\a\,da \,db\\
&+ \int _{b={\frac{1-\sqrt{1-(8/t)}}{2}}} ^{b_2} \int_{a=1-b} ^{\frac{1}{bt(1-b)}}  \int_{\a=1-b} ^{a} \int _{s=0} ^{\frac{1}{ab}}1\,ds\,d\a\,da \,db\\
&+ \int _{b=b_1} ^{\frac{1+\sqrt{1-(8/t)}}{2}} \int_{a=1-b} ^{\frac{1}{bt(1-b)}} \int_{\a=1-b} ^{a} \int _{s=0} ^{\frac{1}{ab}}1\,ds\,d\a\,da \,db\\
&+ \int _{b=\frac{1+\sqrt{1-(8/t)}}{2}} ^1 \int_{a=1-b} ^{1}  \int_{\a=1-b} ^{a} \int _{s=0} ^{\frac{1}{ab}}1\,ds\,d\a\,da \,db
\end{align*}

We find the decay of each term to be
\begin{align*}
& \int_{b=0} ^{\frac{1-\sqrt{1-(8/t)}}{2}} \int_{a=1-b} ^{1}  \int_{\a=1-b} ^{a} \int _{s=0} ^{\frac{1}{ab}}1\,ds\,d\a\,da \,db\sim t^{-2},\\
& \int _{b={\frac{1-\sqrt{1-(8/t)}}{2}}} ^{b_2} \int_{a=1-b} ^{\frac{1}{bt(1-b)}} \int_{\a=1-b} ^{a} \int _{s=0} ^{\frac{1}{ab}}1\,ds\,d\a\,da \,db\sim t^{-1},\\
& \int _{b=b_1} ^{\frac{1+\sqrt{1-(8/t)}}{2}} \int_{a=1-b} ^{\frac{1}{bt(1-b)}}  \int_{\a=1-b} ^{a} \int _{s=0} ^{\frac{1}{ab}}1\,ds\,d\a\,da \,db\sim t^{-1},\\
& \int _{b=\frac{1+\sqrt{1-(8/t)}}{2}} ^1 \int_{a=1-b} ^{1}  \int_{\a=1-b} ^{a} \int _{s=0} ^{\frac{1}{ab}}1\,ds\,d\a\,da \,db\sim t^{-1}.
\end{align*}

We now compute a lower bound for the contribution on $\Omega_4$. Using that the denominator of $R$ is bounded above by 1 and that all the terms in the numerator are positive we have that 
$$R|_{\Omega_4}=\frac{j(a^{-1}-sb)+sa}{\a-a+jb}>\frac{sa}{1}=: L.$$
To compute the contribution of $m_{\Omega}(\{L>t\})$, we notice that
$$L>t\iff s>t/a.$$
This is a valid lower bound for $s$ since $t/a$ is always positive. The upper bounds on $s$ yields the condition
$$t/a< (ab)^{-1}\iff b<t^{-1}$$
yielding the integral
$$m_{\Omega}(\{R>t\}\cap \Omega_4)>m_{\Omega}(\{L>t\}) = \int_{b=0} ^{1/t}\int_{\a=1-b} ^1 \int_{a=\a} ^1 \int_{s=t/a} ^{\frac{1}{ab}}1\,ds\,da\,d\a \,db.$$
Evaluating this integral and approximating the result by its Taylor series, we have that 
$$ \int_{b=0} ^{1/t}\int_{\a=1-b} ^1 \int_{a=\a} ^1 \int_{s=t/a} ^{\frac{1}{ab}}1\,ds\,da\,d\a \,db\sim t^{-2}$$
so that a lower bound on the contribution of $\Omega_4$ is given by a quadratic function.
\end{proof}

As a corollary of the measure estimates from above, we obtain that the distribution function for gaps of an affine lattice are supported near 0.
\begin{proof} (Of Theorem \ref{Estimates} part 2)
Note that the volume of $\Omega_3$ is given by $m_{\Omega}(\Omega_3) = \frac{\pi^2}{6}-1$. Notice that in the measure estimate of the tail of $\Omega_3$ in case I yields that
$$m_\Omega\left( \left\{ R>t \right\}\cap\Omega_3\right)\ge\int_{b=0} ^{1} \int_{\a=0} ^{1-b}\int_{a=1-b} ^1 \int _{s=0} ^{\frac{a^{-1}-t(b+\a)}{b}}1dsdad\a db= \frac{\pi^2}{6} -1 - \frac{ t}{3}$$
whenever $t<1.$ We only have an inequality because there may be contribution from the other pieces $\Omega_i$.  Hence,
$$\int_0 ^\varepsilon g(t)dt > m_\Omega\left( \left\{ R<\varepsilon \right\}\cap\Omega_3\right) = \frac{\pi^2}{6}-1 - m_\Omega\left( \left\{ R>\varepsilon \right\}\cap\Omega_3\right)  \ge \varepsilon/3  .$$
This contribution is positive for any $\varepsilon >0$ and so we conclude that there is support at zero.
\end{proof}

\subsection{Torsion measure computations and estimates for $\Omega$}

We compute the tail of the gap distribution with respect to the torsion measures $m_{q}$:
$$m_{q} (R>t).$$
The support of $m_q$  is on the affine lattices of the form
$$\{\Lambda_{(a,b,0,a/q)}:(a,b)\in \Delta\}.$$
Our main theorem is that the decay is at most quadratic.

\begin{proof}(Of Theorem \ref{Estimates} part 3)

The return map falls under 2 situations

$$R(a,b,0,a/q) = \begin{cases}
 \frac{a^{-1}}{b+ (a/q)}, &\text{ if } b+ \frac{a}{q}<1\\
 \frac{ja^{-1}}{(a/q)-a+jb}, &\text{ if } b+ \frac{a}{q}\ge1 \\
\end{cases}.$$

The conditions on the sum reduce the triangle $\Delta$ into 2 regions above and below the line $b=1- \frac{a}{q}$.
Denote by
$$C_1=\{(a,b)\in \Delta|b+ \frac{a}{q}<1 \}$$
and
$$C_2=\{(a,b)\Delta|b+ \frac{a}{q}\ge1 \}$$

We begin with case I:  $b+ \frac{a}{q}<1.$ Notice, if $q=1$, then there are no $b$'s possible and so the contribution here is null. So we assume $q>1$.

Hence,
$$R>t\iff b<\frac{1}{at}-\frac{a}{q}.$$
This hyperbola intersects the bottom boundary line $a=1-b$ when
$$a = \frac{1\pm\sqrt{1-\frac{4}{t} (1-\frac{1}{q})}}{2(1-\frac{1}{q})}.$$
A quick calculation shows that the root with a ``plus" is larger than 1 (and thus falls outside of the range of $a$'s we are interested in) whenever $q<t$ and in order for the roots to be real, we need a positive discriminant. Since we are only interested in the decay of the tail, we can choose $t$ large enough, to guarantee both conditions. Hence, the hyperbola $\frac{1}{at}-\frac{a}{q}$ intersects the bottom boundary line once at $a = \frac{1-\sqrt{1-\frac{4}{t} (1-\frac{1}{q})}}{2(1-\frac{1}{q})}.$

Furthermore, this hyperbola intersects the top  boundary line $b=1-\frac{a}{q}$ when $a = 1/t$.
Hence, the contribution to the tail of the gap distribution is given by

$$m_q\left(\{R>t\}\cap C_1\right)  = \int_{a=0} ^{1/t} \int_{b=1-a} ^{1-\frac{a}{q}}1\,da\,db + \int_{a=1/t} ^{\frac{1-\sqrt{1-\frac{4}{t} (1-\frac{1}{q})}}{2(1-\frac{1}{q})}} \int_{b=1-a} ^{\frac{1}{at}-\frac{a}{q}}1\,da\,db.$$
\begin{comment}
The first integral evaluates to $ \frac{1}{2t^2}\left(1-\frac{1}{q}\right)  $ and the second to
$$ 
  \frac{-2 - q^2 (2 - 2 t + t^2) + 
 q (4 - 2 t + \sqrt{q (4 + q (t - 4)) t^3}) + 
 4 ( q - 1) q t \log\left(\frac{2 q}{
   q + \sqrt{\frac{q (4 + q ( t -4))}{t}}}\right)}{4 ( q-1) q t^2} 
$$
\end{comment}

By evaluating these integrals and using an argument with Taylor series, one can show that the first term decays quadratically and second term is of order $t^{-3}$ so that the total decay in this case is quadratic. That is,

$$m_q\left(\{R>t\}\cap C_1\right) \sim t^{-2}.$$

Now we consider the second case:  $b+ \frac{a}{q}\ge1.$ Recall we denoted these points by $C_2$. In this case, our return map is 
$$R= \frac{ja^{-1}}{(a/q)-a+jb}$$
where $j=\left\lfloor \frac{1+a(1-\frac{1}{q})}{b}\right\rfloor$

If $q=1$, then our return map reduces to $R=\frac{1}{ab}$ and this intersects the triangle $\{(a,b):a\in(0,1],b\in(1-a,1]\}$ in the same region as in \cite{MR3214280}. Hence, the decay is quadratic.

If $q>1$, then the condition tail condition $R>t$ is equivalent to 
$$b<\frac{1}{at}+\frac{a}{j}\left(1-\frac{1}{q}\right).$$
This hyperbola is always bounded above by the hyperbola  $h(a):=\frac{1}{at}+a\left(1-\frac{1}{q}\right)$ since our floor function is bounded below by 1. Since ultimately we are only interested in the decay of
 $$m_q\left(\{R>t\}\cap C_2\right)  $$
 we will compute the order of decay of $m_q\left(\{b<h(a)\}\cap C_2\right)  .$ We will show this is quadratic decay in which case the quadratic decay of the tail follows since
$$m_q\left(\{R>t\}\cap C_2\right) \le m_q\left(\{b<h(a)\}\cap C_2\right) \sim t^{-2}.$$
Our interest in $h$ comes from the fact that their is no floor function to deal with. Then, it is easy to show that the hyperbola $h$ intersects $b=1$ when
$$a = \frac{1\pm\sqrt{1-\frac{4}{t} (1-\frac{1}{q})}}{2(1-\frac{1}{q})}.$$
As before, the root with a ``plus" falls outside of our region of interest for $t$ sufficiently large.  The hyperbola also intersects the boundary curve $b=1-\frac{a}{q}$ when 
$a = \frac{1\pm\sqrt{1-\frac{4}{t}}}{2}.$ Hence,

\begin{align*}
m_q\left(\{b<h(a)\}\cap C_2\right) &= \int_{a=0} ^{\frac{1-\sqrt{1-\frac{4}{t} (1-\frac{1}{q})}}{2(1-\frac{1}{q})}} \int_{b=1-(a/q)} ^1 1 \,db\,da \\
& +  \int_{a=\frac{1-\sqrt{1-\frac{4}{t} (1-\frac{1}{q})}}{2(1-\frac{1}{q})}} ^{\frac{1-\sqrt{1-\frac{4}{t}}}{2}} \int_{b=1-(a/q)} ^{h} 1 \,db\,da\\
&+ \int_{a=\frac{1+\sqrt{1-\frac{4}{t}}}{2}} ^1 \int_{b=1-(a/q)} ^{h} 1 \,db\,da .
\end{align*}

See Figure \ref{fig: ContributionTorsion} for an illustration of the contribution on $C_2.$
\begin{figure}[H]
	\centering
    \includegraphics[width=4in]{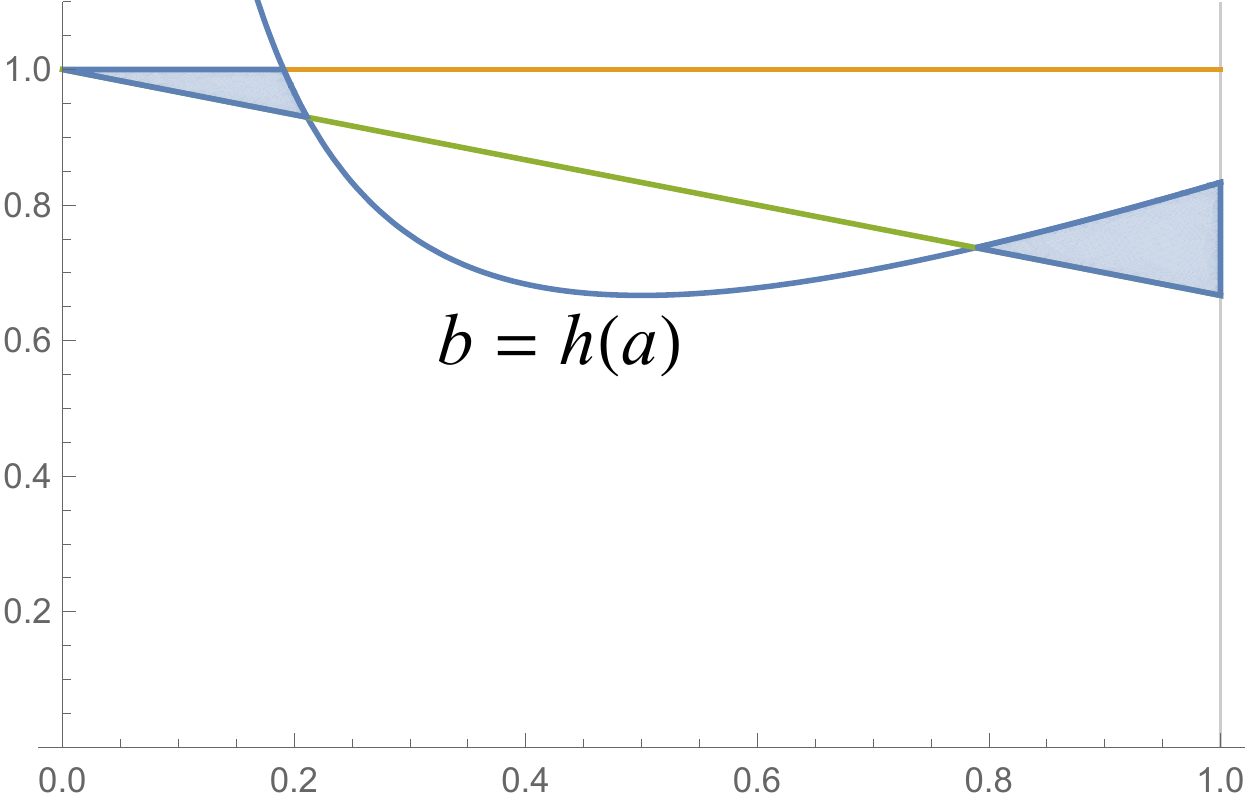}
    \caption{The contribution on $C_2$.} 
    \label{fig: ContributionTorsion}
\end{figure}

We leave it as an exercise to show that the first and last term decay quadratically while the second term decays like $t^{-3}$.

Now we find a lower bound for $C_2$ with respect to the torsion measure $m_q$.  Recall the following calculation from the upper bound on $C_2$,
$$R>t\iff b<\frac{1}{at}+\frac{a}{j}(1-q^{-1}).$$
Then, the set of $(a,b)$ in $C_2$ satisfying this last inequality contains those that have $b<\frac{1}{at}.$ Hence, we have the lower bound
$$m_q\left(\{R>t\}\right)\ge m_q\left(\{b<(at)^{-1}\}\right).$$
This measure is given by the integrals
$$\int_{a=0} ^{1/t}\int_{b=1-(a/q)}^1 1 \,db\,da + \int_{a=1/t} ^{\frac{q-\sqrt{q^2-(4q/t)}}{2}}\int_{b=1-(a/q)}^{1/at}1 \,db\,da.$$
A Taylor series approximation would show that the first integral decays quadratically while the second integral decays like $t^{-3}$.
Hence, the gaps with respect to the torsion measure $m_q$ are bounded above and below by a function with quadratic tail. That is,
$$m_q(C_2\cap\{R>t\})\asymp t^{-2}.$$
%with constants $(2q)^{-1}$ and $(1/2) + (2q)^{-1}.$

Combining this with the results on $C_1$ we have
$$m_q(\{R>t\})\asymp t^{-2}.$$
\end{proof}

%%%%%%%%%%%%%%%%%%%%%%%%%%%%%%%%%%%
%%%%%%%%%%%%%%%%%%%%%%%%%%%%%%%%%%%%%%%%%%
%%%%%%%%%%%%%%%%%%%%%%%%%%%%%%%%%%%%%%%%%%%%%%%%%%%%%%%

%%%%%%%%%%%%%%%%%%%%\in%%%%%%%%%%%%%%%%%%%%%%%%%%%%%%%%%%%%%%

%\bibliography{references}
%\bibliographystyle{plain}

 \end{document}